\documentclass[amssymb,amsfonts,refcheck,12pt,verbatim,righttag]{amsart}
\usepackage{amssymb,mathrsfs}

\usepackage{amssymb}
\usepackage{graphicx}
\usepackage{graphics,color}

\usepackage{amsfonts}
\usepackage{color}
\usepackage{amssymb}
\usepackage{cancel}
\usepackage{soul}

 \numberwithin{equation}{section} % \renewcommand{\rm}{\normalshape} %
\theoremstyle{plain}

\newtheorem{thm}{Theorem}[section]
\newtheorem{lem}[thm]{Lemma}
\newtheorem{pro}[thm]{Proposition}
\newtheorem{cor}[thm]{Corollary}
\newtheorem{ex}[thm]{Example}
\newtheorem{de}[thm]{Definition}
\newtheorem{rem}[thm]{Remark}

\def\R {{\Bbb R}}
\def\N {{\Bbb N}}
\def\Z {{\Bbb Z}}
\def\M {{\mathcal M}}

\def\D {{\mathcal D}}

\def\ba{{\bf a}}

\def\va{{\bf a}}

\newcommand{\Leb}{\mathcal{L}}
\newcommand{\bi}{\mathbf{i}}
\newcommand{\bj}{\mathbf{j}}

\newcommand{\ldim}[1]{\underline{\dim}_{\mathrm{#1}}}
\newcommand{\udim}[1]{\overline{\dim}_{\mathrm{#1}}}
\newcommand{\Dim}[1]{{\dim}_{\mathrm{#1}}}
\newcommand{\diam}{\mathrm{diam}}

\DeclareMathOperator*{\esssup}{ess\,sup}
\DeclareMathOperator*{\essinf}{ess\,inf}

\usepackage[colorlinks,citecolor=blue,pagebackref,urlcolor=black,hypertexnames=false]{hyperref}

\begin{document}
\baselineskip 13.7pt
\title{Dimensions of projected sets and measures on typical self-affine sets}

\author{De-Jun FENG}
\address{
Department of Mathematics\\
The Chinese University of Hong Kong\\
Shatin,  Hong Kong\\
}
\email{djfeng@math.cuhk.edu.hk}

\author{Chiu-Hong Lo}
\address{
Department of Mathematics\\
The Chinese University of Hong Kong\\
Shatin,  Hong Kong\\
}
\email{chlo@math.cuhk.edu.hk}

\author{Cai-Yun Ma}
\address{
Department of Mathematics\\
The Chinese University of Hong Kong\\
Shatin,  Hong Kong\\
}
\email{cyma@math.cuhk.edu.hk}
\keywords{}

\keywords{Affine iterated function systems, coding maps, self-affine sets, projections of Borel sets and measures, local dimensions, exact dimensionality, fractal dimensions}
\thanks {
2020 {\it Mathematics Subject Classification}: 28A80, 37C45, 31A15, 49Q15, 60B05}

%\thanks { The research of Feng  was partially supported by the HKRGC GRF grant}
%2000 {\it Mathematics Subject Classification}: 28A80, 37C45}

\date{}

\begin{abstract}
Let $T_1,\ldots, T_m$ be a family of $d\times d$   invertible real matrices with $\|T_i\|<1/2$ for $1\leq i\leq m$.
 For $\ba=(a_1,\ldots, a_m)\in \R^{md}$, let  $\pi^{\ba}:\; \Sigma=\{1,\ldots, m\}^\N\to \R^d$ denote the coding map associated with the
  affine IFS $\{T_ix+a_i\}_{i=1}^m$.  We show that for every Borel probability measure $\mu$ on $\Sigma$, each of the following dimensions (lower and upper Hausdorff dimensions,  lower and upper packing dimensions) of $\pi^\ba_*\mu$ is constant for $\mathcal L^{md}$-a.e.~$\ba\in \R^{md}$, where $\pi^\ba_*\mu$ stands for  the push-forward of $\mu$ by $\pi^\ba$. In particular, we give a necessary and sufficient condition on $\mu$ so that $\pi^\ba_*\mu$ is exact dimensional for $\mathcal L^{md}$-a.e.~$\ba\in \R^{md}$. Moreover, for every analytic set $E\subset \Sigma$, each of the Hausdorff, packing, lower and upper box-counting dimensions of $\pi^{\ba}(E)$ is constant for  $\mathcal L^{md}$-a.e.~$\ba\in \R^{md}$.  Formal dimension formulas of these projected measures and sets are given. The Hausdorff dimensions of exceptional sets are estimated.
  \end{abstract}

\maketitle
\section{Introduction}
\label{S1}

In this paper, we study various dimensions of the projections of sets and measures under the coding maps associated with typical affine iterated function systems.

By an {\it affine iterated function system} (affine IFS) on $\R^d$  we mean a finite  family ${\mathcal F}=\{f_i\}_{i=1}^m$ of affine mappings from $\R^d$ to $\R^d$,
taking  the form
\begin{equation*}
\label{e-form}
f_i(x)=T_ix+a_i,\qquad i=1,\ldots, m,
\end{equation*}
where $T_i$ are contracting $d\times d$ invertible real  matrices and $a_i\in \R^d$.  It is  well known \cite{Hut81} that  there  exists a unique non-empty compact set $K\subset \R^d$ such that
$$
K=\bigcup_{i=1}^m f_i(K).
$$
We call $K$ the {\it attractor} of $\mathcal F$, or the {\em self-affine set} generated by ${\mathcal F}$.
%In particular, if all the maps in ${\mathcal F}$ are contracting similitudes, we call $K$ a {\em self-similar set}.

In what follows, we let $T_1,\ldots, T_m$ be a fixed family of contracting $d\times d$ invertible real  matrices. Let $(\Sigma,\sigma)$ be the one-sided full shift over the alphabet $\{1,\ldots, m\}$, that is, $\Sigma=\{1,\ldots, m\}^\N$ and $\sigma:\Sigma\to \Sigma$ is  the left shift map.  Endow $\Sigma$ with the  product topology and let $\mathcal P(\Sigma)$ denote the space of  Borel probability measures on $\Sigma$.

For $\va = (a_1, \ldots, a_m) \in \R^{md}$, let $\pi^\ba: \Sigma \to \R^d$ be the coding map associated with the IFS $\{ f_i^{\ba}(x) = T_ix + a_i\}_{i=1}^m$, here we write $f_i^\ba$ instead of $f_i$ to emphasize its dependence of $\ba$. That is,
\begin{equation}
\label{e-pia}
\pi^\ba (\bi) = \lim_{n \to \infty} f^{\va}_{i_1} \circ \cdots \circ f^{\va}_{i_n}(0)
\end{equation}
for $\bi =(i_n)_{n=1}^\infty\in \Sigma$. It is well known \cite{Hut81} that the image $\pi^\ba(\Sigma)$ of $\Sigma$ under $\pi^\ba$ is exactly the attractor of $\{f^\ba_i\}_{i=1}^m$. For $\mu\in \mathcal P(\Sigma)$, let $\pi^\ba_*\mu$ denote the projection of $\mu$ under $\pi^\ba$, that is, $\pi^\ba_*\mu$ is the Borel probability measure on $\R^d$ defined by $$\pi^\ba_*\mu(A)=\mu((\pi^\ba)^{-1}(A))$$
for every Borel set $A\subset \R^d$. This measure is also called the {\it push-forward of $\mu$ by $\pi^\ba$}. In the special case when $\mu$ is the Bernoulli product measure on $\Sigma$ generated by a probability vector $(p_1,\ldots, p_m)$, $\nu:=\pi^\ba_*\mu$ is the unique Borel probability measure on $\R^d$ satisfying
$$
\nu=\sum_{i=1}^m p_i\nu\circ (f_i^\ba)^{-1},
$$
which is called the {\it self-affine measure} associated with $\{f^\ba_i\}_{i=1}^m$ and $(p_1,\ldots, p_m)$. The goal of this paper is to study  dimensional properties of the projections of general Borel sets and measures under $\pi^\ba$.

Let us first introduce some necessary notation and definitions about various dimensions of sets and measures. For  $A\subset \R^d$, we use $\dim_{\rm H}A$, $\dim_{\rm P}A$, $\underline{\dim}_{\rm B}A$ and $ \overline{\dim}_{\rm B}A$ to denote the Hausdorff, packing, lower and upper box-counting dimensions of $A$, respectively (see e.g.~\cite{Fal03, Mattila1995} for the definitions).  If $\underline{\dim}_{\rm B}A= \overline{\dim}_{\rm B}A$ we use $\dim_{\rm B}A$ to denote the common value and call it the box-counting dimension of $A$.

Recall that for a probability measure $\eta$ on $\R^d$,  the {\it local upper and lower dimensions} of $\eta$ at $x\in \R^d$ are defined respectively by
$$\overline{\dim}_{\rm loc}(\eta, x)=\limsup_{r\to 0}\frac{\log \eta (B(x,r))}{\log r},\quad \underline{\dim}_{\rm loc}(\eta, x)=\liminf_{r\to 0}\frac{\log \eta (B(x,r))}{\log r},$$
 where $B(x,r)$ stands for  the closed ball of radius $r$ centered at $x$. If $$\overline{\dim}_{\rm loc}(\eta, x)=\underline{\dim}_{\rm loc} (\eta, x),$$ the common value is denoted as $\dim_{\rm loc}(\eta,x)$ and is called the {\it local dimension} of $\eta$ at $x$. We say that  $\eta$  is  {\it exact dimensional}
if there exists a constant $C$ such that the  local dimension
$\dim_{\rm loc}(\eta, x)$
exists and equals $C$ for $\eta$-a.e.~$x\in \R^d$. It is well known that if $\eta$ is  exact dimensional, then the lower and upper Hausdorff/packing dimensions of $\eta$  coincide and are equal to the involved constant $C$, and so are some other notions of dimension (e.g.~ entropy dimension); see \cite{Young1982, Fal-technique}. Recall that the lower and upper Hausdorff/packing dimensions of $\eta$ are defined by
\begin{equation*}
\begin{split}
\ldim{H} \eta &= \essinf_{x\in {\rm spt}(\eta)} \underline{\dim}_{\rm loc} (\eta, x) , \quad
\udim{H} \eta = \esssup_{x\in {\rm spt}(\eta)} \underline{\dim}_{\rm loc} (\eta, x),\\
\ldim{P} \eta &= \essinf_{x\in {\rm spt}(\eta)}\overline{\dim}_{\rm loc} (\eta, x), \quad
\udim{P} \eta = \esssup_{x\in {\rm spt}(\eta)}\overline{\dim}_{\rm loc} (\eta, x).
\end{split}
\end{equation*}
Equivalently these dimensions of measures can be given in terms of dimensions of sets; see e.g.~ \cite{FLR02}:
\begin{align*}
\ldim{H} \eta &= \inf \{\Dim{H}A: A \text{ is a Borel set with }\eta(A)>0\}, \\
\udim{H} \eta &= \inf \{ \Dim{H} A: A \text{ is a Borel set with }\eta( A)=1 \}, \\
\ldim{P} \eta &= \inf \{ \Dim{P} A: A \text{ is a Borel set with }\eta(A)>0\}, \\
\udim{P} \eta &= \inf \{ \Dim{P} A: A \text{ is a Borel set with }\eta( A)=1 \}.
\end{align*}

In his seminal paper \cite{Fal88},  Falconer  introduced a quantity associated to the matrices $T_1,\ldots, T_m$, nowadays usually called the {\em affinity dimension} $\dim_{\rm AFF}(T_1,\ldots, T_m)$, which is always an upper bound for the upper box-counting dimension of $\pi^\ba(\Sigma)$, and such that when
\begin{equation}
\label{e-norm}
\|T_i\|<\frac{1}{2} \quad \mbox{ for all } 1\leq i\leq m,
\end{equation}
then for ${\mathcal L}^{md}$-a.e.~$\ba$,   $$\dim_{\rm H}\pi^\ba(\Sigma)=\dim_{\rm B}\pi^\ba(\Sigma)=\min \{d, \dim_{\rm AFF}(T_1,\ldots, T_m)\}.$$
 In fact, Falconer proved this with $1/3$ as the upper bound on the norms; it was subsequently shown by Solomyak \cite{Sol98} that $1/2$ suffices.

The affinity dimension $\dim_{\rm AFF}(T_1,\ldots, T_m)$ is defined as follows. Let ${\rm GL}_d(\R)$ denote the collection of all $d\times d$  invertible real matrices. For $T\in {\rm GL}_d(\R)$, let $$\alpha_1(T)\geq\cdots\geq \alpha_d(T)$$ denote the  singular values of $T$. Following \cite{Fal88},  for $s\geq 0$ we define the {\it singular value function} $\phi^s:\; {\rm GL}_d(\R)\to [0,\infty)$ as
\begin{equation}
\label{e-singular}
\phi^s(T)=\left\{
\begin{array}
{ll}
\alpha_1(T)\cdots \alpha_k(T) \alpha_{k+1}^{s-k}(T) & \mbox{ if }0\leq s< d,\\
\det(T)^{s/d} & \mbox{ if } s\geq d,
\end{array}
\right.
\end{equation}
where $k=[s]$ is the integral part of $s$. Then the affinity dimension of the tuple $(T_1,\ldots, T_m)$ is defined by
\begin{equation}\label{e-aff}
\dim_{\rm AFF}(T_1,\ldots, T_m)=\inf\left\{s\geq 0:\; \lim_{n\to \infty}\frac{1}{n} \log \sum_{i_1,\ldots, i_n=1}^m\phi^s(T_{i_1}\cdots T_{i_n})\leq 0\right\}.
\end{equation}

Further developments have been made after the work of Falconer and Solomyak.   In \cite{Kae04} K\"{a}enm\"{a}ki showed that, under the assumption \eqref{e-norm}, there exists a $\sigma$-invariant ergodic measure $\mu$ on $\Sigma$ such that $\underline{\dim}_{\rm H} \pi^\ba_*\mu={\dim}_{\rm H} \pi^\ba(\Sigma)$ for ${\mathcal L}^{md}$-a.e.~$\ba$.  Under the same norm assumption, Jordan, Pollicott and Simon \cite{JPS07} further showed that for every $\sigma$-invariant ergodic measure $\mu$ on $\Sigma$ and for ${\mathcal L}^{md}$-a.e.~$\ba$,  the lower and upper Hausdorff dimensions of $\pi^\ba_*\mu$ coincide and are equal to the so-called Lyapunov dimension of $\mu$; as pointed out by  Jordan \cite{Jor11} and Rossi \cite{Ros14}, in this case $\pi^\ba_*\mu$ is exact dimensional for ${\mathcal L}^{md}$-a.e.~$\ba$. Later, K\"{a}enm\"{a}ki and Vilppolainen \cite{KaVi10} proved that under the assumption \eqref{e-norm}, for every compact subset $E$ of $\Sigma$ with $\sigma E\subset E$ and for ${\mathcal L}^{md}$-a.e.~$\ba$, the Hausdorff and box-counting dimensions of $\pi^\ba(E)$ coincide and are equal to the root of a pressure function. In \cite{JJKKSS14} J\"{a}rvenp\"{a}\"{a} et al.~further proved that for every compact subset $E$ of $\Sigma$, $\dim_{\rm H}\pi^\ba(E)=\min\{d, \dim_{\mathcal M}E\}$ almost surely under the same norm assumption, where $\dim_{\mathcal M}E$ is defined as in \eqref{e-dimM}. Moreover, they obtained some dimensional results for a general class of deterministic or random affine code tree fractals; see also \cite{JJLS2016, JJWW2017}. In \cite{Fami08} Falconer and Miao estimated the Hausdorff dimensions of  exceptional sets of $\ba$ for which $\dim_{\rm H} \pi^\ba(\Sigma)$ takes an exceptionally small value.   Besides these results, for almost all $\ba$ the constancy of the $L^q$ dimensions ($1<q\leq 2$) was also proved  for the projections of  general Borel probability measures under $\pi^\ba$  \cite{Fal99, Fal10}, and the multifractal structure of the projections of Bernoulli product measures under $\pi^\ba$ was investigated in certain cases \cite{BaFe13}.

Inspired by the above developments, in this paper we aim to further investigate the dimensional properties of the projections of arbitrary Borel sets and measures under $\pi^\ba$ for almost all $\ba$.
In particular, we would like to study under which general condition on a given $\mu\in \mathcal P(\Sigma)$, the projection $\pi^\ba_*\mu$ is exact dimensional for  almost all $\ba$. It is known \cite{Fen19} that when $\mu$ is $\sigma$-invariant and  ergodic,  $\pi^\ba_*\mu$ is exact dimensional for every  $\ba\in \R^{md}$; see also \cite{BaKa17,FeHu09} for some earlier results. So it remains to investigate the more general case.

 To state our main results, we still need to introduce some notation and definitions. Let $\mu\in \mathcal P(\Sigma)$. For $x = (x_i)_{i =1}^\infty \in \Sigma$ and $n \in \N$, we write $x|n:\; =x_1 \cdots x_n$ and set $S_n(\mu, x)$ to be the unique number $t\in [0,\infty]$ such that
\begin{equation}
\label{e-2.2}
\phi^t(T_{x|n})= \mu ([x_1 \cdots x_n]),
\end{equation}
where $T_{x|n}:=T_{x_1}\cdots T_{x_n}$ and $[x_1 \cdots x_n]: = \{ (y_i)_{i =1}^\infty : y_i = x_i \text{ for } i = 1, \dots, n\}$.  Next we define
\begin{equation}
\label{e-Smux}
S(\mu, x)= \liminf_{n\to \infty} S_n(\mu,x),\qquad x\in \Sigma,
\end{equation}
and
\begin{equation}
\label{e-Gammamu}
\underline{S}(\mu)  = \essinf_{x\in {\rm spt}\mu} S(\mu, x), \quad
\overline{S}(\mu)  = \esssup_{x\in {\rm spt}\mu} S(\mu, x).
\end{equation}

   For $x,y\in \Sigma$,  let $x\wedge y$ denote  the common initial segment of $x$ and $y$. If $x\neq y$, we let  $\alpha_k(T_{x\wedge y})$ denote the $k$-th singular value of $T_{x\wedge y}$,  $k=1,\ldots, d$. Here for the empty word $\varepsilon$, $T_\varepsilon$ stands for  the $d\times d$ identity matrix.
  Following \cite{JPS07}, we define a function
   $Z_{x\wedge y}: (0,\infty)\to (0,1]$ by
\begin{equation}
\label{ef-4.2}
Z_{x \wedge y}(r)=
\left\{
\begin{array}{ll}
\displaystyle \prod_{k=1}^d  \frac{ \min \{r,\; \alpha_k(T_{x\wedge y})\} } {\alpha_k(T_{x\wedge y})}, & \mbox{ if }x\neq y,\\
1, & \mbox{ otherwise.}
\end{array}
\right.
\end{equation}
It is easy to check that
\begin{equation}
\label{e-Zre}
\frac{1}{Z_{x \wedge y}(r)}
=\prod_{k=1}^d
\frac{ \max \{r,\; \alpha_k(T_{x\wedge y})\} } {r},  \quad \mbox{ if }x\neq y.
\end{equation}
For $x\in \Sigma$ and  $r>0$, define
 \begin{equation}
 \label{e-Gmur}
  G_\mu(x, r)=\int Z_{x\wedge y}(r) \; d\mu(y),
    \end{equation}
and
 \begin{equation}
 \label{e-D1}
 D(\mu, x)=\limsup_{r\to 0} \frac{\log G_\mu(x, r)}{\log r}.
 \end{equation}
Write
\begin{equation}
\label{e-D1'}
 \underline{D}(\mu)=\essinf_{x\in {\rm spt}(\mu)}D(\mu, x),\quad  \overline{D}(\mu)=\esssup_{x\in {\rm spt}(\mu)} D(\mu, x).
 \end{equation}

Now we are ready to formulate our first result.

\begin{thm}
\label{thm-lower-upper}
 Let $\mu\in \mathcal P(\Sigma)$. Then the following properties hold.
\begin{itemize}
\item[(i)] For every $\ba\in \R^{md}$,
\begin{align*}
    \underline{\dim}_{\rm loc}(\pi^\ba_*\mu, \pi^\ba x)&\leq \min\{S(\mu,x),d\} \; \mbox{ and }\\
    \overline{\dim}_{\rm loc}(\pi^\ba_*\mu, \pi^\ba x)&\leq D(\mu,x)
  \end{align*}
for $\mu$-a.e.~$x\in \Sigma$.
  Consequently for every $\ba\in \R^{md}$,
  \begin{align*}
   \ldim{H} \pi^\ba_*\mu &\leq  \min \{ \underline{S}(\mu), d\},\quad
\udim{H} \pi^\ba_*\mu \leq \min \{ \overline{S}(\mu), d \},\\
\ldim{P} \pi^\ba_*\mu &\leq  \underline{D}(\mu),
\quad\qquad\quad\;
\udim{P} \pi^\ba_*\mu \leq \overline{D}(\mu).\\
\end{align*}
\item[(ii)] Assume that $\|T_i\|<1/2$ for $1\leq i\leq m$. Then for $\mathcal L^{md}$-a.e.~$\ba\in \R^{md}$,
    \begin{align*}
    \underline{\dim}_{\rm loc}(\pi^\ba_*\mu, \pi^\ba x)&= \min\{S(\mu,x),d\} \quad \mbox{and}\\
    \overline{\dim}_{\rm loc}(\pi^\ba_*\mu, \pi^\ba x)&= D(\mu,x)
\end{align*}
for $\mu$-a.e.~$x\in \Sigma$. Consequently for $\mathcal L^{md}$-a.e.~$\ba\in \R^{md}$,
   \begin{align*}
\ldim{H} \pi^\ba_*\mu &=  \min \{ \underline{S}(\mu), d\},\quad  \udim{H} \pi^\ba_*\mu = \min \{ \overline{S}(\mu), d \},\\
\ldim{P} \pi^\ba_*\mu &=  \underline{D}(\mu),
\quad\qquad\quad\;
\udim{P} \pi^\ba_*\mu = \overline{D}(\mu).
\end{align*}
\end{itemize}
\end{thm}
This result states in particular that under the assumption \eqref{e-norm}, for every $\mu\in \mathcal P(\Sigma)$, each of the following dimensions (lower and upper Hausdorff/packing dimensions) of $\pi^\ba_*\mu$ is constant almost surely. As a direct consequence, we have the following.

 \begin{thm}
 \label{thm-5.1}
 Let $\mu\in \mathcal P(\Sigma)$. Assume that $\|T_i\|<1/2$ for $1\leq i\leq m$. Then $\pi^\ba_*\mu$ is exact dimensional for $\mathcal L^{md}$-a.e.~$\ba\in \R^{md}$ if and only if one of the following two conditions holds:
 \begin{itemize}
 \item[(i)]
 $\liminf_{n\to \infty} S_n(\mu,x)\geq d$ for $\mu$-a.e.~$x\in \Sigma$.
 \item[(ii)]
 There exists $s\in [0,d)$ such that
 \begin{equation}
 \label{e-equv}
 \liminf_{n\to \infty} S_n(\mu,x)=s \mbox{  and }\limsup_{r\to 0} \frac{\log G_\mu(x,r)}{\log r}=s \quad \mbox{ for $\mu$-a.e.~$x\in \Sigma$}.
 \end{equation}
 \end{itemize}
 \end{thm}

The above theorem gives a complete characterisation of those measures $\mu\in \mathcal P(\Sigma)$ so that
$\pi^\ba_*\mu$ is exact dimensional for almost every $\ba$.  However, the condition \eqref{e-equv} is not easily checked since it involves the quantity
$\limsup_{r\to 0} \frac{\log G_\mu(x,r)}{\log r}$, which is usually difficult to be estimated.
Nevertheless, we are able to provide the following.
 \begin{thm}
 \label{thm-5.2}
 Let $\mu\in \mathcal P(\Sigma)$. Then the following properties hold.
 \begin{itemize}
 \item[(i)] For each $s\in [0,d)$,  \eqref{e-equv} holds if
 \begin{equation}
 \label{e-equv1}
 \lim_{n\to \infty} S_n(\mu,x)=s\quad \mbox{ for $\mu$-a.e.~$x\in \Sigma$}.
 \end{equation}
 \item[(ii)] For every $s\in [0,d)\backslash \N$,  \eqref{e-equv} holds if and only if  \eqref{e-equv1} holds.
 \item[(iii)] Suppose that all $T_i$ ($i=1,\ldots, m$) are scalar multiples of orthogonal matrices. Then for each $s\in [0,d)$, \eqref{e-equv} holds if and only if  \eqref{e-equv1} holds.
   \end{itemize}
 \end{thm}

We remark that  the condition $s\in [0,d)\backslash \N$ in Theorem \ref{thm-5.2}(ii) is sharp. Indeed for every integer $d\geq 2$ and  $s\in \{1,\ldots, d-1\}$,  we can construct a tuple $(T_1,\ldots, T_m)$ of $d\times d$ matrices and a measure $\mu\in \mathcal P(\Sigma)$ such that \eqref{e-equv} does not imply \eqref{e-equv1}. See Example \ref{ex-1} for the details.

 Next we state our result on the Hausdorff, packing, lower and upper box-counting dimensions of  the projections of analytic sets. Recall that a subset of $\Sigma$ is said to be {\em analytic} if it is a continuous image of the Baire space $\N^\N$. It is known that any Borel subset of $\Sigma$ is an analytic set (see e.g. \cite{Rogers70}).

 \begin{thm} \label{thm-main2}
Let $E \subset \Sigma$ be an analytic  set. Then the following properties hold.
\begin{itemize}
\item[(i)] For every $\ba\in \R^{md}$,
\begin{align*}
\dim_{\mathrm{H}} \pi^\ba(E) &\leq  \min \{ \dim_{\M} E, d\},\\
\dim_{\mathrm{P}} \pi^\ba(E) &\leq \sup_{\mu\in {\mathcal P}(\Sigma):\; {\rm spt}(\mu)\subset  E}\underline{D}(\mu),\\
\underline{\dim}_{\rm B} \pi^\ba(E) &\leq \underline{\dim}_C\overline{E},\\
\overline{\dim}_{\rm B} \pi^\ba(E) &\leq \overline{\dim}_C\overline{E},
\end{align*}
where $\dim_{\M}E$ is defined  in \eqref{e-dimM}, and $\underline{\dim}_C\overline{E}$, $\overline{\dim}_C\overline{E}$ are defined  in \eqref{def prof}.
\item[(ii)] Assume that $\|T_i\|<1/2$ for $1\leq i\leq m$. Then  for ${\mathcal L}^{md}$-a.e.~$\ba\in \R^{md}$,
    \begin{align*}
\dim_{\mathrm{H}} \pi^\ba(E) &=  \min \{ \dim_{\M} E, d\},\\
\dim_{\mathrm{P}} \pi^\ba(E) &= \sup_{\mu\in {\mathcal P}(\Sigma):\; {\rm spt}(\mu)\subset  E}\underline{D}(\mu) =\sup_{\mu\in {\mathcal P}(\Sigma):\; {\rm spt}(\mu)\subset  E}\overline{D}(\mu),\\
\underline{\dim}_{\rm B} \pi^\ba(E) &= \underline{\dim}_C\overline{E},\\
\overline{\dim}_{\rm B} \pi^\ba(E) &= \overline{\dim}_C\overline{E},
\end{align*}
Moreover if $\dim_{\M} E>d$, then ${\mathcal L}^d(\pi^\ba(E))>0$ for ${\mathcal L}^{md}$-a.e.~$\ba\in \R^{md}$.
\end{itemize}
\end{thm}

This result states, in particular, that under the assumption \eqref{e-norm}, each of the Hausdorff, packing, lower and upper box-counting dimensions of $\pi^\ba(E)$ is constant almost surely.  We remark that the analyticity assumption on $E$ is not needed for the statements for the lower and upper box-counting dimensions.
%It is worth pointing out that for any analytic subset $E$ of $\Sigma$,  $$\dim_\M E=\sup_{\mu\in \mathcal P(\Sigma):\; {\rm spt}\mu\subset E}\underline{\Gamma}(\mu)=\sup_{\mu\in \mathcal P(\Sigma):\; {\rm spt}\mu\subset E}\overline{\Gamma}(\mu);$$
%see Remark \ref{rem-5.3}.

Besides the above results, in Section \ref{S9} we also  provide some inequalities for the Hausdorff dimensions of the sets of translational vectors at which the dimensions of projected sets and measures are exceptionally small; see Theorems \ref{thm-9.1}-\ref{thm-9.2}.

Although the settings are a bit different, our constancy results  (Theorems \ref{thm-lower-upper} and \ref{thm-main2}) on the dimensions of projected  measures and sets on typical self-affine sets are analogous  to the corresponding constancy results for the dimensions of sets and measures under orthogonal projections.   More precisely, Theorem \ref{thm-lower-upper} is analogous to the constancy results of \cite{HK97, FaON99} for the lower and upper local dimensions of measures and to that of \cite{HuTa94, FaHo97} for the Hausdorff and packing dimensions of measures under orthogonal projections; whilst Theorem \ref{thm-main2} is analogous to the work of \cite{Mar54, Mat75} for the Hausdorff dimension of  sets, to that of \cite{FaHo97} for the packing dimension of  sets, and to that of \cite{How01, Fal21} for the  box-counting dimension of  sets under orthogonal projections. Meanwhile, our results are also analogous to the theorems on the Hausdorff and packing dimensions of the images under factional Brownian motions (see \cite{Kahane1985, Xiao1997, ShiehXiao2010, Falconer2020}).

 For  the proofs of Theorems \ref{thm-lower-upper} and \ref{thm-main2},  besides adopting and extending some ideas and strategies from the papers \cite {Fal88, Fal21, FaHo97, HK97, Jor11, JPS07}, we also need to develop new techniques (e.g. Lemma \ref{lem-1.6} and Proposition \ref{key Prop}) to analyse the covering properties of the projections of  sets and measures under the coding maps.   The proof of Theorem \ref{thm-5.2} involves a  lengthy and delicate estimation of $G_\mu(x,r)$. Using a similar strategy we are able to provide a
 simple criterion (see Theorem \ref{thm-8.1}) for the exact dimensionality of projected measures under typical orthogonal projections. For the proofs of Theorems \ref{thm-9.1} and \ref{thm-9.2}, we apply and extend some estimations in \cite{Fami08}.

 Recently significant progress has been made in characterizing concrete (planar) self-affine sets and self-affine measures of which the Hausdorff dimensions coincide with the affinity and Lyapunov dimensions; see \cite{BHR19, HocRap22} and the references therein.  It is expected that one may specify concrete $\ba$ for which the equalities in Theorems \ref{thm-lower-upper}(ii) and \ref{thm-main2}(ii) hold under reasonable assumptions on $T_1,\ldots, T_m$.

The paper is organized as follows. In Sections \ref{S2} and \ref{S3} we investigate the lower and upper local dimensions of projected measures separately. Theorem \ref{thm-lower-upper} is a simple combination of Theorems \ref{thm-lower} and  \ref{thm-packing}. In Section \ref{S4} we prove Theorems \ref{thm-5.1} and \ref{thm-5.2}. In Sections \ref{S5}, we investigate the Hausdorff and packing dimensions of projected sets. In Section \ref{S6} we investigate the lower and upper  box-counting dimensions of projected sets. Theorem \ref{thm-main2} is then a direct combination of Theorems  \ref{main2}, \ref{thm-packing-sets} and \ref{thm-box}. In Section \ref{S8} we prove Theorem \ref{thm-8.1}, which is an analogue of Theorem \ref{thm-5.2} for orthogonal projections. In Section \ref{S9}, we prove Thorems \ref{thm-9.1} and \ref{thm-9.2}  which estimate  the Hausdorff dimensions of the exceptional sets. In Section \ref{S10} we give some final remarks and questions.
\section{Lower local dimensions of projected measures}
\label{S2}

Throughout this section, let $T_1,\ldots, T_m$ be a family  of $d \times d$  invertible real matrices with $\|T_i\|<1$ for $1\leq i\leq m$. For $\ba=(a_1,\ldots, a_m)\in \R^{md}$, let $\pi^\ba:\Sigma\to \R^d$ be the coding map associated with the IFS $\{f_i^\ba(x)=T_ix+a_i\}_{i=1}^m$;  see \eqref{e-pia}. For short we write $f_I^\ba:=f^\ba_{i_1}\circ \cdots\circ f^\ba_{i_n}$ and $T_I:=T_{i_1}\cdots T_{i_n}$ for $I=i_1\cdots i_n\in \Sigma_n:=\{1,\ldots, m\}^n$.

 Recall that we have defined the quantities  $S(\mu,x)$, $\underline{S}(\mu)$ and $\overline{S}(\mu)$ for $\mu\in \mathcal P(\Sigma)$ and $x\in \Sigma$ in \eqref{e-Smux} and \eqref{e-Gammamu}. The main result of this section is the following.

\begin{thm}
\label{thm-lower}
 Let $\mu\in \mathcal P(\Sigma)$. Then the following properties hold.
\begin{itemize}
\item[(i)] For every $\ba\in \R^{md}$, $$\underline{\dim}_{\rm loc}(\pi^\ba_*\mu, \pi^\ba x)\leq \min\{S(\mu,x),d\} \quad \mbox{ for $\mu$-a.e.~$x\in \Sigma$},
$$  consequently,
   $\ldim{H} \pi^\ba_*\mu \leq  \min \{ \underline{S}(\mu), d\}$,
$\udim{H} \pi^\ba_*\mu \leq \min \{ \overline{S}(\mu), d \}.
$

\item[(ii)] Assume that $\|T_i\|<1/2$ for $1\leq i\leq m$. Then for $\mathcal L^{md}$-a.e.~$\ba\in \R^{md}$, $$\underline{\dim}_{\rm loc}(\pi^\ba_*\mu, \pi^\ba x)= \min\{S(\mu,x),d\} \quad \mbox{ for $\mu$-a.e.~$x\in \Sigma$},
$$
 consequently,
   $
\ldim{H} \pi^\ba_*\mu =  \min \{ \underline{S}(\mu), d\}$, $\udim{H} \pi^\ba_*\mu = \min \{ \overline{S}(\mu), d \}.
$
\end{itemize}
\end{thm}

To prove part (i) of the above theorem, we need the following result.
\begin{lem}[Jordan {\cite{Jor11}}] \label{LemJ}
Let $\ba\in \R^{md}$. There is a positive constant $c>0$ which depends on $\ba$ and $T_1,\ldots, T_m$ such that the following property holds.  For every $\epsilon\in (0,1)$, $\mu\in \mathcal P(\Sigma)$ and  $\ell\in \{0,1,\ldots, d-1\}$,  we have for $\mu$-a.e.~$x \in \Sigma$,
\begin{equation}
\label{e-lemJ}
\pi^\ba_*\mu \left( B(\pi^\ba x, c \alpha_{\ell+1}(T_{x|n})) \right) \ge (1-\epsilon)^n \dfrac{ \mu([x|n])}{N_{\ell}(x|n)}\quad  \mbox{ for large enough $n$},
\end{equation}
where
\begin{equation}
\label{e-Nxn}
N_{\ell}(x|n) := \alpha_1(T_{x|n}) \cdots \alpha_{\ell}(T_{x|n}) \alpha_{\ell+1}^{-\ell} (T_{x|n}).
\end{equation}
\end{lem}
\begin{proof}
For the reader's convenience, we include here the detailed argument of Jordan \cite{Jor11}.

Let $\ba=(a_1,\ldots, a_m)\in \R^{md}$. Take a large $R=R(\ba, T_1,\ldots, T_m)>0$ such that
$$
f^\ba_i(B(0,R))\subset B(0, R)  \qquad \mbox{ for } i=1,\ldots,m,
$$
where $f^\ba_i(x):=T_ix+a_i$.  Clearly the attractor $\pi^{\ba}(\Sigma)$ of the IFS $\{f^\ba_i\}_{i=1}^m$ is contained in $B(0, R)$. Take $c=4 R\sqrt{d}$. Below we show that the statement of the lemma holds for such $c$.

Let $\epsilon\in (0,1)$, $\mu\in \mathcal P(\Sigma)$ and $\ell\in \{0,\ldots, d-1\}$. For $n\in \N$, let $\Lambda_n$ denote the set of the points $x\in \Sigma$ such that
$$
\pi^\ba_*\mu \left( B(\pi^\ba x, c \alpha_{\ell+1}(T_{x|n})) \right) < (1-\epsilon)^n \dfrac{ \mu([x|n])}{N_{\ell}(x|n)}.
$$
To prove that \eqref{e-lemJ} holds for $\mu$-a.e.~$x\in \Sigma$,   by the Borel-Cantelli lemma it suffices to show that
\begin{equation}
\label{e-teta}
\sum_{n=1}^\infty \mu(\Lambda_n)<\infty.
\end{equation}
For this purpose, below let us  estimate $\mu(\Lambda_n)$.

 Fix $n\in \N$ and $I\in \Sigma_n$. Notice that  $f^\ba_I(B(0, R))$ is an ellipsoid of semi-axes
$$
R\alpha_1(T_{I})\geq \ldots\geq R\alpha_d(T_{I}),
$$
so it can be covered by $2^\ell \prod_{i=1}^\ell \frac{\alpha_i(T_{I})}{\alpha_{\ell+1}(T_{I})}$ balls of radius $2R\sqrt{d}\alpha_{\ell+1}(T_{I})$. Since $\pi^\ba([I])$ is contained in $f^\ba_I(B(0, R))$,  it follows that there exists a nonnegative  integer $L$ satisfying
\begin{equation}
\label{e-JL}
L\leq 2^\ell \prod_{i=1}^\ell \frac{\alpha_i(T_{I})}{\alpha_{\ell+1}(T_{I})}=2^\ell N_{\ell}(I)
\end{equation}
 such that $\pi^\ba(\Lambda_n\cap [I])$ can be covered by $L$ balls of radius $2R\sqrt{d}\alpha_{\ell+1}(T_{I})$, say, $B_1,\ldots, B_L$. We may assume that $\pi^\ba(\Lambda_n\cap [I])\cap B_i\neq \emptyset$ for each $1\leq i\leq L$. Hence for each $i$, we may pick $x^{(i)}\in \Lambda_n \cap [I]$ such that $\pi^\ba x^{(i)} \in B_i$. Clearly
\begin{equation}
\label{e-Bi}
B_i\subset B\left(\pi^\ba x^{(i)}, 4R\sqrt{d}\alpha_{\ell+1}(T_{I})\right)=B\left(\pi^\ba x^{(i)}, c\alpha_{\ell+1}(T_{I})\right).
\end{equation}
Since $x^{(i)}\in \Lambda_n \cap [I]$, by the definition of $\Lambda_n$ we obtain
\begin{equation}
\label{e-pi}
\pi^\ba_*\mu \left(B\left(\pi^\ba x^{(i)}, c\alpha_{\ell+1}(T_{I})\right)\right)< (1-\epsilon)^n\frac{\mu([I])}{N_{\ell}(I)}.
\end{equation}
It follows that
\begin{eqnarray*}
\mu(\Lambda_n\cap [I])&\leq& \mu\circ (\pi^\ba)^{-1}(\pi^\ba(\Lambda_n\cap [I]))\\
&\leq & \pi^\ba_*\mu\left(\bigcup_{i=1}^LB_i\right)\\
&\leq & \pi^\ba_*\mu\left(\bigcup_{i=1}^LB\left(\pi^\ba x^{(i)}, c\alpha_{\ell+1}(T_{I})\right)\right)\qquad \mbox{(by \eqref{e-Bi})}\\
&\leq & L (1-\epsilon)^n\frac{\mu([I])}{N_{\ell}(I))}\qquad \mbox{(by \eqref{e-pi})}\\
&\leq & 2^\ell(1-\epsilon)^n\mu([I])\qquad \mbox{(by \eqref{e-JL})}.
\end{eqnarray*}
Summing over $I\in \Sigma_n$ yields that $\mu(\Lambda_n)\leq 2^\ell (1-\epsilon)^n$, which implies \eqref{e-teta}.
\end{proof}

\begin{proof}[Proof of Theorem \ref{thm-lower}(i)]
Let $\ba\in \R^{md}$ and $\mu\in \mathcal P(\Sigma)$. We need to show that  for $\mu$-a.e.~$x\in \Sigma$,
\begin{equation*}
\label{e-q1}
\underline{\dim}_{\rm loc}(\pi^\ba_*\mu, \pi^\ba x)\leq \min\{S(\mu,x),d\}.
\end{equation*}
For this purpose, it is enough to show that for every $\delta>0$,
\begin{equation}
\label{e-q1'}
\underline{\dim}_{\rm loc}(\pi^\ba_*\mu, \pi^\ba x)\leq \min\{S(\mu,x),d\}+\delta \quad \mbox{ for $\mu$-a.e.~$x\in \Sigma$}.
\end{equation}
To this end, let $\delta>0$. Pick $\epsilon \in (0,1)$ such that
\begin{equation}
\label{e-to2}
\frac{\log (1-\epsilon)}{\log \alpha_+}<\delta,
\end{equation} where $\alpha_+:=\max\{\|T_i\|:\; i=1,\ldots, m\}$.

Set $$A_k:=\{x\in \Sigma:\; k\leq S(\mu,x)<k+1\},  \quad k=0,\ldots, d-1.
$$ Since $\underline{\dim}_{\rm loc}(\pi^\ba_*\mu, \pi^\ba x)\leq d$ for $\mu$-a.e.~$x\in \Sigma$,  it suffices to show that \eqref{e-q1'} holds for $\mu$-a.e.~$x\in \bigcup_{k=0}^{d-1}A_k$.

Fix $k\in \{0,\ldots, d-1\}$. By Lemma \ref{LemJ}, there exists $A'_k\subset A_k$ with $\mu(A_k\backslash A_k')=0$ such that for each $x\in A_k'$,
\begin{equation}
\label{e-lemJ'}
\pi^\ba_*\mu \left( B(\pi^\ba x, c \alpha_{k+1}(T_{x|n})) \right) \ge (1-\epsilon)^n \dfrac{ \mu([x|n])}{N_{k}(x|n)}\quad  \mbox{ for large enough $n$},
\end{equation}
where $N_{k}(x|n)$ is defined as in \eqref{e-Nxn}.

Now let $x\in A_k'$. Let $\gamma\in (0, k+1-S(\mu,x))$. Then there exists a subsequence $(n_j)$ of natural numbers such that
\begin{equation}
\label{e-to1}
k\leq S_{n_j}(\mu,x)+\gamma<k+1\quad \mbox{ and }\quad  \lim_{j\to \infty} S_{n_j}(\mu,x)=S(\mu,x).
\end{equation}
Observe that
\begin{align*}
\dfrac{ \mu([x|n_j])}{N_{k}(x|n_j)}=\dfrac{ \phi^{S_{n_j}(\mu,x)}(T_{x|n_j})}{\phi^k(T_{x|n_j})\alpha_{k+1}^{-k}(T_{x|n_j})}
\geq \dfrac{ \phi^{S_{n_j}(\mu,x)+\gamma}(T_{x|n_j})}{\phi^k(T_{x|n_j})\alpha_{k+1}^{-k}(T_{x|n_j})}
=\alpha_{k+1}^{S_{n_j}(\mu,x)+\gamma}(T_{x|n_j}),
\end{align*}
where in the last equality we have used \eqref{e-to1}. Hence by \eqref{e-lemJ'},
\begin{align*}
\underline{\dim}_{\rm loc}(\pi^\ba_*\mu, \pi^\ba x)&\leq \liminf_{j\to \infty} \frac{\log \pi^\ba_*\mu \left( B(\pi^\ba x, c \alpha_{k+1}(T_{x|n_j})) \right)}{\log  \alpha_{k+1}(T_{x|n_j}) }\\
& \leq \liminf_{j\to \infty} \frac{\log \left((1-\epsilon)^{n_j}\dfrac{ \mu([x|n_j])}{N_{k}(x|n_j)}\right)}{\log  \alpha_{k+1}(T_{x|n_j}) }\\
&\leq \liminf_{j\to \infty} \left(S_{n_j}(\mu,x)+\gamma+ \frac{n_j\log (1-\epsilon)}{\log  \alpha_{k+1}(T_{x|n_j}) }\right)\\
&\leq S(\mu,x)+\gamma+\delta \qquad \mbox{ (use $\alpha_{k+1}(T_{x|n_j})\leq (\alpha_+)^{n_j}$ and \eqref{e-to2})}.
\end{align*}
Since $\gamma$ is arbitrarily taken in $(0, k+1-S(\mu,x))$, we get
$$\underline{\dim}_{\rm loc}(\pi^\ba_*\mu, \pi^\ba x)\leq S(\mu,x)+\delta.$$
That is,  \eqref{e-q1'} holds for every $x\in A_k'$,  so it holds  for $\mu$-a.e.~$x\in A_k$, as desired.
\end{proof}

Next we turn to the proof of part (ii) of Theorem \ref{thm-lower}. We need several lemmas.
\begin{lem}[\cite{SY97}]
\label{lem-SY}
Let $\nu\in \mathcal P(\R^d)$ with compact support and $x\in \R^d$. Then
$$
\underline{\dim}_{\rm loc}(\nu,x)=\sup\left\{s\geq 0:\; \int |x-y|^{-s} d\nu(y)<\infty\right\}.
$$
\end{lem}
\begin{proof}
The equality was first observed in \cite{SY97}. The reader is referred to \cite[Theorem 3.4.2]{BishopPeres17} for an implicit proof.
\end{proof}

\begin{lem}
\label{lem-Smu}
Let $\mu\in \mathcal P(\Sigma)$ and $x\in \Sigma$. Then
$$
S(\mu,x)=\sup\left\{s\geq 0:\; \int \frac{1}{\phi^s(T_{x\wedge y})} d\mu(y)<\infty\right\}.
$$
\end{lem}
\begin{proof}
We first show that if $s>S(\mu,x)$, then $\int \frac{1}{\phi^s(T_{x\wedge y})} d\mu(y)=\infty$.  Choose $\delta>0$ so that $s-\delta>S(\mu,x)$. Then there exists a subsequence $(n_j)$ of natural numbers such that $S_{n_j}(\mu,x)<s-\delta$, which implies that
$$
\mu([x|n_j])=\phi^{S_{n_j}(\mu,x)}(T_{x|n_j})\geq \phi^{s-\delta}(T_{x|n_j})\geq \phi^{s}(T_{x|n_j})(1/\alpha_+)^{n_j\delta},
$$
where $\alpha_+=\max\{\|T_i\|:\; i=1,\ldots, m\}$. It follows that
$$\int \frac{1}{\phi^s(T_{x\wedge y})} d\mu(y)\geq \int_{[x|n_j]} \frac{1}{\phi^s(T_{x|n_j})} d\mu(y)= \frac{\mu([x|n_j])}{\phi^s(T_{x|n_j})} \geq (1/\alpha_+)^{n_j\delta}.
$$
Letting $j\to \infty$ we have $\int \frac{1}{\phi^s(T_{x\wedge y})} d\mu(y)=\infty$.

Next we show that $\int \frac{1}{\phi^s(T_{x\wedge y})} d\mu(y)<\infty$ for $0\leq s<S(\mu,x)$. Choose $\delta>0$ such that $s+\delta<S(\mu,x)$. Then there exists $n_0$ such that $S_n(\mu,x)>s+\delta$ for all $n\geq n_0$. It follows that for $n\geq n_0$,
$$\mu([x|n])=\phi^{S_{n}(\mu,x)}(T_{x|n})<\phi^{s+\delta}(T_{x|n})\leq \phi^{s}(T_{x|n}) \alpha_+^{n\delta},$$
which implies, in particular, that $\mu(\{x\})=0$. Hence
\begin{align*}
\int \frac{1}{\phi^s(T_{x\wedge y})} d\mu(y)&=\sum_{n=0}^\infty \frac{1}{\phi^s(T_{x|n})}(\mu([x|n])-\mu([x|n+1]))\\
&\leq \sum_{n=0}^\infty \frac{\mu([x|n])}{\phi^s(T_{x|n})}\\
&\leq \sum_{n=0}^{n_0-1} \frac{\mu([x|n])}{\phi^s(T_{x|n})}+\sum_{n=n_0}^\infty\alpha_+^{n\delta}<\infty.
\end{align*}
This completes the proof.
\end{proof}

\begin{lem} [{\cite[Lemma 3.1]{Fal88}, \cite[Proposition 3.1]{Sol98}}]
\label{lem-2.5}
Let $\rho>0$. If $s$ is non-integral with $0<s<d$ and $\|T_i\|<1/2$ for $1\leq i\leq m$, then there exists a number $c=c(\rho, T_1,\ldots, T_m)>0$ such that
\begin{equation}
\label{e-Falconer}
\int_{B_\rho}\frac{d\ba}{|\pi^\ba x-\pi^\ba y|^s}\leq \frac{c}{\phi^s(T_{x\wedge y})}
\end{equation}
 for all distinct $x,y\in \Sigma$, where $B_\rho$ denotes the closed ball in $\R^{md}$  of radius $\rho$ centred at the origin.
\end{lem}

Now we are ready to prove part (ii) of Theorem \ref{thm-lower}.

\begin{proof}[Proof of Theorem \ref{thm-lower}(ii)]
According to part (i) of the theorem, we only need to show that for $\mathcal L^{md}$-a.e.~$\ba\in \R^{md}$,
$$
 \underline{\dim}_{\rm loc}(\pi^\ba_*\mu, \pi^\ba x)\geq  \min\{S(\mu,x),d\}\quad \mbox{ for $\mu$-a.e.~$x\in \Sigma$}.
 $$

 To this end, we adapt the arguments in the proof of \cite[Theorem 4.1]{HK97}.  Let $\rho>0$.  For a given non-integral $s\in (0,d)$ and a positive integer $N$, let $\Lambda_N$ be the set of $x$ for which
 $$
 \int_\Sigma\frac{1}{\phi^s(T_{x\wedge y})} d\mu(y)<N.
 $$
 Notice that by Lemma \ref{lem-Smu}, the set of all $x$ for which $S(\mu,x)>s$ is contained in the union of $\Lambda_N$ for $N\geq 1$.  Appying Fubini's theorem,
  \begin{align*}
 \int_{B_\rho}\int_{\Lambda_N} \int_{\R^d} \frac{d\pi^\ba_*\mu(z)}{|\pi^\ba x-z|^s} d\mu(x) d\ba&= \int_{B_\rho}\int_{\Lambda_N} \int_{\Sigma} \frac{1}{|\pi^\ba x-\pi^\ba y|^s} d\mu(y)d\mu(x) d\ba\\
&=\int_{\Lambda_N}  \int_{\Sigma} \int_{B_\rho} \frac{1}{|\pi^\ba x-\pi^\ba y|^s}d\ba
 d\mu(y)d\mu(x) \\
 &\leq \int_{\Lambda_N}  \int_{\Sigma}  \frac{c}{\phi^s(T_{x\wedge y})}d\mu(y)d\mu(x) \quad\mbox{(by \eqref{e-Falconer})}\\
 &\leq cN.
  \end{align*}
 It follows that for $\mathcal L^{md}$-a.e.~$\ba\in B_\rho$,
 $\displaystyle \int_{\Lambda_N} \int_{\R^d} \frac{d\pi^\ba_*\mu(z)}{|\pi^\ba x-z|^s} d\mu(x)<\infty$ and hence
 $$
 \int_{\R^d} \frac{d\pi^\ba_*\mu(z)}{|\pi^\ba x-z|^s}<\infty \quad\mbox{ for $\mu$-a.e.~$x\in \Lambda_N$}.
 $$
 Taking the union over $N$, we have for $\mathcal L^{md}$-a.e.~$\ba\in B_\rho$,
 $$
 \int_{\R^d} \frac{d\pi^\ba_*\mu(z)}{|\pi^\ba x-z|^s}<\infty \quad\mbox{ for $\mu$-a.e.~$x$ with $S(\mu,x)>s$}.
 $$
It follows from  Lemma \ref{lem-SY} that  for $\mathcal L^{md}$-a.e.~$\ba\in B_\rho$,
 $$
 \underline{\dim}_{\rm loc}(\pi^\ba_*\mu, \pi^\ba x)\geq s \quad\mbox{ for $\mu$-a.e.~$x$ with $S(\mu,x)>s$}.
 $$
Thus we have shown that for all non-integral $s\in (0,d)$,
$$
\mu\left(\{x\in \Sigma:\; S(\mu,x)>s> \underline{\dim}_{\rm loc}(\pi^\ba_*\mu, \pi^\ba x)\}\right)=0
$$
for $\mathcal L^{md}$-a.e.~$\ba\in B_\rho$.  Taking the union over all non-integral rational $s$ in $(0,d)$, we conclude that
for $\mathcal L^{md}$-a.e.~$\ba\in B_\rho$,
$$\mu\left(\{x\in \Sigma:\; \min\{S(\mu,x),d\}> \underline{\dim}_{\rm loc}(\pi^\ba_*\mu, \pi^\ba x)\}\right)=0,$$
for if $\min\{S(\mu,x),d\}> \underline{\dim}_{\rm loc}(\pi^\ba_*\mu, \pi^\ba x)$, then there is a non-integral rational $s$ in $(0,d)$ such that $S(\mu,x)>s> \underline{\dim}_{\rm loc}(\pi^\ba_*\mu, \pi^\ba x)$.
\end{proof}

\section{Upper local dimensions of projected measures}
\label{S3}

In this section we investigate the upper local dimensions of the projections of Borel measures under the coding map $\pi^\ba$. Recall that we have defined the quantities $Z_{x\wedge y}(r)$, $G_\mu(x,r)$, $D(\mu,x)$, $\underline{D}(\mu)$ and $\overline{D}(\mu)$ for $\mu\in \mathcal P(\Sigma)$,   $x,y\in \Sigma$ and $r>0$ in Section \ref{S1}; see \eqref{ef-4.2}-\eqref{e-D1'}.

The main result of this section is the following.

 \begin{thm}
 \label{thm-packing}
 Let $\mu\in \mathcal P(\Sigma)$. Then the following properties hold.
 \begin{itemize}
 \item[(i)] For every  $\ba\in \R^{md}$,
 $$\overline{\dim}_{\rm loc}(\pi^\ba_*\mu, \pi^\ba x)\leq D(\mu,x) \mbox{  for $\mu$-a.e.~$x\in \Sigma$},
 $$
   consequently,
 $\ldim{P}(\pi^\ba_*\mu)\leq \underline{D}(\mu)$  and $\udim{P}(\pi^\ba_*\mu)\leq \overline{D}(\mu).$
 \item[(ii)] Assume that $\|T_i\|<1/2$ for $1\leq i\leq m$.  Then for $\mathcal L^{md}$-a.e.~$\ba\in \R^{md}$,
 $$\overline{\dim}_{\rm loc}(\pi^\ba_*\mu, \pi^\ba x)= D(\mu,x) \mbox{  for $\mu$-a.e.~$x\in \Sigma$},
 $$
consequently,   $\ldim{ P}(\pi^\ba_*\mu)=\underline{D}(\mu)$ and $\udim{P}(\pi^\ba_*\mu)=\overline{D}(\mu)$.
 \end{itemize}
 \end{thm}

\subsection{Proof of part (i) of Theorem \ref{thm-packing} }

 The proof  is based on several lemmas.

 For $ \rho >0$, let  $\mathcal{D_{\rho}}$ be the partition of $\R^d$ defined by
	\begin{equation}
\label{e-drho}
  \mathcal{D_{\rho}}= \left\{  \prod_{i=1}^{d} [k_i\rho,(k_{i}+1)\rho): \; k_i \in \Z\right\}.
  \end{equation}
  We begin with a simple geometric observation.

 \begin{lem}
 \label{lem-ae}
 Let $\ba\in \R^{md}$. Then there exists $C'=C'(\ba)>0$  such that  for all $r>0$ and  $\bi\in \Sigma_*:=\bigcup_{n=0}^\infty \Sigma_n$,
 $$
 \# \left \{Q: \; Q\in  \D_{r/\sqrt{d}},\; Q\cap f^\ba_{\bi} (\pi^\ba(\Sigma))\neq \emptyset \right\} \leq  \frac{C'}{Z_{\bi}(r)},
 $$
 where $Z_{\bi}(r)$ is defined as in \eqref{ef-4.2}.
 \end{lem}
 \begin{proof}
 Let $\ba\in \R^{md}$, $r>0$ and $\bi\in \Sigma_*$.  Write for brevity $$u:= 2 \max\{1, \; {\rm diam}(\pi^\ba(\Sigma))\}.$$
 Notice that $f^\ba_{\bi} (\pi^\ba(\Sigma))$ is contained in a rectangular parallelepiped of side lengths $u\alpha_1(T_{\bi})$, $\ldots$, $u \alpha_d(T_{\bi})$.  We can divide such a parallelepiped
 into at most
 $$
  \prod_{k=1}^d \frac{2\max\{u\alpha_k(T_{\bi}),\; r\}} {r}\leq \prod_{k=1}^d \frac{2u\max\{\alpha_k(T_{\bi}),\; r\}} {r}= \frac{(2u)^d}{Z_{\bi}(r)}
 $$
 cubes of side length $r$, here we have used \eqref{e-Zre} in the last equality.  Take $C'=(2u)^d (d+1)^d$.  Clearly $C'$ is independent of $r$ and $\bi$. The lemma then follows from the fact that  each cube of side length $r$ intersects at most $(d+1)^d$ elements in $\D_{r/\sqrt{d}}$.
 \end{proof}

Recall the definition of  $G_{\mu}(x,r)$ for $\mu\in \mathcal P(\Sigma)$, $x\in \Sigma$ and $r>0$; see \eqref{e-Gmur}.
\begin{lem}
\label{lem-gmu}
Let $\alpha_+:=\max\{\|T_i\|:\; i=1,\ldots, m\}$ and $r\in (0,1)$. Let $\ell=\ell(r)$ be the smallest integer so that $\alpha_+^\ell\leq r$, that is, $\ell=\lceil \log r/\log \alpha_+ \rceil$. Then for every $x\in \Sigma$,
\begin{equation}
\label{e-gmu}
G_\mu(x,r)= \left(\sum_{j=0}^{\ell-1} Z_{x|j}(r)\left(\mu([x|j])-\mu([x|j+1])\right)\right)+Z_{x|\ell}(r)\mu([x|\ell]).
\end{equation}
Consequently, the mapping $x\mapsto G_\mu(x,r)$ is continuous on $\Sigma$.
\end{lem}
\begin{proof}
Let $x\in \Sigma$.  We first show that
 \begin{equation}
 \label{ef-4.11}
 Z_{x|j}(r)=1 \quad \mbox{ for every }j\geq \ell.
  \end{equation}
  To see this, observe that  for  $j\geq \ell$, $\|T_{x|j}\|\leq \alpha_+^j\leq \alpha_+^\ell\leq r$, so $\alpha_k(T_{x|j})\leq r$ for $1\leq k\leq d$, which implies that
 $$
 Z_{x|j}(r)=\prod_{k=1}^d \frac{\min\{r, \alpha_k(T_{x|j})\}}
 {\alpha_k(T_{x|j})}=1.
 $$
This proves \eqref{ef-4.11}. Now by definition,
\begin{eqnarray*}
  G_\mu(x, r)&=& \int Z_{x\wedge y}(r) \; d\mu(y)\\
  &= & \left( \sum_{j=0}^{\infty} Z_{x|j}(r) \mu\{y:\; |x\wedge y|=j\}\right)+\mu(\{x\})\\
  &=&\left(\sum_{j=0}^{\ell-1} Z_{x|j}(r) \mu\{y:\; |x\wedge y|=j\}\right) +\left(\sum_{j=\ell}^{\infty} \mu\{y: |x\wedge y|=j\}\right)\\
  &\mbox{}&\qquad\qquad\qquad +\mu(\{x\}) \qquad\qquad\qquad\mbox{(by \eqref{ef-4.11})} \\
  &=&\left(\sum_{j=0}^{\ell-1} Z_{x|j}(r) \mu\{y:\; |x\wedge y|=j\}\right) +\mu([x|\ell])\\
  &= &\left(\sum_{j=0}^{\ell-1} Z_{x|j}(r) \mu\{y:\; |x\wedge y|=j\} \right) + Z_{x|\ell}(r)  \mu([x|\ell])\qquad\mbox{(by \eqref{ef-4.11})}\\
  &= &\left(\sum_{j=0}^{\ell-1} Z_{x|j}(r)\left(\mu([x|j])-\mu([x|j+1])\right)\right)+Z_{x|\ell}(r)\mu([x|\ell]).
 \end{eqnarray*}
 This proves \eqref{e-gmu}. The continuity of $x\mapsto G_\mu(x,r)$ follows immediately.
\end{proof}

 The following result plays a key  role in our proof of Theorem \ref{thm-packing}(i).

 \begin{lem}
 \label{lem-1.6}
 Let $\ba\in \R^{md},  \gamma\in (0,1)$ and  $\epsilon>0$. Then for $\mu$-a.e. $x\in \Sigma$, there exists $n(x)\in \N$ such that
 \begin{equation}
 \label{e-muba}
 \pi^\ba_*\mu\left(B_{\gamma^n}(\pi^\ba x)\right)\geq \gamma^{n\epsilon} \mu([x|\ell]) Z_{x|\ell}(\gamma^n) \quad \mbox{ for all }n\geq n(x) \mbox { and } \ell\geq 0. \end{equation}
 \end{lem}
 \begin{proof}
Let $\ba\in \R^{md},  \gamma\in (0,1)$ and  $\epsilon>0$.  Let  $n\in \N$. For $\ell\geq 0$, set
$$ A_{n,\ell}=\{x\in \Sigma: \; \pi^\ba_*\mu\left(B_{\gamma^n}(\pi^\ba x)\right)< \gamma^{n\epsilon} \mu([x|\ell]) Z_{x|\ell}(\gamma^n)\}.$$

We first prove that
\begin{equation}
\label{e-4.7}
 \mu(A_{n,\ell})\leq C'\gamma^{n\epsilon},  \quad  \ell=0, 1,\ldots,
 \end{equation}
where $C'$ is the constant in Lemma \ref{lem-ae}.

To this end, let $\ell\geq0$.  For $\bi\in \Sigma_\ell$,  we let  ${\mathcal F}_{\bi}$  denote the collection of  $Q\in \D_{\gamma^n/\sqrt{d}}$ satisfying the following two conditions:
 \begin{itemize}
 \item[(i)]
 $Q\cap f^\ba_{\bi}(\pi^\ba(\Sigma))\neq \emptyset$.
 \item[(ii)]
 $\pi^\ba_*\mu(Q)<  \gamma^{n\epsilon} \mu([\bi]) Z_{\bi}(\gamma^n).$
 \end{itemize}
 By Lemma \ref{lem-ae}, $\# {\mathcal F}_{\bi}\leq C'/Z_{\bi}(\gamma^n)$. It follows that for $\bi\in \Sigma_\ell$,
 \begin{equation}
 \label{e-4.5}
 \begin{split}
 \pi^\ba_*\mu\left(\bigcup_{Q\in{\mathcal F}_{\bi}}Q \right)& \leq
 \sum_{Q\in{\mathcal F}_{\bi} }  \pi^\ba_*\mu (Q)\\
 &<
  \frac{C'}{Z_{\bi}(\gamma^n)}\cdot \gamma^{n\epsilon} \mu([\bi])Z_{\bi}(\gamma^n)\\
  &=C'\gamma^{n\epsilon} \mu([\bi]).
 \end{split}
 \end{equation}

 We claim that
 \begin{equation}
 \label{e-4.6}
 \pi^\ba(A_{n,\ell}\cap [\bi])\subset \bigcup_{Q\in{\mathcal F}_{\bi}}Q\quad \mbox{ for every } \bi\in \Sigma_\ell.
 \end{equation}
  To see this, let $\bi\in \Sigma_\ell$ and $x\in A_{n,\ell}\cap [\bi]$. Let
$Q(x)$ be the unique element in $\D_{\gamma^n/\sqrt{d}}$ which contains the point $\pi^\ba x$.
  Clearly
 $\pi^\ba x\in Q(x)\cap f^\ba_{\bi}(\pi^\ba(\Sigma))$ since $x\in [\bi]$.  Moreover since ${\rm diam}(Q(x))\leq \gamma^n$, it follows that $Q(x)\subset B_{\gamma^n}(\pi^\ba x)$. Hence
 $$
 \pi^\ba_*\mu(Q(x))\leq \pi^\ba_*\mu(B_{\gamma^n}(\pi^\ba x))< \gamma^{n\epsilon} \mu([\bi]) Z_{\bi}(\gamma^n),$$
 where we used the assumption $x\in A_{n,\ell}$  in the second inequality.
 Therefore $Q(x)\in {\mathcal F}_{\bi}$, and so $\pi^\ba x\in \bigcup_{Q\in{\mathcal F}_{\bi}}Q$ by using the fact that $\pi^\ba x\in Q(x)$. This completes  the proof of \eqref{e-4.6}.

 By \eqref{e-4.6} and  \eqref{e-4.5},  for every $\bi\in \Sigma_\ell$,
 $$
 \mu(A_{n,\ell}\cap [\bi])\leq \pi^\ba_*\mu \left(\pi^\ba(A_{n,\ell}\cap [\bi])\right) \leq \pi^\ba_*\mu\left(\bigcup_{Q\in{\mathcal F}_{\bi}}Q \right)\leq  C'\gamma^{n\epsilon} \mu([\bi]).
 $$
 Summing   $\bi$ over $\Sigma_\ell$ yields that
 $\mu(A_{n,\ell})\leq C'\gamma^{n\epsilon}$. This proves \eqref{e-4.7}.

 Next we prove that
 \begin{equation}
 \label{e-4.8}
 A_{n,\ell}=\emptyset  \quad \mbox{  for all  } \; \ell>\frac{n\log\gamma}{\log \alpha_+}-\frac{\log {\rm diam} (\pi^\ba(\Sigma))}{\log \alpha_+},
 \end{equation}
 where   \begin{equation}
 \label{eta}
 \alpha_+:=\max\{\|T_i\|:\; i=1,\ldots, m\}.
 \end{equation}
For this purpose, let  $\ell>\frac{n\log\gamma}{\log \alpha_+}-\frac{\log {\rm diam} (\pi^\ba(\Sigma))}{\log \alpha_+}$. Then  $\alpha_+^\ell {\rm diam}(\pi^\ba(\Sigma))\leq \gamma^n$.  Let $x\in \Sigma$.   Notice that
 $$
 {\rm diam} (\pi^\ba ([x|\ell]))={\rm diam}(f^\ba_{x|\ell}(\pi^\ba(\Sigma)))\leq \alpha_+^\ell {\rm diam}(\pi^\ba(\Sigma))\leq \gamma^n.
 $$
It follows that  $\pi^\ba([x|\ell])\subset B_{\gamma^n}(\pi^\ba x)$, so
\begin{equation}
\label{ef-4.10}
\mu([x|\ell])\leq \pi^\ba_*\mu\left(\pi^\ba([x|\ell])\right) \leq \pi^\ba_*\mu \left(B_{\gamma^n}(\pi^\ba x)\right).
\end{equation}
 This implies that  $x\not\in A_{n,\ell}$; since if  $x\in A_{n,\ell}$ then by the definition of
 $A_{n,\ell}$,
 $$\pi^\ba_*\mu\left(B_{\gamma^n}(\pi^\ba(x)\right)< \gamma^{n\epsilon} \mu([x|\ell]) Z_{x|\ell}(\gamma^n)\leq \mu([x|\ell]),$$
 which contradicts \eqref{ef-4.10}.
  Since $x$ is arbitrarily taken from $\Sigma$,  it follows that $A_{n,\ell} =\emptyset$. This proves \eqref{e-4.8}.

  Finally,  let $A_n=\bigcup_{\ell=0}^\infty A_{n,\ell}$. Then by \eqref{e-4.7} and \eqref{e-4.8},   $$\mu(A_n)\leq C'\gamma^{n\epsilon} \cdot  \left(\frac{n\log\gamma}{\log \alpha_+}-\frac{\log {\rm diam} (\pi^\ba(\Sigma))}{\log \alpha_+}+1\right). $$
  Therefore $\sum_{n=1}^\infty \mu(A_n)<\infty$. By the Borel-Cantelli lemma,   $$\mu\left(\bigcap_{k=1}^\infty\bigcup_{n=k}^\infty A_n\right)=0.$$
Hence for $\mu$-a.e.~$x\in \Sigma$, there exists $n(x)$ such that $x\not\in A_n$ for every $n\geq n(x)$, from which
\eqref{e-muba} follows.  \end{proof}

 As an application of Lemma \ref{lem-1.6}, we have the following.

  \begin{cor}
  \label{lem-1.7}
  Let $\ba\in \R^{md}$,  $\gamma\in (0,1)$ and  $\epsilon>0$. Then for $\mu$-a.e. $x\in \Sigma$,
  $$
  \pi^\ba_*\mu\left(B_{\gamma^n}(\pi^\ba x)\right)\geq \gamma^{n\epsilon} n^{-1}\left(\frac{\log \gamma}{\log \alpha_+}+2\right)^{-1} G_\mu(x, {\gamma^n}),
  $$
  when $n$ is large enough.
    \end{cor}

 \begin{proof}
   Let $\Gamma$ denote the set of point $x\in \Sigma$ such that there exists $n(x)\in \N$ so that
 \begin{equation}
 \label{ef-4}
 \pi^\ba_*\mu\left(B_{\gamma^n}(\pi^\ba x)\right)\geq \gamma^{n\epsilon} \mu([x|\ell]) Z_{x|\ell}(\gamma^n)\quad
 \mbox{ for all } n\geq n(x)\mbox{ and }\ell\geq 0.
 \end{equation}
 Then  by Lemma \ref{lem-1.6}, $\mu(\Gamma)=1$.

 Let $\alpha_+$ be defined as in \eqref{eta}. Fix $x\in \Gamma$ and let $n\geq n(x)$. Let $\ell$ be the smallest integer so that
 $\alpha_+^\ell\leq \gamma^n$.  Then
\begin{equation}
\label{ef-4.11'}
 \ell\leq \frac{n\log \gamma}{\log \alpha_+}+1.
 \end{equation}
 By Lemma \ref{lem-gmu} (in which we take $r=\gamma^n$),
\begin{eqnarray*}
  G_\mu(x, {\gamma^n})    &\leq &\sum_{j=0}^{\ell} Z_{x|j}(\gamma^n) \mu([x|j]) \\
  &\leq & (\ell+1)  \gamma^{-n\epsilon}
 \pi^\ba_*\mu\left(B_{\gamma^n}(\pi^\ba x)\right) \qquad(\mbox{by \eqref{ef-4}})\\
 &\leq & n\left(\frac{\log \gamma}{\log \alpha_+}+2\right) \gamma^{-\epsilon n}  \pi^\ba_*\mu\left(B_{\gamma^n}(\pi^\ba x)\right)\qquad(\mbox{by \eqref{ef-4.11'}}),
\end{eqnarray*}
 which yields the desired inequality.
 \end{proof}

 \begin{proof}[Proof of Theorem \ref{thm-packing}(i)]
  It follows directly from Corollary \ref{lem-1.7}.
 \end{proof}

\subsection{Proof of part (ii) of Theorem \ref{thm-packing}}

Let us begin with the following.
\begin{lem}
\label{lem-zxn}
Let $x\in \Sigma$, $r>0$ and $n\geq 0$. Then
\begin{align}
 Z_{x|n}(r)&=\min\left\{\frac{r^k}{\phi^k(T_{x|n})}:\; k=0,1,\ldots, d\right\}\quad  \mbox{ and } \label{e-aa1} \\
 Z_{x|n}(r)&\leq \frac{r^t}{\phi^t(T_{x|n})}\quad \mbox{ for all }t\in [0,d]. \label{e-aa2}
\end{align}
\end{lem}
\begin{proof}
Write for brevity that $\alpha_k=\alpha_k(T_{x|n})$ for $k=1,\ldots,d$.  By \eqref{ef-4.2},
\begin{align*}
Z_{x|n}(r)=\left\{
\begin{array}{ll}
1 & \mbox{ if } r\geq \alpha_1,\\
\frac{r^k}{\alpha_1\cdots \alpha_k }&  \mbox{  if  }\alpha_{k+1}\leq r<\alpha_k \mbox { for some }k\in \{1,\ldots, d-1\},\\
\frac{r^d}{\alpha_1\cdots \alpha_d }&  \mbox{  if  } r<\alpha_d.
\end{array}
\right.
\end{align*}
Now \eqref{e-aa1} follows from the above equality by a routine check.  To see \eqref{e-aa2}, we only need to consider the case when $t$ is a non-integer in $(0,d)$. Let $k$ be the unique integer so that $t\in (k,k+1)$. Clearly $0\leq k\leq d-1$. Now $t=pk+(1-p)(k+1)$ for some  $p\in (0,1)$.  By \eqref{e-aa1},
$$
Z_{x|n}(r)\leq \left(\frac{r^k}{\alpha_1\cdots \alpha_k }\right)^p\left(\frac{r^{k+1}}{\alpha_1\cdots \alpha_{k+1} }\right)^{1-p}=\frac{r^t}{\alpha_1\cdots \alpha_k\alpha_{k+1}^{1-p}}=\frac{r^t}{\phi^t(T_{x|n})},
$$
where we use the fact that $1-p=t-k$ in the last equality.
\end{proof}

 Recall that for $\rho>0$,  $B_\rho$ denotes the closed ball in $\R^{md}$ of radius $\rho$ centred at the origin.
 To prove Theorem \ref{thm-packing}(ii), we need  the following.
 \begin{lem} [{\cite[Lemma 5.2]{JPS07}}]
 \label{lem-tran}
 Assume that $\|T_i\|<1/2$ for $1\leq i\leq m$. Let $\rho>0$.    There exists $C =C(\rho, T_1,\ldots, T_m)> 0$
such that for all $x,y\in \Sigma$ and $r>0$,
 \begin{equation*}
 \label{e-trans}
 \mathcal L^{md} \{\ba\in B_\rho:\; |\pi^\ba x-\pi^\ba y |\leq r\}\leq C\cdot Z_{x \wedge y}(r).
 \end{equation*}
 \end{lem}
 \medskip

 \begin{pro}
 \label{pro-1.3}
  Assume that $\|T_i\|<1/2$ for $1\leq i\leq m$. Let $x\in \Sigma$ and $\rho>0$. Then
 $$
  \overline{\dim}_{\rm loc}(\pi^\ba_*\mu, \pi^\ba x)\geq D(\mu,x)
  $$  for $\mathcal L^{md}$-a.e.~$\ba\in B_\rho$.
\end{pro}

 \begin{proof}
We adopt an idea from the proof of \cite[Lemma 4(a)]{FaHo97}.
Let $x\in \Sigma$ and $r>0$. Applying Fubini's theorem and Lemma \ref{lem-tran},
\begin{eqnarray*}
\int_{B_\rho}\pi^\ba_*\mu\left(B_r (\pi^\ba x)\right) d\ba&=&\int_{B_\rho} \int_{\R^d} {\bf 1}_{\{z:\; |z-\pi^\ba x |\leq r\}}\; d\pi^\ba_*\mu (z) d\ba\\
& =&
 \int_{B_\rho}\int_{\Sigma} {\bf 1}_{\{y:\;|\pi^\ba y-\pi^\ba x |\leq r\}}\; d\mu(y)  d\ba \\
 & =&
 \int_\Sigma \int_{B_\rho} {\bf 1}_{\{\ba:\;|\pi^\ba y-\pi^\ba x |\leq r\}}\;  d\ba d\mu(y)  \\
 &=& \int_\Sigma  \mathcal L^{md} \{\ba\in {B_\rho}:\; |\pi^\ba x-\pi^\ba y |\leq r\} \; d\mu(y)\\
& \leq & C\int_\Sigma   Z_{x\wedge y}(r)\; d\mu(y)\\
&=& C\cdot G_\mu(x,r),
\end{eqnarray*}
 where $C$ is the constant given in Lemma \ref{lem-tran}.
 Hence by Fatou's lemma,  for any $t\in \R$,
 \begin{eqnarray}
 \int_{B_\rho} \liminf_{r\to 0} r^{-t} \pi^\ba_*\mu\left(B_r (\pi^\ba x)\right) \;d\ba& \leq & \liminf_{r\to 0} \int_{B_\rho} r^{-t} \pi^\ba_*\mu\left(B_r (\pi^\ba x)\right)\; d\ba \nonumber
 \\
 &\leq& C \liminf_{r\to 0} r^{-t} G_\mu(x,r).  \label{e-D2}
 \end{eqnarray}

  Next we assume that $t<D(\mu,x)$. By the definition of $D(\mu,x)$ (see \eqref{e-D1}), we get $\liminf_{r\to 0} r^{-t} G_\mu(x,r)=0$. Combining this with \eqref{e-D2} yields that
  $$
  \int_{B_\rho} \liminf_{r\to 0} r^{-t} \pi^\ba_*\mu\left(B_r (\pi^\ba x)\right) \;d\ba=0,
  $$
 which implies that
  $
  \liminf_{r\to 0} r^{-t} \pi^\ba_*\mu\left(B_r (\pi^\ba x)\right)=0$  for $\mathcal L^{md}$-a.e.~$\ba\in {B_\rho}$.
   Hence
  $$
  \overline{\dim}_{\rm loc}(\pi^\ba_*\mu, \pi^\ba x)\geq t
  $$
  for $\mathcal L^{md}$-a.e.~$\ba\in {B_\rho}$.  This concludes the proposition by letting $t\to D(\mu,x)$.
  \end{proof}

 \begin{proof}[Proof of Theorem \ref{thm-packing}(ii)]
 Notice that  the mapping $(\ba, x)\mapsto \pi^\ba x$ is continuous. It follows that for every $r>0$, the mapping
 $(\ba,x)\mapsto \pi^\ba_*\mu\left(B_r(\pi^\ba x)\right)$ is upper semi-continuous. Meanwhile by Lemma \ref{lem-gmu}
 the mapping $x\mapsto G_\mu(x,r)$
 is continuous.  Since
 $$
 \overline{\dim}_{\rm loc}(\pi^\ba_*\mu, \pi^\ba (x))=\limsup_{n\to +\infty} \frac{\log \pi^\ba_*\mu\left(B_{2^{-n}}(\pi^\ba x)\right)}{-n\log 2},\quad
 D(\mu,x)=\limsup_{n\to +\infty} \frac{\log G_\mu(x, 2^{-n})}{-n\log 2},
 $$
 it follows that the mappings
 $$
 (\ba, x)\mapsto \overline{\dim}_{\rm loc}(\pi^\ba_*\mu, \pi^\ba x), \quad x\mapsto D(\mu,x)
 $$
 are both Borel measurable.  Hence for any  $\rho>0$,
 the set
 $$
 \Theta_\rho:=\{(\ba, x)\in B_\rho\times \Sigma:\;   \overline{\dim}_{\rm loc}(\pi^\ba_*\mu, \pi^\ba x)\geq D(\mu,x)\}
 $$
 is Borel measurable.  By  Fubini's theorem and Proposition \ref{pro-1.3},
 \begin{align*}
 (\mathcal L^{md}\times \mu) (\Theta_\rho)&=\int_{\Sigma} \mathcal L^{md} \{\ba\in B_\rho:\;\overline{\dim}_{\rm loc}(\pi^\ba_*\mu, \pi^\ba x)\geq D(\mu,x)\}\;d\mu(x)\\
 &= \int_\Sigma\mathcal L^{md}(B_\rho)\;d\mu(x)\\
 &=\mathcal L^{md}(B_\rho).
 \end{align*}
 So applying  Fubini's theorem again,  $$
 \mathcal L^{md}(B_\rho)=(\mathcal L^{md}\times \mu) (\Theta_\rho)=\int_{B_\rho} \mu \left\{x\in \Sigma:\;\overline{\dim}_{\rm loc}(\pi^\ba_*\mu, \pi^\ba x)\geq D(\mu,x)\right\}\;d\ba,
 $$
 which implies that for $\mathcal L^{md}$-a.e.~$\ba\in B_\rho$,
 $$
  \overline{\dim}_{\rm loc}(\pi^\ba_*\mu, \pi^\ba x)\geq D(\mu,x) \quad \mbox{ for $\mu$-a.e.~$x\in \Sigma$}.
 $$
 Combining this with  part (i) of the theorem yields that  for $\mathcal L^{md}$-a.e.~$\ba\in B_\rho$,  $$
  \overline{\dim}_{\rm loc}(\pi^\ba_*\mu, \pi^\ba x)= D(\mu,x) \quad \mbox{ for $\mu$-a.e.~$x\in \Sigma$},
 $$
 from which part (ii) follows.
 \end{proof}

 \section{Exact dimensionality of projected measures}
 \label{S4}

 In this section, we prove Theorems \ref{thm-5.1}-\ref{thm-5.2}. We also construct an example (see Example \ref{ex-1}) to show that for every integer $d\geq 2$ and  $s\in \{1,\ldots, d-1\}$,   \eqref{e-equv} does not always imply \eqref{e-equv1} in the general affine setting.

 We begin with the proof of Theorem \ref{thm-5.1}.

 \begin{proof}[Proof of Theorem \ref{thm-5.1}]
By Theorem \ref{thm-lower-upper},
\begin{equation}
\label{e-tt1}
\begin{split}
\underline{\dim}_{\rm H}\pi^\ba_*\mu&=\min\{\underline{S}(\mu), d\},\quad \overline{\dim}_{\rm H}\pi^\ba_*\mu=\min\{\overline{S}(\mu), d\},\\
\underline{\dim}_{\rm P}\pi^\ba_*\mu&=\underline{D}(\mu),\qquad\quad\quad \;\;\overline{\dim}_{\rm P}\pi^\ba_*\mu=\overline{D}(\mu).
\end{split}
\end{equation}
for $\mathcal L^{md}$-a.e.~$\ba\in \R^{md}$.  Hence $\pi^\ba_*\mu$ is exact dimensional  for $\mathcal L^{md}$-a.e.~$\ba\in \R^{md}$ if and only if
\begin{equation}
\label{e-gammad}
\min\{\underline{S}(\mu), d\}=\overline{D}(\mu).
\end{equation}

Notice that  the upper packing dimension of a measure on $\R^d$ is greater than or equal to its lower Hausdorff dimension, but does not exceed $d$, so  by \eqref{e-tt1},
$$\min\{\underline{S}(\mu), d\}\leq \overline{D}(\mu)\leq d.$$  Thus \eqref{e-gammad} holds if and only if one of the following conditions is satisfied:
\begin{itemize}
\item[(a)] $\underline{S}(\mu)\geq d$;
\item[(b)] $\underline{S}(\mu)=\overline{D}(\mu)=s$ for some $s\in [0,d)$.
\end{itemize}
Clearly, the condition (a) is equivalent to (i). Meanwhile, (b) is equivalent to (ii). To see it, notice that (ii) clearly implies (b). For the converse part,  from \eqref{e-tt1} we see that $\min\{\underline{S}(\mu), d\}\leq \min\{\overline{S}(\mu), d\}\leq \overline{D}(\mu)$. So (b) implies $\underline{S}(\mu)=\overline{S}(\mu)=\overline{D}(\mu)=s$, which is equivalent to (ii).  \end{proof}

Next we turn to the proof of Theorem \ref{thm-5.2}. As it is
 quite lengthy and delicate,  the reader may skip it in a first reading.

 We first prove several lemmas.
 \begin{lem}
 \label{lem-5.4}
 Let $\mu$ be a Borel probability measure on $\Sigma$ and $\delta\in (0,1)$. Then
  \begin{equation}
 \label{e-covering}
 \limsup_{n\to \infty}\frac{1}{n}  \log \left(
 \frac{\mu([x_1\ldots x_{\lceil(1-\delta)n\rceil}]) }{\mu([x_1\ldots x_n ])}  \right)\leq \delta \log m \quad \mbox{ for $\mu$-a.e. $x\in \Sigma$, }
 \end{equation}
 where $\lceil a \rceil$ represents the least integer not smaller than $a$.
 \end{lem}
 \begin{proof}
 Let $\epsilon>\delta \log m$. For $n\in \N$, let $\Theta_n$ denote the collection of $I\in \Sigma_n$ so that
 \begin{equation*}
 \label{e-covering1}
 \mu([I])\leq e^{-n\epsilon} \mu([I|\lceil (1-\delta)n\rceil]).
 \end{equation*}
 Summing over $I\in \Theta_n$ and noticing that for each $J\in \Sigma_{\lceil (1-\delta)n\rceil}$ there are at most
 $m^{\delta n}$ many elements $I\in \Sigma_n$ with $I|\lceil (1-\delta)n\rceil=J$,  we have
 $$
 \mu\left( \bigcup_{I\in \Theta_n} [I] \right)\leq e^{-n\epsilon} \sum_{I\in \Theta_n}\mu\Big(\big[I|\lceil (1-\delta)n\rceil\big]\Big) \leq  e^{-n\epsilon} m^{\delta n}=e^{-n (\epsilon-\delta \log m)}.
 $$
 Hence
 $$
 \sum_{n=1}^\infty \mu\left( \bigcup_{I\in \Theta_n} [I] \right)<\infty.
 $$
 By the Borel-Cantelli lemma,
 $
 \mu\left(\bigcap_{n=1}^\infty\bigcup_{k=n}^\infty \bigcup_{I\in \Theta_k} [I]  \right)=0.
 $
It follows that  $$
  \limsup_{n\to \infty}\frac{1}{n}  \log \left(
 \frac{\mu([x_1\ldots x_{\lceil(1-\delta)n\rceil}]) }{\mu([x_1\ldots x_n ])}  \right)\leq \epsilon\quad \mbox{ for $\mu$-a.e. $x\in \Sigma$.}
 $$
 Letting $\epsilon\to \delta\log m$ gives   \eqref{e-covering}.
  \end{proof}

\begin{lem}
\label{lem-5.6}
Write
\begin{equation}
\label{e-eta}
\alpha_+=\max\{\|T_i\|:\; i=1,\ldots, m\},\quad \alpha_-=\min\{\alpha_d(T_i):\; i=1,\ldots, m\}.
\end{equation}
 Then for all $I\in \Sigma_*$, $i\in \{1,\ldots, m\}$ and $k\in \{1,\ldots, d\}$,
\begin{equation}
\label{e-geta}
\alpha_-\leq \frac{\alpha_k(T_{Ii})}{\alpha_k(T_I)}\leq \alpha_+.
\end{equation}
\end{lem}
\begin{proof}
This follows from the general fact that
$$
\alpha_k(A)\alpha_d(B)\leq \alpha_k(AB)\leq \alpha_k(A)\alpha_1(B)
$$
for all $A, B\in {\rm GL}_d(\R)$ and $k\in \{1,\ldots, d\}$. The reader is referred to \cite[Theorem 3.3.16(d)]{HoJo91} for the second inequality. To see the first inequality, simply notice that
$$
\alpha_k(A)=\alpha_k(ABB^{-1})\leq \alpha_k(AB)\alpha_1(B^{-1})=\alpha_k(AB)/\alpha_d(B)
$$
by using the second inequality and the identity $\alpha_1(B^{-1})=1/\alpha_d(B)$.
\end{proof}

Recall that the quantities $S_n(\mu,x)$ and $G_\mu(x,r)$ are defined in \eqref{e-2.2} and \eqref{e-Gmur}, respectively.
\begin{lem}
\label{lem-5.5}
Let $\mu\in \mathcal P(\Sigma)$ and $x\in \Sigma$. Then the following statements hold.
\begin{itemize}
\item[(i)] If $\limsup_{n\to \infty} S_n(\mu,x)=t<d$, then
$$
\limsup_{r\to 0} \frac{\log G_\mu(x,r)}{\log r}\leq t.
$$
\item[(ii)] If $\liminf_{n\to \infty}S_n(\mu,x)=s<d$, then
$$
\liminf_{r\to 0} \frac{\log G_\mu(x,r)}{\log r}\geq s.
$$
\end{itemize}
\end{lem}

\begin{proof}
We first prove (i). Suppose that $\limsup_{n\to \infty} S_n(\mu,x)=t<d$. Let $k=\lfloor t \rfloor$. Let $\epsilon>0$ be small enough so that $\lfloor t +\epsilon\rfloor=k$. Observe that for $r>0$,
\begin{align*}
G_\mu(x,r)&=\int Z_{x\wedge y}d\mu(y)\\
&=\left(\sum_{j=0}^\infty Z_{x|j}(r) (\mu([x|j])-\mu([x|j+1]))\right)+\mu(\{x\}).
\end{align*}
Since $Z_{x|j}(r)\leq Z_{x|j+1}(r)$ and $Z_{x|j}(r)\leq 1$ for every $j\geq 0$,  it follows that for  $n\geq 0$,
\begin{equation}\label{e-GZ}
\begin{split}
G_\mu(x,r)&\geq \left(\sum_{j=n}^\infty Z_{x|j}(r) (\mu([x|j])-\mu([x|j+1]))\right)+\mu(\{x\})\\
&\geq \left(\sum_{j=n}^\infty Z_{x|n}(r) (\mu([x|j])-\mu([x|j+1]))\right)+Z_{x|n}(r)\mu(\{x\})\\
&= Z_{x|n}(r)\mu([x|n]).
\end{split}
\end{equation}

Since $\limsup_{n\to \infty} S_n(\mu,x)=t$, there exists $n_0\in \N$ such that $S_n(\mu,x)<t+\epsilon$ for all $n\geq n_0$, which implies
\begin{equation}
\label{e-GZ1}
\mu([x|n])\geq \phi^{t+\epsilon}(T_{x|n})\quad \mbox{ for all }\;n\geq n_0.
\end{equation}
For $n\geq n_0$, taking $r=\alpha_{k+1}(T_{x|n})$ in \eqref{e-GZ} gives
\begin{align*}
G_\mu(x,\alpha_{k+1}(T_{x|n}))&\geq Z_{x|n}(\alpha_{k+1}(T_{x|n}))\mu([x|n])\\
&\geq Z_{x|n}(\alpha_{k+1}(T_{x|n}))\phi^{t+\epsilon}(T_{x|n})\qquad\mbox{ (by \eqref{e-GZ1})}\\
&=\frac{\alpha_{k+1}(T_{x|n})^k}{\phi^k(T_{x|n})}\cdot \phi^{t+\epsilon}(T_{x|n})\\
&=\alpha_{k+1}(T_{x|n})^{t+\epsilon},
\end{align*}
so
$$\frac{\log G_\mu(x,\alpha_{k+1}(T_{x|n})) } {\log \alpha_{k+1}(T_{x|n})}\leq t+\epsilon.
$$
Since the sequence $\left(\alpha_{k+1}(T_{x|n})\right)_{n=n_0}^\infty$ is monotone decreasing with
$$
\frac{1}{\alpha_+}\leq \frac{\alpha_{k+1}(T_{x|n})}{\alpha_{k+1}(T_{x|n+1})}\leq \frac{1}{\alpha_-}
$$
(see \eqref{e-geta}), it follows that
$$
\limsup_{r\to 0}\frac{\log G_\mu(x,r)}{\log r}=\limsup_{n\to \infty} \frac{\log  G_\mu(x,\alpha_{k+1}(T_{x|n})) } {\log \alpha_{k+1}(T_{x|n})}\leq t+\epsilon.
$$
Letting $\epsilon\to 0$  proves (i).

Now we turn to the proof of (ii).  Suppose that $\liminf_{n\to \infty}S_n(\mu,x)=s<d$. To avoid the triviality we may assume that $s>0$.  Let $\epsilon>0$ be small enough so that
$$\lfloor s-\epsilon\rfloor<s-2\epsilon.$$
Set $k=\lfloor s-\epsilon\rfloor$. Since  $\liminf_{n\to \infty}S_n(\mu,x)=s$, there exists $n_0$ such that
\begin{equation}\label{e-GZ2}
\mu([x|n])=\phi^{S_n(\mu,x)}(T_{x|n})<\phi^{s-\epsilon}(T_{x|n})\quad\mbox{ for all }n\geq n_0.
\end{equation}
This implies, in particular, that $\mu(\{x\})=\lim_{n\to \infty}\mu([x|n])=0$. Hence
\begin{align*}
G_\mu(x,r)&=\sum_{n=0}^\infty Z_{x|n}(r)(\mu([x|n])-\mu([x|n+1]))\\
&\leq \sum_{n=0}^{n_0} Z_{x|n}(r)(\mu([x|n])-\mu([x|n+1]))+\sum_{n=n_0+1}^\infty Z_{x|n}(r)\mu([x|n])\\
&\leq \sum_{n=0}^{n_0} Z_{x|n_0}(r)(\mu([x|n])-\mu([x|n+1]))+\sum_{n=n_0+1}^\infty Z_{x|n}(r)\mu([x|n])\\
&\leq  Z_{x|n_0}(r)+\sum_{n=n_0+1}^\infty Z_{x|n}(r)\mu([x|n])\\
&\leq  \frac{r^d}{\phi^{d}(T_{x|n_0})}+\sum_{n=n_0+1}^\infty \frac{r^{s-2\epsilon}}{\phi^{s-2\epsilon}(T_{x|n})}\cdot\phi^{s-\epsilon}(T_{x|n})\qquad\mbox{(by \eqref{e-aa2}, \eqref{e-GZ2})}\\
&=\frac{r^d}{\phi^{d}(T_{x|n_0})}+r^{s-2\epsilon} \sum_{n=n_0+1}^\infty \alpha_{k+1}(T_{x|n})^\epsilon \\
&\leq C r^{s-2\epsilon}
\end{align*}
for  some $C>0$ independent of $r$. It follows that $$\liminf_{r\to 0}\frac{\log G_\mu(x,r)}{\log r}\geq s-2\epsilon.$$ Letting $\epsilon\to 0$ gives $\liminf_{r\to 0}\log G_\mu(x,r)/\log r\geq s$.
\end{proof}

Now we are ready to prove Theorem \ref{thm-5.2}.

\begin{proof}[Proof of Theorem \ref{thm-5.2}] We first prove (i).
For $s\in [0,d)$ and $x\in \Sigma$, by Lemma \ref{lem-5.5} we have the implication
$$
\lim_{n\to \infty} S_n(\mu,x)=s \Longrightarrow \lim_{r\to 0}\frac{\log G_\mu(x,r)}{\log r}=s.$$
Hence we have the implication \eqref{e-equv1} $\Longrightarrow$ \eqref{e-equv} whenever $s\in [0,d)$.

Next we prove (ii). According to (i) we only need to prove the  direction \eqref{e-equv} $\Longrightarrow$ \eqref{e-equv1} whenever $s\in [0,d)\backslash \N$.  To this end, we fix $s\in [0,d)\backslash \N$. Suppose \eqref{e-equv} holds. Below we prove \eqref{e-equv1} by contradiction.

Suppose on the contrary that \eqref{e-equv1} does not hold. Then there exists a Borel set $A\subset \Sigma$ with $\mu(A)>0$ and $\tau\in (0,1)$ such that
\begin{equation}
\label{e-lar}
\limsup_{n\to \infty}S_n(\mu,x)>s+\tau \quad \mbox{ for all }x\in A.
\end{equation}
By \eqref{e-equv} and Lemma \ref{lem-5.4},  removing a subset of zero $\mu$-measure from $A$ if necessary, we may assume that
\begin{equation}
\label{e-deltan}
\limsup_{n\to \infty} \frac{1}{n}\log \left( \frac{\mu\big([x|\lceil(1-q)n\rceil]\big)}{\mu([x|n])}\right)\leq q \log m \quad \mbox{ for all $x\in A$ and $q\in {\Bbb Q}\cap (0,1)$},
\end{equation}
and
\begin{equation}
\label{e-5.12}
\liminf_{n\to \infty} S_n(\mu,x)=s=\limsup_{r\to 0} \frac{\log G_\mu(x,r)}{\log r}\quad  \mbox{ for all $x\in A$}.
\end{equation}

Now fix a point $x\in A$. In what follows we derive a contradiction by considering the cases $s\in (0,d)\backslash \N$ and $s=0$ separately.

{\sl \bf Case 1}. $s\in (0,d)\backslash \N$.

Let $k=\lfloor s \rfloor$. Then $0\leq k\leq d-1$ and $k<s<k+1$. Replacing $\tau$ by a smaller positive number if necessary, we may assume that $k<s\pm\tau<k+1$. Choose $\delta\in {\Bbb Q}\cap (0,1)$ so that
\begin{equation}
\label{e-delta}
\delta\leq \frac{\tau \log (1/\alpha_+)}{4\log m+ 2k \log (1/\alpha_-)},
\end{equation}
where $\alpha_+$ and $\alpha_-$ are defined as in \eqref{e-eta}.  Then we pick  $\epsilon>0$ small enough so that
\begin{equation}
\label{e-epsilon}
\epsilon<\min \left\{\frac{\tau}{2},\; \frac{(\tau-\epsilon)\delta\log (1/\alpha_+)}{2\log (1/\alpha_-)} \right\}.
\end{equation}
Below we will show that
\begin{equation}
\label{e-claim}
\limsup_{r\to \infty} \frac{\log G_\mu(x,r)}{\log r}\geq s+\frac{\epsilon}{2},
\end{equation}
which clearly contradicts \eqref{e-5.12}.

By \eqref{e-5.12} and \eqref{e-deltan}, there exists $n_0\in \N$ such that
\begin{equation}
\label{e-5.16}
\mu([x|n])\leq \phi^{s-\epsilon}(T_{x|n})\quad \mbox{ for all }n\geq n_0
\end{equation}
and
\begin{equation}
\label{e-5.16'}
\mu\big([x| (1-\delta)n]\big)\leq m^{2\delta n}\mu([x|n])\quad \mbox{ for all }n\geq n_0.
\end{equation}
In a customary abuse of notation, we write $(1-\delta)n$ instead of $\lceil (1-\delta)n \rceil$ in \eqref{e-5.16'}.

Similarly by \eqref{e-lar}, there exists an arbitrarily large $N\in \N$ such that
\begin{equation}
\label{e-5.17}
\mu([x|N])\leq \phi^{s+\tau}(T_{x|N}).
\end{equation}
Set
$$
r_N=\alpha_{k+1}(T_{x|N}).
$$
We may require that $N$ is so large that $(1-\delta)N > n_0$ and $(1/\alpha_+)^{N\epsilon/2}>N$, where the second inequality implies that
\begin{equation}
\label{e-5.19'}
N\leq r_N^{-\epsilon/2}.
\end{equation}

By \eqref{e-epsilon}, $\alpha_+^{\delta(\tau-\epsilon)}\leq (\alpha_-)^{2\epsilon}$. This together with \eqref{e-geta} yields that
\begin{equation}
\label{e-5.18}
\alpha_+^{\delta(\tau-\epsilon)N}\leq (\alpha_-)^{2\epsilon N}\leq \alpha_{k+1}(T_{x|N})^{2\epsilon}=r_N^{2\epsilon}.
\end{equation}
Similarly by \eqref{e-delta}, $(\alpha_-)^{-k\delta} m^{2\delta}\leq \alpha_+^{-\tau/2}$. This together with \eqref{e-geta} yields that
\begin{equation}
\label{e-5.18'}
(\alpha_-)^{-k\delta N} m^{2\delta N}\leq \alpha_+^{-N \tau /2}\leq r_N^{-\tau/2}.
\end{equation}

Now let us estimate $G_\mu(x, r_N)$. By \eqref{e-5.16}, $$\mu(\{x\})=\lim_{n\to 0} \mu([x|n])=0.$$
 It follows that
\begin{align*}
G_\mu(x, r_N)&=\sum_{n=0}^\infty Z_{x|n}(r_N)(\mu([x|n])-\mu([x|n+1]))\\
&\leq \sum_{n=0}^\infty Z_{x|n}(r_N)\mu([x|n])\\
&\leq \sum_{n=0}^{n_0}Z_{x|n}(r_N)\mu([x|n])+\sum_{n=n_0}^{ (1-\delta)N}Z_{x|n}(r_N)\mu([x|n])\\
&\qquad +\sum_{n=(1-\delta)N}^{(1+\delta)N}Z_{x|n}(r_N)\mu([x|n])+\sum_{n=(1+\delta)N}^\infty Z_{x|n}(r_N)\mu([x|n])\\
&=: (I)+(II)+(III)+(IV),
\end{align*}
where, and in what follows,  we write $(1\pm \delta)N$ instead of $\lceil (1\pm \delta)N \rceil$ in an abuse of notation. Below we estimate these sub-sums separately.

We start from the estimation of (I).
\begin{align*}
(I):&=\sum_{n=0}^{n_0}Z_{x|n}(r_N)\mu([x|n])\\
&\leq (n_0+1) Z_{x|n_0}(r_N)\leq (n_0+1) \frac{r_N^d}{\phi^d(T_{x|n_0})} \leq \frac{(n_0+1) }{\phi^d(T_{x|n_0})}\cdot r_N^{s+\epsilon}.
\end{align*}

Next we estimate (II).
\begin{align*}
(II):&=\sum_{n=n_0}^{ (1-\delta)N}Z_{x|n}(r_N)\mu([x|n])\\
&\leq \sum_{n=n_0}^{(1-\delta)N} \frac{ r_N^{s+\tau}}{\phi^{s+\tau}(T_{x|n})}\cdot \phi^{s-\epsilon}(T_{x|n}) \quad \mbox{ (by \eqref{e-aa2} and \eqref{e-5.16})}\\
&=\sum_{n=n_0}^{ (1-\delta)N} r_N^{s+\tau}\alpha_{k+1}(T_{x|n})^{-\tau-\epsilon}
\qquad \mbox{(since $k<s+\tau,s-\epsilon<k+1$)}\\
&=\sum_{n=n_0}^{(1-\delta)N }r_N^{s-\epsilon}\left(\frac{\alpha_{k+1}(T_{x|N})}{\alpha_{k+1}(T_{x|n})}\right)^{\tau+\epsilon}\\
&\leq Nr_N^{s-\epsilon} (\alpha_+^{\delta N-1})^{\tau+\epsilon}\qquad \mbox{ (by \eqref{e-geta})}\\
&\leq N\alpha_+^{-\tau-\epsilon} r_N^{s-\epsilon} \alpha_+^{\delta N(\tau-\epsilon)}\\
&\leq \alpha_+^{-\tau-\epsilon} r_N^{s+\epsilon/2}\qquad \qquad \mbox{ (by \eqref{e-5.19'} and \eqref{e-5.18})}.
\end{align*}

Then we turn to the estimation of (III).
\begin{align*}
(III):&=\sum_{n= (1-\delta)N}^{ (1+\delta)N}Z_{x|n}(r_N)\mu([x|n])\\
&\leq (1+\delta)N Z_{x| (1+\delta)N}(r_N)\mu\big([x|(1-\delta)N]\big)\\
&\leq (1+\delta)N\cdot \frac{r_N^k}{\phi^k(T_{x|(1+\delta)N})}\cdot m^{2\delta N}\mu([x|N]) \quad \mbox{ (by \eqref{e-aa2} and \eqref{e-5.16'})}\\
&\leq (1+\delta)N\cdot \frac{r_N^k}{\phi^k(T_{x|(1+\delta)N})}\cdot m^{2\delta N}\phi^{s+\tau}(T_{x|N})
\qquad\mbox{ (by \eqref{e-5.17})}\\
&\leq (1+\delta)N\cdot \frac{r_N^k}{\phi^k(T_{x|N})}\cdot(\alpha_-)^{-k(\delta N+1)}\cdot m^{2\delta N}\phi^{s+\tau}(T_{x|N})\qquad \mbox{ (by \eqref{e-geta})}\\
&= (1+\delta)N\cdot r_N^{s+\tau}\cdot(\alpha_-)^{-k(\delta N+1)}\cdot m^{2\delta N}\\
&\leq (1+\delta)N(\alpha_-)^{-k} r_N^{s+\tau/2} \qquad \mbox{ (by \eqref{e-5.18'})}\\
&\leq (1+\delta)(\alpha_-)^{-k} r_N^{s+\epsilon/2}\qquad \mbox{ (by using $\epsilon\leq \tau/2$ and \eqref{e-5.19'})}.
\end{align*}

Finally we estimate (IV).
\begin{align*}
(IV):&=\sum_{n= (1+\delta)N}^{\infty}Z_{x|n}(r_N)\mu([x|n])\\
&\leq \sum_{n= (1+\delta)N}^{\infty} \frac{r_N^{s-\tau}}{\phi^{s-\tau}(T_{x|n})}\cdot \phi^{s-\epsilon}(T_{x|n}) \quad \mbox{ (by \eqref{e-aa2} and \eqref{e-5.16})}\\
&= \sum_{n=(1+\delta)N}^{\infty} r_N^{s-\tau}\alpha_{k+1}(T_{x|n})^{\tau-\epsilon}\qquad \mbox{(since $k<s-\tau,s-\epsilon<k+1$)}\\
&= \sum_{n= (1+\delta)N}^{\infty} r_N^{s-\epsilon}\left(\frac{\alpha_{k+1}(T_{x|n})}{\alpha_{k+1}(T_{x|N})}\right)^{\tau-\epsilon}\\
&= \sum_{n= (1+\delta)N}^{\infty} r_N^{s-\epsilon} \alpha_+^{(n-N)(\tau-\epsilon)} \quad \mbox{ (by  \eqref{e-geta})}\\
&\leq r_N^{s-\epsilon} \alpha_+^{\delta N (\tau-\epsilon)} (1-\alpha_+^{\tau-\epsilon})^{-1}\\
&\leq (1-\alpha_+^{\tau-\epsilon})^{-1} r_N^{s+\epsilon} \quad \mbox{ (by  \eqref{e-5.18})}.
\end{align*}

Summing up these estimates, we see that
$$
G_{\mu}(x, r_N)\leq C r_N^{s+\epsilon/2},
$$
where $C>0$ is a constant independent of $N$. Since $N$ can be taken arbitrarily large, we obtain \eqref{e-claim}.
\medskip

{\sl\bf Case 2}. $s=0$.

In this case, we choose $\delta\in {\Bbb Q}\cap (0,1)$ such that
\begin{equation}
\label{e-delta1}
\delta<\frac{\tau\log (1/\alpha_+)}{4 \log m},
\end{equation}
and take a small $\epsilon>0$ such that
\begin{equation}
\label{e-epsilon1}
\epsilon<\min\left\{\frac{\tau}{2}, \; \frac{\delta\log \alpha_+}{\log \alpha_-}\right\}.
\end{equation}
In what follows we will show that
\begin{equation}
\label{e-claim1}
\limsup_{r\to 0}\frac{\log G_\mu(x,r)}{\log r}\geq \epsilon/2,
\end{equation}
which contradicts \eqref{e-5.12}.

By \eqref{e-lar} and \eqref{e-deltan} (in which $s=0$), we may find arbitrarily large $N\in \N$ such that
\begin{equation}
\label{e-5.25}
\mu([x|N])\leq \phi^\tau(T_{x|N})=\alpha_1(T_{x|N})^\tau,
\end{equation}
and
\begin{equation}
\label{e-5.26}
\mu([x|(1-\delta)N])\leq m^{2\delta N}\mu([x|N]).
\end{equation}

By \eqref{e-delta1} and \eqref{e-geta},
$m^{2\delta N}\leq \alpha_+^{-N \tau/2}\leq \alpha_1(T_{x|N})^{-\tau/2}$. So according to \eqref{e-5.25} and \eqref{e-5.26},
\begin{equation}
\label{e-5.27}
\mu\big([x|(1-\delta)N]\big)\leq m^{2\delta N}\mu([x|N])\leq \alpha_1(T_{x|N})^{-\tau/2}\cdot \alpha_1(T_{x|N})^\tau=\alpha_1(T_{x|N})^{\tau/2}.
\end{equation}
Meanwhile by \eqref{e-epsilon1}, $\alpha_+^\delta\leq (\alpha_-)^\epsilon$, which together with \eqref{e-geta} yields
\begin{equation}
\label{e-5.28}
\alpha_+^{\delta N}\leq (\alpha_-)^{N\epsilon}\leq \alpha_1(T_{x|N})^\epsilon.
\end{equation}

Let $r_N=\alpha_1(T_{x|N})$. We may require that $N$ is so large that $(1/\alpha_+)^{N\epsilon/2}>N$, which implies that
\begin{equation}
\label{e-5.19''}
N\leq r_N^{-\epsilon/2}.
\end{equation}
Below we estimate $G_\mu(x, r_N)$. Since $\mu(\{x\})=0$ by \eqref{e-lar}, we have
\begin{align*}
G_\mu(x, r_N)&=\sum_{n=0}^\infty Z_{x|n}(r_N)(\mu([x|n])-\mu([x|n+1]))\\
&\leq \sum_{n=0}^{(1-\delta)N} Z_{x|n}(r_N) + \sum_{n=(1-\delta)N}^\infty  (\mu([x|n])-\mu([x|n+1]))\\
&\leq N Z_{x|(1-\delta)N}(r_N) + \mu([x|(1-\delta)N])\\
&\leq N \frac{r_N}{\alpha_1(T_{x|(1-\delta)N})}+\alpha_1(T_{x|N})^{\tau/2} \quad \mbox{ (by \eqref{e-aa2} and \eqref{e-5.27})} \\
&=N\frac{\alpha_1(T_{x|N})}{\alpha_1(T_{x|(1-\delta)N})}+r_N^{\tau/2}\\
&\leq N \alpha_+^{\delta N} +r_N^{\tau/2}\\
&\leq N r_N^\epsilon +r_N^{\tau/2} \quad \mbox{ (by  \eqref{e-5.28})}\\
&\leq 2 r_N^{\epsilon/2}  \quad \mbox{ (using $\epsilon<\tau/2$ and  \eqref{e-5.19''})}.
\end{align*}
This implies \eqref{e-claim1} since $N$ can be taken arbitrarily large. So we have completed the proof of part (ii).

Finally we prove (iii). Here  $T_i$ ($i=1,\ldots, m$) are assumed to be scalar multiples of orthogonal matrices. By (ii) it suffices to show that for $s\in (0,d)$, \eqref{e-equv} $\Longrightarrow$ \eqref{e-equv1}. For this purpose we only need to modify the proof of part (ii) (for Case 1) slightly.  More precisely, we only need to replace `$k<s<k+1$' by `$k\leq s<k+1$' in the first paragraph in the reasoning of Case 1, and remove the explanations `since $k<s\pm\tau,s-\epsilon<k+1$' in the estimations of the sums (II) and (IV). Indeed since $T_i$ are scalar multiples of orthogonal matrices, it follows that
$$
\alpha_1(T_{x|n})=\cdots=\alpha_d(T_{x|n}) \qquad \mbox{ for all }x\in \Sigma \mbox{ and }n\in \N,
$$
hence the equality $\phi^{s-\epsilon}(T_{x|n})/\phi^{s\pm \tau}(T_{x|n})=\alpha_{k+1}(T_{x|n})^{\mp \tau-\epsilon}$ holds unconditionally.
\end{proof}	

In the remainder of this section, we construct an example to show that for every integer $d\geq 2$ and  $s\in \{1,\ldots, d-1\}$,  the conditions \eqref{e-equv} and \eqref{e-equv1} are not equivalent.

\begin{ex}
\label{ex-1}
{\rm
Let $d\geq 2$ and $k\in \{1,\ldots, d-1\}$. Set $m=3^{2k}$. Define
$$
T_1=\cdots=T_m={\rm diag}(\rho_1,\ldots, \rho_d),
$$
where $\rho_1=\cdots=\rho_k=1/3$ and $\rho_{k+1}=\cdots=\rho_d=1/9$. Define
$$M_i=8^i,\qquad i=1,2,\ldots.$$
Write
$$
\mathcal A_1=\{1,\ldots, 3^k\},\quad \mathcal A_2=\{1,\ldots, 3^{2k}\} \;\mbox{ and }\;\mathcal A_3=\{1\}.
$$
Define for $j\geq 1$,
$$
\mathcal B_j=\left\{
\begin{array}{ll}
\mathcal A_2, & \quad  \mbox{ if }\; j\in [M_i+1, (9/8)M_i] \;  \mbox{ for some }i,\\
\mathcal A_3, & \quad\mbox{ if }\;j\in [(9/8)M_i+1, (5/4) M_i]\; \mbox{ for some }i,\\
\mathcal A_1, & \quad\mbox{ otherwise}.
\end{array}
 \right.
$$
Set $\Sigma=\{1,\ldots, m\}^\N$.
 Construct a compact subset $X$ of $\Sigma$ by
$$
X=\prod_{j=1}^\infty \mathcal B_j:=\left\{(x_j)_{j=1}^\infty\in \Sigma:\; x_j\in \mathcal B_j\right\}.
$$
Then we define a product probability measure $\mu$ on $X$ by
$$
\mu=\prod_{j=1}^\infty {\bf p}_j,
$$
where for each $j$, ${\bf p}_j$ is the equal-weighted probability vector in $\R^{\# \mathcal B_j}$, i.e.,  $$
{\bf p}_j=\left\{\frac{1}{\# \mathcal B_j},\ldots, \frac{1}{\# \mathcal B_j} \right\}.
$$
We will show that
\begin{equation}
\label{e-j1}
\liminf_{n\to \infty} S_n(\mu,x)=\limsup_{r\to 0} \frac{\log G_\mu(x,r)}{\log r}=k \quad \mbox{ for all $x\in X$},
\end{equation}
and
\begin{equation}
\label{e-j2}
\limsup_{n\to \infty} S_n(\mu,x)>k\quad \mbox{ for all $x\in X$},
\end{equation}
which indicate that \eqref{e-equv} does not imply \eqref{e-equv1}.
}
\begin{proof}[Justification of \eqref{e-j1} and \eqref{e-j2}]
By the definitions of $\phi^s$ and $T_1,\ldots, T_m$, we have for $x\in X$, $n\in \N$, $s\in [0,d]$ and $0<r<1$,
\begin{equation}
\label{e-phis}
\phi^s(T_{x|n})=\left\{
\begin{array}{ll}
3^{-ns} & \mbox{ if }s\in [0,k],\\
3^{-nk} 9^{-n(s-k)} &\mbox{ if } s\in (k,d],
\end{array}
\right.
\end{equation}
and
\begin{equation}
\label{e-zxn}
Z_{x|n}(r)=\left\{
\begin{array}{ll}
1 &\mbox{ if }n\geq \frac{\log(1/r)}{\log 3},\\
3^{kn} r^k  &\mbox{ if } \frac{\log (1/r)}{\log 9}\leq n\leq \frac{\log(1/r)}{\log 3},\\
3^{kn}9^{(d-k)n}r^d &\mbox{ if } n\leq \frac{\log (1/r)}{\log 9}.
\end{array}
\right.
\end{equation}

Meanwhile by the constructions of $X$ and $\mu$, it is readily checked that for every $x\in X$ and $j\in \N$,
\begin{equation*}
\mu([x|j])=\left\{
\begin{array}{ll}
3^{-M_ik} 3^{-2(j-M_i)k} &  \mbox{ if }\; j\in [M_i+1, (9/8)M_i] \;  \mbox{ for some }i,\\
3^{-(5/4)M_ik} & \mbox{ if }\;j\in [(9/8)M_i+1, (5/4) M_i]\; \mbox{ for some }i,\\
3^{-jk} & \mbox{ otherwise}.
\end{array}
 \right.
\end{equation*}
It follows that
\begin{align}
&\mu([x|n])\leq 3^{-nk} \mbox{ for all }n\in \N, \label{e-T1}\\
&\mu([x|n])= 3^{-nk} \mbox{ for all }n\in \N\Big\backslash \bigcup_{i=1}^\infty\left [M_i+1, \frac{5}{4}M_i\right],\label{e-T2}\\
&\mu([x|n])-\mu([x|n+1])\geq  3^{-nk-1} \mbox{ for all }n\in \N\Big\backslash \bigcup_{i=1}^\infty\left [M_i+1, \frac{5}{4}M_i\right],\label{e-T2'} \\
&\mu([x|n])= 3^{-(5/4)M_i k} \mbox{ if } n=\frac{9}{8}M_i \mbox{ for }i\in \N.\label{e-T3}
\end{align}
By the definition of $S_n(\mu,x)$ (see \eqref{e-2.2}), \eqref{e-phis} and \eqref{e-T1}-\eqref{e-T3},  we obtain for every $x\in X$,
$$
\left\{
\begin{array}{ll}
S_n(\mu,x)\geq k & \mbox{ for all }n\in \N,\\
S_n(\mu,x)=k &\mbox{ for all }n\in \N\Big\backslash \bigcup_{i=1}^\infty\left [M_i+1, \frac{5}{4}M_i\right],\\
S_n(\mu,x)=\frac{19}{18}k &\mbox{ if } n=\frac{9}{8}M_i \mbox{ for }i\in \N,
\end{array}
\right.
$$
which implies that
\begin{equation}
\label{e-5.37}
\liminf_{n\to \infty} S_n(\mu,x)=k <\limsup_{n\to \infty} S_n(\mu,x)\quad \mbox{ for all }x\in X.
\end{equation}

In what follows we show that
\begin{equation*}
\label{e-5.38}
\limsup_{r\to 0}\frac{\log G_\mu(x,r)}{\log r}=k \quad \mbox{ for all }x\in X.
\end{equation*}
 By \eqref{e-5.37} and Lemma \ref{lem-5.5}(ii),  we have $\limsup_{r\to 0}\frac{\log G_\mu(x,r)}{\log r}\geq k$. So it is enough to show $\limsup_{r\to 0}\frac{\log G_\mu(x,r)}{\log r}\leq k$. To this end, it is sufficient to prove that
 \begin{equation}
 \label{e-5.40'}
 G_\mu(x, 3^{-N})\geq  3^{-kN-1} \mbox{ for all $x\in X$ and enough large $N\in \N$}.
 \end{equation}

 Let $x\in X$ and let $N>M_2$ be any given integer.  Let $i$ be the unique integer so that $$M_i<N\leq M_{i+1}.$$
 Then either $ \frac{5}{4}M_i<N\leq M_{i+1}$, or $ \frac{5}{4}M_{i-1}<\lceil N/2\rceil \leq M_{i}$.  So by  \eqref{e-T2'},
 \begin{equation}
 \label{e-5.40}
 \begin{split}
& \mbox{ either }\; \mu([x|N])-\mu([x|N+1])\geq 3^{-kN-1}\\
 & \mbox{ or } \; \mu([x| \lceil N/2\rceil ])-\mu([x| \lceil N/2\rceil +1])\geq 3^{-k \lceil N/2\rceil -1}.
 \end{split}
 \end{equation}
 Taking $r=3^{-N}$ in
 \eqref{e-zxn} yields that
\begin{equation}
\label{e-5.41}
Z_{x|n}(3^{-N})=3^{k(n-N)} \quad \mbox{ for all } N/2\leq n\leq N.
 \end{equation}
 Hence
 \begin{align*}
 G_\mu(x, 3^{-N})&\geq \sum_{n=0}^\infty Z_{x|n}(3^{-N})\Big(\mu\big([x|n]\big)-\mu\big([x|n+1]\big)\Big)\\
 &\geq  Z_{x|\lceil N/2\rceil}(3^{-N})\Big(\mu\big([x| \lceil N/2\rceil ]\big)-\mu\big([x| \lceil N/2\rceil +1]\big)\Big)\\
 &\qquad\quad +Z_{x|N}(3^{-N})\Big(\mu\big([x|N]\big)-\mu\big([x|N+1]\big)\Big)\\
 &= 3^{k(\lceil N/2\rceil-N)}\Big(\mu\big([x| \lceil N/2\rceil ]\big)-\mu\big([x| \lceil N/2\rceil +1]\big)\Big)\\
 &\qquad\quad +\Big(\mu\big([x|N]\big)-\mu\big([x|N+1]\big)\Big)\qquad \mbox{(by \eqref{e-5.41})}\\
 &\geq 3^{-kN-1}\qquad \mbox{(by \eqref{e-5.40})}.
 \end{align*}
 This proves \eqref{e-5.40'}.
\end{proof}
\end{ex}

\section{Hausdorff and packing dimensions of projected sets}
\label{S5}

In this section we investigate the Hausdorff and packing dimensions of projected sets on typical self-affine sets.

\subsection{Hausdorff dimension}

For $E\subset \Sigma$ and $n\in \N$, we call $\mathcal C\subset \Sigma_*$ {\it a cover of $E$ with order $n$} if
$$
\bigcup_{\bi\in \mathcal C}[\bi]\supset E\quad \mbox{ and } \quad |\bi|\geq n\mbox{ for  all } \bi\in \mathcal C.
$$
Following Falconer \cite{Fal88} we define, for each $s \ge 0$, a net measure $\M^s$ of Hausdorff type on $\Sigma$ by
\[
\M^s(E) = \lim_{n\to \infty} \M^s_n(E),\qquad E\subset \Sigma,
\]
where $$\M^s_n(E) = \inf \left\{ \sum_{\mathbf{i}\in \mathcal C} \phi^s(T_{\mathbf{i}}) : \; \mathcal C \mbox{ is a cover of $E$ with order $n$}\right\}.$$
Then $\M^s$ is an outer measure which restricts to a measure on the Borel $\sigma$-algebra of $\Sigma$.  We further define
\begin{equation}
\label{e-dimM}
\dim_{\M} E = \inf \{ s\ge 0: \M^s(E) < \infty \}
 = \sup \{ s \ge 0 : \M^s(E) =  \infty\}.
\end{equation}
It is known \cite{Fal88} that  $\dim_{\M} \Sigma$ is equal to  the affinity dimension $\dim_{\rm AFF}(T_1,\ldots, T_m)$  defined in \eqref{e-aff}.

The main result in this subsection is the following, which slightly generalises the results obtained in \cite{Fal88, Sol98, JPS07, KaVi10, JJKKSS14}.

 \begin{thm} \label{main2}
Let $E \subset \Sigma$ be an analytic  set. Then the following properties hold.
\begin{itemize}
\item[(i)] For every $\ba\in \R^{md}$, $\dim_{\mathrm{H}} \pi^\ba(E) \leq  \min \{ \dim_{\M} E, d\}$.
\item[(ii)] Assume that $\|T_i\|<1/2$ for $1\leq i\leq m$. Then for ${\mathcal L}^{md}$-a.e.~$\ba\in \R^{md}$,
\[
\dim_{\mathrm{H}} \pi^\ba(E) = \min \{ \dim_{\M} E, d\}.
\]
Moreover if $\dim_{\M} E>d$, then ${\mathcal L}^d(\pi^\ba(E))>0$ for ${\mathcal L}^{md}$-a.e.~$\ba\in \R^{md}$.
\end{itemize}
\end{thm}

This result follows from the arguments of \cite{Fal88, JPS07} in a straightforward manner.  For the reader's convenience, below  we provide  some details.

\begin{pro} \label{proJPS}
Assume that $\|T_i\|<1/2$ for $1\leq i\leq m$. Let $\mu$ be a Borel probability measure on $\Sigma$ and $s\geq 0$.  Assume that there is $C>0$  such that
\begin{equation*}
\mu([\bi]) \le C \phi^s(T_\bi) \quad \mbox{ for all }\; \bi \in \Sigma^*,  \label{assumptionInProp3.1}
\end{equation*}
where $\phi^s(\cdot)$ is defined as in \eqref{e-singular}.
Then  for $\Leb^{md}$-a.e.  $\ba \in \R^{md}$,
\begin{enumerate}
\item
If $s \le d$, then $\ldim{H} \pi^\ba_*\mu \ge s$;
\item
If $s >d $, then $ \pi^\ba_*\mu \ll \Leb^d$.
\end{enumerate}
\end{pro}

\begin{proof}
This is a combination of Proposition 2 and Lemma 7 in \cite{JPS07}, whilst part (i) was also implicitly proved in \cite[Theorem 5.3]{Fal88}.
\end{proof}

Recall that for $s\geq 0$ and $A\subset \R^d$, the $s$-dimensional Hausdorff measure $\mathcal H^s(A)$ is defined by $\mathcal H^s(A)=\lim_{\delta\to 0}\mathcal H^s_\delta(A)$, where $$\mathcal H^s_\delta(A)=\inf\left\{\sum_{i=1}^\infty {\rm diam}(A_i)^s:\; \{A_i\} \mbox{ is a $\delta$-cover of $A$}\right\}.$$ See e.g.~\cite{Fal03} for the details.

\begin{proof}[Proof of Theorem \ref{main2}]
%Let $E \subset \Sigma$ be a Borel set.
We first prove part (i) by following the argument  of \cite{Fal88}.   Fix $\ba\in \R^{md}$. We may assume that $\dim_{\M} E<d$,  otherwise there is nothing left to prove.    Let $\dim_{\M} E<s<d$. By definition, $\mathcal M^s(E)=0$.  Let  $\ell = \lfloor s \rfloor$ be the integral part of $s$.

Let $B$ be a closed ball of diameter at least $1$ and large enough such that $ \bigcup_{i=1}^m f_i(B) \subset B$.  Then $\pi^\ba(\Sigma)\subset B$.  Given $\delta>0$, choose $n \in \N$ so large that
\[
\diam (f_{\bi}(B)) < \delta \quad \text{ for all } \bi\in \Sigma_* \mbox{  with } |\bi | \ge n.
\]
Let $\mathcal{C}\subset \Sigma_*$ such that $\bigcup_{\bi\in \mathcal C} [\bi]\supset E$ and  $|\bi| \ge n$ for all $\bi\in \mathcal{C}$.    Then $$\pi^\ba(E) \subset \bigcup_{\bi \in \mathcal{C}} \pi^\ba([\bi]) =\bigcup_{\bi \in \mathcal{C}} f_{\bi}^\ba(\pi^\ba(\Sigma))\subset \bigcup_{\bi \in \mathcal{C}} f_{\bi}^\ba(B)$$ and each ellipsoid $f_{\bi} ^\ba (B)$ is contained in at most
\[
\beta_{\ell}(\bi):=
\left[ 4\diam(B)\frac{\alpha_1(T_\bi)}{\alpha_{\ell+1}(T_\bi)} \right]\cdots \left[ 4\diam(B)\frac{\alpha_{\ell}(T_\bi)}{\alpha_{\ell+1}(T_\bi)} \right](4\diam(B))^{d-\ell}
\]
cubes of side length $\alpha_{\ell+1}(T_\bi)$. Clearly  these cubes are  of diameter $\sqrt{d}\alpha_{\ell+1}(T_\bi) \le \sqrt{d} \delta$.  It follows that
\begin{align*}
\mathcal{H}^s_{\sqrt{d}\delta} (\pi^\ba(E))
& \le \sum_{\bi \in \mathcal{C}} \beta_{\ell}(\bi) \left(\sqrt{d} \alpha_{\ell+1}(T_\bi)\right)^s \\
& \le \sum_{\bi \in \mathcal{C}} (4\diam(B))^d \alpha_1(T_\bi) \cdots \alpha_{\ell}(T_\bi) \alpha_{\ell+1}^{-\ell}(T_\bi)  \left(\sqrt{d} \alpha_{\ell+1}(T_\bi)\right)^s \\
& \le (4\sqrt{d}\diam(B))^d \sum_{\bi \in \mathcal{C}} \phi^s(T_\bi).
\end{align*}
This holds for all $\mathcal{C}\subset \Sigma_*$ such that $\bigcup_{\bi\in \mathcal C} [\bi]\supset E$ and  $|\bi| \ge n$ for  $\bi\in \mathcal{C}$. Hence
\[
\mathcal{H}^s_{\sqrt{d}\delta} (\pi^\ba(E)) \le (4 \sqrt{d} \diam(B))^d \M^s_n(E).
\]
Letting  $n \to \infty$ and then $\delta\to 0$ gives
\[
\mathcal{H}^s(\pi^\ba(E)) \le (4 \sqrt{d} \diam(B))^d \M^s(E)=0.
\]
So $\dim_{\mathrm{H}} \pi^\ba(E) \leq s$. Since  $s\in (\dim_{\M} E,\,d)$ is arbitrary, it follows that $\dim_{\mathrm{H}}  \pi^\ba(E)\leq \dim_{\M} E$. This proves (i).

Next we prove (ii).
To avoid triviality we  assume that $\dim_{\M}E>0$. Let $0 < s < \dim_{\M}E$.  Then $\M^s(E) =  \infty$.  By \cite[Theorem 55]{Rogers70} there exists a compact subset $F$ of $E$ such that $0<\M^s(F)<\infty$.  A slight modification of  the proof of \cite[Sect. II, Theorem 1]{Car67} or \cite[Theorem 5.4]{Fal86} shows that there exist $\mu\in \mathcal P(\Sigma)$ and $c>0$ such that $\mu$ is supported on $F$ and
\begin{equation}
\label{e-teq}
\mu ( [\bi]) \le c \phi^s(T_{\bi})\quad  \mbox{ for all }\bi\in \Sigma_*.
\end{equation}
 For the reader's convenience, below we give a sketched proof of this fact by adapting the arguments of \cite[Sect. II, Theorem 1]{Car67}.

Define $f\colon \Sigma_*\to (0,\infty)$ by $f(\bi)=\phi^s(T_{\bi})$, and set
$$\M^s_\infty(F):=\inf\left\{\sum_{n=1}^\infty f(\bi_n):\; \bi_n\in \Sigma_*,\, \bigcup_{n=1}^\infty [\bi_n]\supset F\right\}.$$
Since $0<\M^s(F)<\infty$, it follows that $0<\M^s_\infty(F)<\infty$. For the construction of $\mu$, let $\xi$ denote the unique Bore probability measure on $\Sigma$ such that $\xi([\bi])=m^{-|\bi|}$ for every $\bi\in \Sigma_*$. Fix $n$ and construct a finite Borel measure $\eta_n$ on $\Sigma$ such that $\eta_n([\bi])=f(\bi)$ for all $\bi\in \Sigma_n$ for which $[\bi]\cap F\neq \emptyset$. We let $\eta_n$ have constant density (with respect to $\xi$) on $[\bi]$ for each $\bi\in \Sigma_n$. If for some $\bj\in \Sigma_{n-1}$,
$$\eta_n([\bj])>f(\bj),$$
we reduce the density of $\eta_n$ on the corresponding cylinders $[\bi]$ with $\bi\in \Sigma_n$ and $[\bi]\subset [\bj]$, such that the mass on $[\bj]$ becomes $f(\bj)$. The resulting set function on $\Sigma_{n}$ is called $\eta_n^{(n-1)}$. We treat $\eta_n^{(n-1)}$ in a similar way and after $n$ steps we obtain a set function $\eta_n^{(0)}$. This function has the property that
$$
\eta_n^{(0)}([\bj])\leq f(\bj)\quad  \mbox{ for all }\bj\in \bigcup_{k=0}^n\Sigma_k.
$$
Moreover, for each $x\in F$ there exists $k=k(x)\in \{0,1,\ldots, n\}$ such that
$$
\eta_n^{(0)}([x|k])= f(x|k);
$$
and this implies that $\eta_n^{(0)}(\Sigma)\geq \M^s_\infty(F)$.

Let $n\to \infty$ and choose a weakly convergent subsequence $\eta_{n_j}^{(0)}\to \eta$, as $j\to \infty$.
Then $\eta$ is supported on $F$,  $\eta([\bi])\leq f(\bi)$  for all $\bi\in \Sigma_*$ and $\eta(\Sigma)\geq \M^s_\infty(F)$. Define $\mu=\frac{1}{\eta(\Sigma)}\eta$ and let $c=1/\M^s_\infty(F)$. Clearly $\mu$ is supported on $F$ and \eqref{e-teq} holds.

 If $\dim_{\mathcal M}E>d$, we may require that $s>d$ and then apply Proposition \ref{proJPS} to obtain that for $\mathcal L^{md}$-a.e.~$\ba\in \R^{md}$, $\pi^\ba_*\mu \ll \Leb^d$, implying that   $\mathcal L^d(\pi^\ba(E))>0$ since $\pi^\ba_*\mu$ is supported on $\pi^\ba(E)$.  In what follows we assume that  $\dim_{\mathcal M}E\leq d$.
 by Proposition \ref{proJPS}, $\underline{\dim}_{\rm H}\pi^\ba_*\mu\geq s$ for $\mathcal L^{md}$-a.e.~$\ba\in \R^{md}$. Since $\mu$ is supported on $E$, it follows that ${\dim}_{\rm H}\pi^\ba(E)\geq s$ for $\mathcal L^{md}$-a.e.~$\ba$.  Letting $s\to  \dim_{\M} E$ yields that  ${\dim}_{\rm H}\pi^\ba(E)\geq  \dim_{\M} E$  for $\mathcal L^{md}$-a.e.~$\ba$.   Since  ${\dim}_{\rm H}\pi^\ba(E)\leq   \dim_{\M} E$ for each $\ba$ by part (i), we obtain the equality for $\mathcal L^{md}$-a.e.~$\ba$. This proves part (ii).  \end{proof}

\begin{rem}
\label{rem-5.3} For each analytic subset $E$ of $\Sigma$,
$$\dim_\M E=\sup_{\mu\in \mathcal P(\Sigma):\; {\rm spt}\mu\subset E}\underline{S}(\mu)=\sup_{\mu\in \mathcal P(\Sigma):\; {\rm spt}\mu\subset E}\overline{S}(\mu).$$
Indeed the inequality $\dim_\M E\leq \sup_{\mu\in \mathcal P(\Sigma):\; {\rm spt}\mu\subset E}\underline{S}(\mu)$ was implicitly proved in the proof of Theorem \ref{main2}(ii). To see  $\dim_\M E\geq \sup_{\mu\in \mathcal P(\Sigma):\; {\rm spt}\mu\subset E}\overline{S}(\mu)$, suppose that $\overline{S}(\mu)>s$ for some $\mu\in \mathcal P(\Sigma)$ with ${\rm spt}\mu\subset E$. Then there exist $E'\subset E$ of positive measure and $n_0$ such that  $\mu([x|n])\leq \phi^s(T_{x|n})$ for all $x\in E'$ and $n\geq n_0$, which implies that $\M^s(E')\geq \mu(E')>0$, hence
$\dim_\M E\geq \dim_\M E'\geq s$, as desired.
\end{rem}

%\newpage

\subsection{Packing dimension}

In this subsection we prove the following.
\begin{thm}
 \label{thm-packing-sets}
 Let $E$ be an analytic subset of $\Sigma$. Then
 \begin{itemize}
 \item[(i)] For every  $\ba\in \R^{md}$,
  $$\dim_{\rm P}(\pi^\ba(E))\leq \sup_{\mu\in {\mathcal P}(\Sigma):\; {\rm spt}(\mu)\subset  E}\underline{D}(\mu).$$
   \item[(ii)] Assume that $\|T_i\|<1/2$ for $1\leq i\leq m$. Then for $\mathcal L^{md}$-a.e.~$\ba\in \R^{md}$,
   $$\dim_{\rm P}(\pi^\ba(E))= \sup_{\mu\in {\mathcal P}(\Sigma):\; {\rm spt}(\mu)\subset  E}\underline{D}(\mu)=\sup_{\mu\in {\mathcal P}(\Sigma):\; {\rm spt}(\mu)\subset  E}\overline{D}(\mu).$$
 \end{itemize}
 \end{thm}

Theorem \ref{thm-packing-sets} is based on Theorem \ref{thm-packing} and the  following  well known result (see e.g. \cite[Proposition 2.8]{Fal-technique}).

\begin{lem}
\label{lem-4.10} Let $F$ be an analytic subset of $\R^d$. Then
$$
%\dim_{\rm H} F=\sup_{\mu\in \mathcal P(\R^d): \;{\rm spt}(\mu)\subset F}\underline{\dim}_{\rm H}\mu,
%\qquad
\dim_{\rm P} F=\sup_{\mu\in \mathcal P(\R^d): \;{\rm spt}(\mu)\subset F}\underline{\dim}_{\rm P}\mu.
$$
\end{lem}
\begin{proof}[Proof of Theorem \ref{thm-packing-sets}] Let  $E$ be an analytic subset of $\Sigma$. To prove part (i), let $\ba\in \R^{md}$ and $0\leq t<\dim_{\rm P}(\pi^\ba(E))$. Let $\mathcal N=\N^\N$ be the Baire space (with product topology, $\N$ being discrete). Since $E$ is analytic, there is a continuous surjective map $f:\;\mathcal N \to E$. Define $g=\pi^\ba\circ f$.  Then  $\pi^\ba(E)=g(\mathcal N)$. Choose $s>0$ such that $t<s<\dim_{\rm P}(\pi^\ba(E))$. Since $\pi^\ba(E)=g(\mathcal N)$,  it follows from  Theorem 2 of \cite{Haase1986} that there exists a compact $C\subset \mathcal N$ such that $\dim_{\rm P}(g(C))\geq s$. Write $K=f(C)$. Then $K$ is a compact subset of $E$ with $\dim_{\rm P}(\pi^\ba(K))= \dim_{\rm P}(g(C))\geq  s>t$. By Lemma \ref{lem-4.10}, there exists a Borel probability measure $\nu$ on $\pi^\ba(K)$ such that $\underline{\dim}_{\rm P}\nu\geq t$. Now by Theorem 1.20 of \cite{Mattila1995}, there exists a Borel probability measure $\mu$ on $K$ such that $\nu=\pi^\ba_*\mu$.  Clearly $\mu\in \mathcal P(\Sigma)$ with ${\rm spt}(\mu)\subset E$ and $\underline{\dim}_{\rm P}(\pi^\ba_*\mu)\geq t$.  By Theorem \ref{thm-packing}(i), $\underline{D}(\mu)\geq \underline{\dim}_{\rm P}(\pi^\ba_*\mu)$. Thus $\underline{D}(\mu)\geq t$. This proves part (i) of the theorem.

Next we prove part (ii). It suffices to show that for any $t\geq 0$ with
$$t<\sup_{\mu\in {\mathcal P}(\Sigma): \;{\rm spt}(\mu)\subset E} \overline{D}(\mu),$$
one has $\dim_{\rm P}(\pi^\ba(E))\geq t$ for $\mathcal L^{md}$-a.e.~$\ba\in \R^{md}$.  To this end, let $t$ be such a number. Choose
$\mu\in {\mathcal P}(\Sigma)$ such that ${\rm spt}(\mu)\subset E$ and $\overline{D}(\mu)>t$.  By Theorem \ref{thm-packing}(ii), $\overline{\dim}_{\rm P}(\pi^\ba_*\mu)=\overline{D}(\mu)$ for $\mathcal L^{md}$-a.e.~$\ba\in \R^{md}$.
Since $\pi^\ba_*\mu$ is supported on $\pi^\ba(E)$, it follows that
$\dim_{\rm P}(\pi^\ba(E))\geq \overline{D}(\mu)>t$  for $\mathcal L^{md}$-a.e.~$\ba\in \R^{md}$.
\end{proof}

\section{Box-counting dimensions of projected sets}
\label{S6}

In this section, we investigate the lower and upper box-counting dimensions of projected sets on typical self-affine sets. By adapting arguments  from Falconer \cite{Fal21} (in which he gave a capacity approach to the box-counting dimensions of  orthogonal projections of sets), we establish a projection theorem (see Theorem \ref{thm-box}) for the box-counting dimensions
in the typical self-affine setting.

We first introduce the following two definitions, which are the variants of  the corresponding definitions in \cite{Fal21} in the setting of orthogonal projections.

\begin{de}  For a nonempty compact set $E\subset \Sigma$ and $r>0$, we write
	\begin{equation}
	\label{e-Cr}
	C_r(E)=\left(\inf_{\mu\in \mathcal{P}(E)}\iint Z_{x\land y}(r)\;d\mu(x)d\mu(y)\right)^{-1},
	\end{equation}
and call it the {\it $r$-capacity} of $E$, where $\mathcal P(E)$ is the collection of all Borel probability measures supported on $E$.
\end{de}
	
	\begin{de}
		Let $E$ be  a non-empty compact subset of $\Sigma$.  Define
		\begin{equation}\label{def prof}
		\overline{\dim}_CE= \limsup_{r \to 0}\frac{\log C_r(E)}{-\log r},  \quad 	\underline{\dim}_CE= \liminf_{r \to 0}\frac{\log C_r(E)}{-\log r}.
		\end{equation}
 We call them the lower and upper capacity dimensions of $E$.
			\end{de}
	
The main result of this section is the following.

\begin{thm}
\label{thm-box}
		Let $E$ be  a non-empty subset of $\Sigma$.  Then the following properties hold.
		\begin{itemize}
		\item[(i)] For all $\ba \in \R^{m d}$, $\underline{\dim}_{\rm B} \mathcal\pi^\ba(E)\leq \underline{	\dim}_C\overline{E}$ and $\overline{\dim}_{\rm B} \mathcal\pi^\ba(E)\leq \overline{	\dim}_C\overline{E}$.
		\medskip
		\item[(ii)] Suppose that $\|T_i\|<1/2$ for $1\leq i\leq m$. Then for $\mathcal{L}^{md}$-a.e.~$\ba\in \R^{m d}$,  $\overline{\dim}_{\rm B} \mathcal\pi^\ba(E)= \overline{	\dim}_C\overline{E}$ and $\overline{\dim}_{\rm B} \mathcal\pi^\ba(E)= \overline{	\dim}_C\overline{E}$.
		\end{itemize}
		\end{thm}
	
	To prove the above result, let us recall the definitions of lower and upper box-counting dimensions. For a non-empty bounded set $ F\subset \R^d $, let $N_r(F)$ be the minimum number of sets of diameter $r$ that can cover $F$.   Then the lower and upper  box-counting dimensions of $F$ are defined by
	$$ \underline{\dim}_{\rm B}F= \liminf_{r \to 0}\frac{\log N_r(F)}{-\log r}, \quad  \overline{\dim}_{\rm B}F= \limsup_{r \to 0}\frac{\log N_r(F)}{-\log r}.$$
	
The following result states that for a given non-empty compact set $E\subset \Sigma$ and $\ba\in \R^{md}$,  the covering number $ N_r(\pi^{\textbf{a}}(E)) $ is nearly controlled by the $r$-capacity $  C_r(E) $ when $ r $ is sufficiently small. The formulation  of this result is inspired by  \cite[Corollary 2.4]{Fal21}.

	\begin{pro}
	\label{key Prop}
	Let $\ba\in \R^{md}$ and  $E$  a non-empty compact subset of $\Sigma$. Then
		\begin{equation*}
		N_r(\pi^\ba(E))\leq \left(\frac{\log r}{\log \alpha_+}+2\right)C'\cdot C_r(E) \quad \mbox{ for  all }  \;  0<r<\alpha_+,
		\end{equation*}
where $ \alpha_+$ is defined as in \eqref{e-eta}, and $C'$ is the constant given in Lemma \ref{lem-ae}.
	\end{pro}

	The proof of Proposition \ref{key Prop} is based on the following energy-minimising property, which is standard in potential theory (see e.g. \cite[Theorem 2.4]{Fuglede1960} or \cite[Lemma 2.1]{Fal21}), and only uses the fact that the mapping $(x,y)\mapsto Z_{x\wedge y}(r)$ is positive, symmetric and continuous on $\Sigma\times \Sigma$ for each $r>0$.
	
		\begin{lem}\label{lem-potential}
		Let $E\subset \Sigma$ be a non-empty compact set  and $r>0$. Then the infimum in \eqref{e-Cr} is attained by a measure $\mu_0 \in \mathcal{P}(E)$. Moreover
		\begin{equation*}\label{e-inq}
		\int Z_{x\land y}(r)\;d\mu_0(y)\ge  \frac{1}{C_r(E)} \quad \mbox{ for all }\; x\in E,
		\end{equation*}
		 with equality for $\mu_0$-a.e.~$x\in E$.
	\end{lem}
	
	For $ \rho >0$, let  $\mathcal{D_{\rho}}$ be the partition of $\R^d$ defined as in \eqref{e-drho}.

\begin{proof}[Proof of  Proposition \ref{key Prop}]
Let $\ba\in \R^{md}$ and  $ E\subset \Sigma $ a non-empty compact set. 	
 Let  $ 0<r< \alpha_+$. To estimate $N_r(\pi^\ba(E))$, we write
 $$
 \mathcal{D}_{r/\sqrt{d}}(E)=\left\{Q\in \mathcal{D}_{r/\sqrt{d}}:\; Q\cap \pi^{\mathbf{a}}(E)\neq \emptyset\right\}.
 $$
 Clearly $\mathcal{D}_{r/\sqrt{d}}(E)$ is a cover of $\pi^\ba(E)$. Since each cube in $\mathcal{D}_{r/\sqrt{d}}(E)$ is of diameter $r$, it follows that $N_r(\pi^\ba(E))\leq \#\mathcal{D}_{r/\sqrt{d}}(E)$.

 Now for each $ Q \in  \mathcal{D}_{r/\sqrt{d}}(E)$, we pick a point  $z(Q)\in Q\cap \pi^{\mathbf{a}}(E) $. Set
	\begin{equation*}\label{def A}
	A=\left \{z(Q):\;  Q \in  \mathcal{D}_{r/\sqrt{d}}(E)\right\}.
	\end{equation*}
 Clearly $A\subset \pi^{\mathbf{a}}(E)$ and $\#A=\#\mathcal{D}_{r/\sqrt{d}}(E)$. So
 \begin{equation}
 \label{e-Ain}
 \#A\geq N_r(\pi^\ba(E)).
 \end{equation}
 In what follows we will give an upper bound of $\#A$.

 Let $\ell$ be the smallest integer so that $\alpha_+^\ell\leq r$. That is,
  $ \ell=\left\lceil\frac{\log r}{\log \alpha_+}\right\rceil$.
Let $\mu_0\in \mathcal P(E)$ such that the infimum in \eqref{e-Cr} is attained at $\mu_0$. The existence of $\mu_0$ follows from Lemma \ref{lem-potential}. Moreover by Lemma \ref{lem-potential},
\begin{equation}\label{e-inq'}
		\int Z_{x\land y}(r)\;d\mu_0(y)\ge  \frac{1}{C_r(E)} \quad \mbox{ for all }\; x\in E.
		\end{equation}

Now we construct a finite sequence  $ \{\nu_n\}_{n=0 } ^{\ell} $ of probability measures on $A$ as follows:
		\begin{equation}\label{def mea}
	\nu_n=\sum_{I\in \Sigma_{n}}\;\;\sum_{Q \in \mathcal{D}_{r/\sqrt{d}}(E):\; Q\cap \pi^{\mathbf{a}}([I]\cap E)\neq \emptyset}\frac{\mu_0([I])}{N_I}\cdot \delta_{z(Q)}, \quad n=0,\ldots, \ell,
	\end{equation}
	where $$N_I:=\#\left\{Q\in \mathcal{D}_{r/\sqrt{d}}(E):\; Q\cap\pi^{\mathbf{a}}([I]\cap E)\neq \emptyset \right\}$$ and
	 $ \delta_{y} $ stands for  the Dirac measure at the point $y$. Here we take the convention that $\frac{0}{0}=0$.
Clearly, for each $n$ the total mass of $\nu_n$ is equal to
	 $$
	 \sum_{I\in \Sigma_{n}}\;\;\sum_{Q \in \mathcal{D}_{r/\sqrt{d}}(E):\; Q\cap \pi^{\mathbf{a}}([I]\cap E)\neq \emptyset}\frac{\mu_0([I])}{N_I}=\sum_{I\in \Sigma_{n}}\mu_0([I])=1.
	 $$
So $ \nu_n\in   \mathcal{P}(A)$ for $n=0,\ldots, \ell$.

Set $ \nu= \nu_0+\cdots +\nu_\ell $. Then  $\nu(A)=\ell+1$.   The following inequality is a key point in our proof:
			\begin{equation}
\label{e-kin}
		\nu(\{z\})\geq \frac{1}{C'\cdot C_r(E)} \quad \mbox { for all }  z\in A,
		\end{equation}
where $C'>0$ is the constant given in Lemma \ref{lem-ae}. 	
To prove \eqref{e-kin}, let  $Q\in \mathcal D_{r/\sqrt{d}}(E)$.  By \eqref{def mea}, for every $0\leq n\leq \ell$,
		$$  \nu_n(\{z(Q)\})=\sum_{I\in \Sigma_{n}}\frac{\mu_0([I])}{\#\left\{\widetilde{Q}\in \mathcal{D}_{r/\sqrt{d}}(E):\; \widetilde{Q}\cap\pi^{\mathbf{a}}([I]\cap E)\neq \emptyset \right\}}. $$
		Since $z(Q)\in Q\cap \pi^\ba(E)$, there exists $x \in E $ such that  $  \pi^{\mathbf{a}}(x) =z(Q)$.  It follows  that for every $0\leq n\leq \ell$,
		\begin{align*}
			\begin{split}
				\nu_n(\{z(Q)\}) &\geq  \frac{\mu_0([x |n])}{\#\left\{\widetilde{Q}\in \mathcal{D}_{r\sqrt{d}}(E): \widetilde{Q}\cap\pi^{\mathbf{a}}([x |n]\cap E)\neq \emptyset \right\}}\\
				&\geq \frac{1}{C'}\mu_0([x |n])Z_{x |n}(r)\qquad\mbox{(by Lemma \ref{lem-ae})}.
			\end{split}
		\end{align*}
Summing over $n$ yields that
\begin{align*}
\nu(\{z(Q)\})
&\geq \frac{1}{C'}\sum_{n=0}^\ell\mu_0([x |n])Z_{x |n}(r)\\
&\geq \frac{1}{C'} G_{\mu_0}(x, r)  \qquad\mbox{(by Lemma \ref{lem-gmu})}\\
&=\frac{1}{C'}\int Z_{x\wedge y}(r)\; d\mu_0(y)\\
&\geq \frac{1}{C'\cdot C_r(E)}  \qquad\mbox{(by \eqref{e-inq'})}.
\end{align*}
This proves \eqref{e-kin}.

Since $\nu(A)=\ell+1$, by  \eqref{e-kin} we obtain that
$$\#A\leq (\ell+1)C' \cdot C_r(E)=\left(\left\lceil\frac{\log r}{\log \alpha_+}\right\rceil+1\right)C'\cdot C_r(E)\leq \left(\frac{\log r}{\log \alpha_+}+2\right)C'\cdot C_r(E).
$$
Combining this with \eqref{e-Ain} completes the proof of the proposition.
\end{proof}

Now we are ready to prove Theorem \ref{thm-box}.

\begin{proof}[Proof of Theorem \ref{thm-box}(i)]
 Let $\ba\in \R^{md}$ and $E$ a non-empty subset of $\Sigma$. Observe that
\begin{equation}
\label{e-eqB}
\overline{\dim}_B\pi^\ba(E)=\overline{\dim}_B\pi^\ba(\overline{E}), \quad \underline{\dim}_B\pi^\ba(E)=\underline{\dim}_B\pi^\ba(\overline{E}).
\end{equation}
By Proposition \ref{key Prop}, for all $0<r<\alpha_+$,
		$$
  N_r(\pi^{\ba}(\overline{E}))\leq \left(\frac{\log r}{\log \alpha_+}+2\right)C'\cdot C_r(\overline{E}),$$
  		so
		\[  \frac{\log N_r(\pi^{\textbf{a}}(\overline{E}))}{-\log r} \leq \frac{\log C_r(\overline{E})}{-\log r}+ \frac{\log (C'((\log r/ \log \alpha_+)+2))}{-\log r}.\]
		Taking lower and upper limits as $ r \to 0 $ yields the desired upper bounds for the lower and upper box-counting dimensions of $\pi^\ba(\overline{E})$.
\end{proof}

To prove Theorem \ref{thm-box}(ii), we still need the following.	
	\begin{lem}[{\cite[Lemma 2.2]{Fal21}}]
\label{from F}
		Let $ F\subseteq \R^d$ be non-empty and compact and let $ r > 0 $. Suppose that there
		is a measure $ \nu \in \mathcal{P}(F)  $ such that for some $\beta> 0$
		$$(\nu \times \nu)\{(x, y):|x-y| \leq r\} \leq \beta.$$
		Then
		\begin{equation*}
		N_{r}(F) \geq \frac{c_{d}}{ \beta},
		\end{equation*}
		where $ c_{d}$ depends only on $ d $.
	\end{lem}
	
	\begin{proof}[Proof of Theorem \ref{thm-box}(ii)]
	Our arguments are mainly adapted from the proof of Theorem 1.1 in \cite{Fal21}.
		
		Due to \eqref{e-eqB} we may assume that $E\subset \Sigma$ is non-empty and  compact. Let $\rho>0$. It is enough to show that $$
\overline{\dim}_{\rm B}\mathcal\pi^\ba(E)\geq \overline{	\dim}_CE,\quad  \underline{\dim}_{\rm B}\mathcal\pi^\ba(E)\geq \underline{	\dim}_CE
$$
for $\mathcal L^{md}$-a.e.~$\ba\in B_\rho$.

 By Lemma \ref{lem-tran}, there exists $C=C(\rho, m, d)>0$ such that
		\begin{equation}\label{Leb Z}
		\mathcal{L}^{md}\left\{\ba \in B_{\rho}:\;\left|\pi^{\ba}\bi-\pi^{\ba}\bj \right | \leq r\right\} \leq C\cdot Z_{\bi \wedge \bj}(r)
		\end{equation}
for all $\bi, \bj\in \Sigma$ and $r>0$.
		
Let $ \mu \in\mathcal{ P}(E) $. Using
		Fubini’s theorem and \eqref{Leb Z},
				\begin{align}
		\int_{B_\rho}  \left( \mu\times \mu\right) & \{(\bi,\bj):\; \left|\pi^{\ba}\bi-\pi^{\ba}\bj \right | \leq r\}\;d \ba \nonumber\\
		={}&\iint \mathcal{L}^{md}\{\ba\in B_{\rho}:\; \left|\pi^{\ba}\bi-\pi^{\ba}\bj \right | \leq r\}\;d\mu(x)d\mu(y) \nonumber\\
		\leq{}& \iint C\cdot Z_{\bi\wedge \bj}(r)\;d\mu(\bi)d\mu(\bj). \label{apply}
		\end{align}
			If $ 	\overline{\dim}_CE>t'>t>0$, then there exists a non-increasing sequence $  \{ r_k\}_{k=1}^{\infty} $ with $ r_{k} \to 0 $ and $ 0<r_k<2^{-k} $, such that $C_{r_k}(E)\geq r_k^{-t'}$. Thus by Lemma \ref{lem-potential}, for each $k$ there exists $\mu_k \in \mathcal{P}(E)$ such that
		\[ \iint Z_{\bi\wedge \bj}(r_k)\;d\mu_{k}(\bi)d\mu_{k}(\bj)=\frac{1}{C_{r_k}(E)}\leq r_{k}^{t'}. \]\
		
Applying \eqref{apply} to each $\mu_k$ and summing over $k$,
		\begin{align*}
			\int_{B_\rho} & \left(\sum_{k=1}^{\infty} r_{k}^{-t}  (\pi^\ba_*\mu_k\times \pi^\ba_*\mu_k) \{(x,y): \left|x-y \right | \leq r_k\}\right)\;d \ba \\
			&\mbox{}\qquad= \sum_{k=1}^{\infty} r_{k}^{-t}\int_{B_\rho} \left( \mu_k\times \mu_k\right) \{(\bi,\bj): \left|\pi^{\ba}\bi-\pi^{\ba}\bj \right | \leq r_k\}\;d \ba\\
			&\mbox{}\qquad\leq C\sum_{k=1}^{\infty} r_{k}^{-t} \iint Z_{\bi\wedge \bj}(r_k)\;d\mu_k(\bi)d\mu_k(\bj)   \\
			&\mbox{}\qquad \leq  C\sum_{k=1}^{\infty}r_{k}^{t'-t}\leq  C\sum_{k=1}^{\infty}2^{-k(t'-t)}<\infty.
		\end{align*}
		Hence for $ \mathcal{L}^{m d}$-a.e.~$\ba\in B_\rho$ there is   $ M_{\ba} < \infty$  such that
		$$ (\pi^\ba_*\mu_k\times \pi^\ba_*\mu_k)\{(x,y): \left|x-y \right| \leq r_k\} \leq M_{\ba} r_{k}^{t},  \quad \text{for  } k\in \N.$$
	 For such $ \textbf{a} ,$  applying  Lemma \ref{from F}  to the set $ \pi^\ba(E) $ yields that 	
		$$ N_{r_k}(\pi^{\ba}(E))\geq c_d (M_{\textbf{a}})^{-1} r_{k}^{-t},  \quad \text{for   } k\in \N.$$
 It follows that   $$\overline{\dim}_{\rm B}(\pi^\ba(E))=\limsup_{r \to 0} \frac{\log N_r(\pi^\ba (E))} {-\log r}\geq t.$$ This holds for all $ t < 	\overline{\dim}_CE $, thus $ \overline{\dim}_{\rm B}\mathcal\pi^\ba(E) \geq \overline{\dim}_CE $ for $ \mathcal{L}^{md}$-a.e.~$ \ba \in B_\rho$.
		
		The inequality for the lower box-counting dimension for almost all $\ba\in B_\rho$ follows in a similar manner, 		noting that it is enough to take $ r = 2^{-k} $, $ k \in  N $, when considering the limits as $  r \to 0 $ in the definitions of lower box dimension and lower  capacity dimension.
		Thus we have proved Theorem \ref{thm-box}(ii).
				\end {proof}
\section{The proofs of Theorems \ref{thm-lower-upper} and \ref{thm-main2}}

\begin{proof}[Proof of Theorem \ref{thm-lower-upper}] It follows directly by combining Theorems  \ref{thm-lower} and  \ref{thm-packing}.
\end{proof}

\begin{proof}[Proof of Theorem \ref{thm-main2}]
This follows directly by combining Theorems \ref{main2}, \ref{thm-packing-sets} and \ref{thm-box}.
\end{proof}

\section{An analogous result for orthogonal projections}
\label{S8}

Let $n,m$ be two positive integers with $n> m$. We denote the Grassmann manifold of all $m$-dimensional linear subspaces of $\R^n$ by $G_{n,m}$ and the natural invariant probability measure on $G_{n,m}$ by $\gamma_{n,m}$. For $V\in G_{n,m}$, let $P_V:\R^n\to V$ denote the orthogonal projection onto $V$. For $\mu\in \mathcal P(\R^n)$, let $\mu_V$ denote the projection of $\mu$ under $P_V$, that is, $\mu_V=\mu\circ P_V^{-1}$.

In this section, we prove the following result, which is an analogue of Theorem \ref{thm-5.2} for the orthogonal projections.

\begin{thm}
\label{thm-8.1}
Let $\mu$ be a Borel probability measure on $\R^n$ with compact support. Then $\mu_V$ is exact dimensional for $\gamma_{n,m}$-a.e.~$V\in G_{n,m}$ if and only if one of the following two conditions fulfils:
\begin{itemize}
\item[(i)]  $\underline{\dim}_{\rm H} \mu\geq m$;
\item[(ii)] $\mu$ is exact dimensional with dimension smaller than $m$.
\end{itemize}
\end{thm}

The proof of Theorem \ref{thm-8.1} is based on the following known result, in which \eqref{e-8.1} is due to Hunt and Kaloshin \cite[Theorem 4.1]{HK97} and \eqref{e-8.2} is due to Falconer and O'Neil \cite[Proposition 3.2]{FaON99}.

\begin{thm}[\cite{HK97, FaON99}]
\label{thm-8.2}
Let $\mu$ be a Borel probability measure on $\R^n$ with compact support. Then for $\mu$-a.e.~$x\in \R^n$ and $\gamma_{n,m}$-a.e.~$V\in G_{n,m}$,
\begin{equation}
\label{e-8.1}
\underline{\dim}_{\rm loc}(\mu_V, P_Vx)=\min\{m, \underline{\dim}_{\rm loc}(\mu, x)\}
\end{equation}
and
\begin{equation}
\label{e-8.2}
\overline{\dim}_{\rm loc}(\mu_V, P_Vx)=\limsup_{r\to 0} \frac{\log F_\mu^m(x,r)}{\log r},
\end{equation}
where
\begin{equation}
\label{e-FM}
F_\mu^m(x,r):=\int_{\R^n}\min\left\{1,\;\frac{r^m}{|y-x|^m}\right\}\; d\mu(y).
\end{equation}
\end{thm}

Indeed, according to Theorem \ref{thm-8.2}, $\mu_V$ is exact dimensional for $\gamma_{n,m}$-a.e.~$V\in G_{n,m}$ if and only if either $\underline{\dim}_{\rm H}\mu\geq m$, or there exists $s\in [0,m)$ such that
\begin{equation}
\label{e-8equi}
\underline{\dim}_{\rm loc}(\mu, x)=\limsup_{r\to 0} \frac{\log F_\mu^m(x,r)}{\log r}=s \quad \mbox{ for $\mu$-a.e.~$x\in \R^n$}.
\end{equation}
Hence to prove Theorem \ref{thm-8.1}, it remains to show that for each $s\in [0,m)$, \eqref{e-8equi} is equivalent to the following.
\begin{equation}
\label{e-8equi'}
\underline{\dim}_{\rm loc}(\mu, x)=\overline{\dim}_{\rm loc}(\mu, x)=s \quad \mbox{ for $\mu$-a.e.~$x\in \R^n$}.
\end{equation}

To prove the equivalence of \eqref{e-8equi} and \eqref{e-8equi'} for $s\in [0,m)$, we use an approach similar to that of the proof of Theorem \ref{thm-5.2}.

Let us first present some required auxiliary  results. For $x\in \R^n$ and $r>0$, below we use $B(x,r)$ instead of $B_r(x)$ to denote the closed ball of radius $r$  centered at $x$.

\begin{lem}
[{\cite[Corollary 2.3]{FaMa96}}]
\label{lem-FaMa}
Let $0<a<1$ and $\epsilon>0$. For every Borel probability measure on $\R^n$ the following holds at $\mu$-a.e.~$x$: if $\rho>0$ is sufficiently small, then
$$
\mu(B(x,r))\leq \left(\frac{4r}{\rho}\right)^{n(1+\epsilon)}\mu(B(x,\rho))
$$
for all $r$ with $\rho^a\leq r\leq 1$.
\end{lem}

\begin{cor}
\label{cor-FM}
Let $\mu\in \mathcal P(\R^n)$ and $\delta\in (0,1)$. Then
$$
\limsup_{r\to 0} \frac{\log \left(\mu(B(x, r^{1-\delta}))/\mu(B(x,r))\right) }{\log (1/r)}\leq \delta n\quad \mbox{ for $\mu$-a.e.~$x\in \R^n$}.
$$
\end{cor}
\begin{proof}
Let $\epsilon>0$. Applying Lemma \ref{lem-FaMa} to $a=1-\delta$ and $r=\rho^{1-\delta}$ yields that for $\mu$-a.e.~$x\in \R^n$ and for sufficiently small $\rho>0$,
$$
\frac{\mu(B(x, \rho^{1-\delta}))}{\mu(B(x,\rho))}\leq (4\rho^{-\delta})^{n(1+\epsilon)},
$$
which implies
$$
\limsup_{\rho\to 0}  \frac{\log \left(\mu(B(x, \rho^{1-\delta}))/\mu(B(x,\rho))\right) }{\log (1/\rho)}\leq \delta n (1+\epsilon).
$$
Letting $\epsilon\to 0$ gives the desired inequality.
\end{proof}

\begin{proof}[Proof of Theorem \ref{thm-8.1}] According to   the remark after Theorem \ref{thm-8.2},  it remains to  prove the equivalence of \eqref{e-8equi} and \eqref{e-8equi'} for $s\in [0,m)$.

By \eqref{e-FM},
$$F_\mu^m(x,r)=\int_{\R^n}\min\left\{1,\;\frac{r^m}{|y-x|^m}\right\}\; d\mu(y)\geq \mu(B(x,r)),$$
it follows that
$$
\limsup_{r\to 0} \frac{\log F_\mu^m(x,r)}{\log r}\leq \overline{\dim}_{\rm loc}(\mu, x)
$$
for each $x\in \R^d$, leading to the direct implication \eqref{e-8equi'}$ \Longrightarrow$ \eqref{e-8equi} for $s\in [0,m)$.  In what follows we prove the reverse direction \eqref{e-8equi}$ \Longrightarrow$ \eqref{e-8equi'}.

Fix $s\in [0, m)$. Suppose that \eqref{e-8equi} holds. We need to show that $\overline{\dim}_{\rm loc}(\mu, x)\leq s$ for $\mu$-a.e.~$x\in \R^n$. To this end, we use contradiction. Suppose on the contrary that there exist a Borel set $A\subset \R^n$ with $\mu(A)>0$ and a positive number $\tau$ such that
\begin{equation}
\label{e-8.6}
\overline{\dim}_{\rm loc}(\mu, x)> s+\tau \quad \mbox{  for all $x\in A$}.
\end{equation}
By \eqref{e-8equi} and Corollary \ref{cor-FM}, removing a subset with $\mu$ measure zero from $A$ if necessary, we may assume that for all $x\in A$,
\begin{equation}
\label{e-8.7}
\underline{\dim}_{\rm loc}(\mu, x)=\limsup_{r\to 0} \frac{\log F_\mu^m(x,r)}{\log r}=s
\end{equation}
and
 \begin{equation}
 \label{e-8.8}
 \limsup_{r\to 0} \frac{\log \left(\mu(B(x, r^{1-\delta}))/\mu(B(x,r))\right) }{\log (1/r)}\leq \delta n\quad \mbox{ for all $\delta\in \Bbb Q\cap (0,1)$}.
\end{equation}

To derive a contradiction, we fix $x\in A$ and  $\delta\in \Bbb Q\cap (0,1)$ such that
\begin{equation}
\label{e-8delta}
0<\delta<\frac{\tau}{4n}.
\end{equation}
 Next we continue our argument by considering the cases $s\in (0,m)$ and $s=0$ separately.

 {\bf Case 1}: $s\in (0,m)$.

 In this case we take a small $\epsilon\in (0,s)$ so that
 \begin{equation}
 \label{e-8.10}
 \delta m+(1-\delta)(s-\epsilon)>s.
 \end{equation}
By \eqref{e-8.7}, there exists $r_0>0$ such that
\begin{equation}
\label{e-8se}
\mu(B(x,r))\leq r^{s-\epsilon}\quad \mbox{ for all } 0<r\leq r_0.
\end{equation}

By \eqref{e-8.6} and \eqref{e-8.8}, there exists an arbitrarily small $r_1\in (0,r_0)$ such that $$\mu(B(x, r_1))\leq r_1^{s+\tau}$$ and
\begin{equation}
\label{e-8.11}
\mu(B(x, r_1^{1-\delta}))\leq r_1^{-2\delta n}\mu(B(x, r_1))\leq r_1^{s+\tau-2\delta n}\leq r_1^{s+\tau/2},
\end{equation}
where we used \eqref{e-8delta} in the last inequality.

Observe that
\begin{align}
F_\mu^m(x,r_1)&=\int_{\R^n} \min\left\{1,\;\frac{r_1^m}{|y-x|^m}\right\}\; d\mu(y)\nonumber\\
&\leq \int_{B\left(x, r_1^{1-\delta}\right)}1\; d\mu(y)+\int_{\{y:\; |y-x|>r_0\}}\frac{r_1^m}{|y-x|^m}\;d\mu(y)\nonumber\\
&\qquad+\int_{\left\{y:\; r_1^{1-\delta}<|y-x|<r_0\right\}}\frac{r_1^m}{|y-x|^m}\;d\mu(y)\nonumber\\
&\leq r_1^{s+\tau/2}+(r_1/r_0)^m+  r_1^m T \qquad \mbox{ (by \eqref{e-8.11})},  \label{e-Fmu}
\end{align}
where $$T:=\int_{\left\{y:\; r_1^{1-\delta}<|y-x|<r_0\right\}}\frac{1}{|y-x|^m}\;d\mu(y).$$
To estimate $T$, write $E:=\left\{y:\; r_1^{1-\delta}<|y-x|<r_0\right\}$. Using \cite[Theorem 1.15]{Mattila1995} by a change of variable,
\begin{equation}
\label{e-8.14}
\begin{split}
T&=\int \frac{{\bf 1}_E(y)}{|y-x|^m}\; d\mu(y)=\int_{0}^\infty \mu(\{y\in E:\; |x-y|^{-m}\geq u\})\; du\\
&=\int_{0}^\infty \mu(E\cap B(x, u^{-1/m})) \; du=m\int_{0}^\infty \mu(E\cap B(x, u)) u^{-m-1} \; du\\
&=m\int_{r_1^{1-\delta}}^\infty \mu(E\cap B(x, u)) u^{-m-1} \; du,
\end{split}
\end{equation}
where in the last equality we use the fact that $E\cap B(x, u)=\emptyset$ if $u<r_1^{1-\delta}$. Hence
\begin{align*}
T&=m\int_{r_1^{1-\delta}}^\infty \mu(E\cap B(x, u)) u^{-m-1} \; du\\
&\leq \int_{r_1^{1-\delta}}^{r_0} m\mu( B(x, u)) u^{-m-1} \; du+\int_{r_0}^\infty m u^{-m-1}\; du\\
&\leq \int_{r_1^{1-\delta}}^{r_0} m u^{s-\epsilon-m-1} \; du+\int_{r_0}^\infty mu^{-m-1}\; du\qquad \mbox{(by \eqref{e-8se})}\\
&\leq \frac{m}{m+\epsilon-s}r_1^{-m+\delta m+(1-\delta)(s-\epsilon)}+r_0^{-m}.
\end{align*}
Combining this with \eqref{e-Fmu} yields
$$F_\mu^m(x,r_1)\leq C \left( r_1^{s+\tau/2}+r_1^m+ r_1^{\delta m+(1-\delta)(s-\epsilon)}\right),
$$
where $C$ only depends $m, s$ and $r_0$.  Since $r_1$ can be taken arbitrarily small, by \eqref{e-8.10} it follows that
$$\limsup_{r\to 0} \frac{\log F_\mu^m(x,r)}{\log r}\geq \min\{s+\tau/2, \;m, \;\delta m+(1-\delta)(s-\epsilon)\}>s,
$$
which contradicts \eqref{e-8.7}.

{\bf Case 2}. $s=0$.

In this case, we simply take $r_0=1$ and don't need to introduce the quantity $\epsilon$. We follow the identical argument starting from the second paragraph in Case 1 until the end of \eqref{e-8.14}. Then by \eqref{e-8.14},
\begin{align*}
T&=m\int_{r_1^{1-\delta}}^\infty \mu(E\cap B(x, u)) u^{-m-1} \; du
\leq m\int_{r_1^{1-\delta}}^\infty  u^{-m-1} \; du
=r_1^{-(1-\delta)m}.
\end{align*}
Combining this with \eqref{e-Fmu} (in which $s=0$) yields that
$$F_\mu^m(x,r_1)\leq   r_1^{\tau/2}+r_1^m+ r_1^{\delta m}.
$$
Since $r_1$ can be taken arbitrarily small, it follows that
$$\limsup_{r\to 0} \frac{\log F_\mu^m(x,r)}{\log r}\geq \min\{\tau/2, \;\delta m\}>0=s,
$$
which contradicts \eqref{e-8.7} again.
\end{proof}

\section{Hausdorff dimensions of exceptional sets}
\label{S9}
For each $\ba=(a_1,\ldots, a_m)\in \R^{md}$, let $\pi^\ba:\Sigma\to \R^d$ denote the coding map associated with the affine IFS
$\{T_ix+a_i\}$ in $\R^d$; see \eqref{e-pia}.  Recall that we have defined the quantities $S(\mu,x)$, $\overline{S}(\mu)$, $\underline{S}(\mu)$, $D(\mu,x)$, $\underline{D}(\mu)$ and $\overline{D}(\mu)$ for $\mu\in \mathcal P(\Sigma)$ and $x\in \Sigma$ in \eqref{e-Smux}, \eqref{e-Gammamu},  \eqref{e-D1} and \eqref{e-D1'}, and the quantities $\dim_{\mathcal M}E$, $\underline{\dim}_CE$, $\overline{\dim}_CE$ in  \eqref{e-dimM} and  \eqref{def prof}.

The main results of this section are the following, where Theorem \ref{thm-9.2} is an improvement of Theorem \ref{thm-9.1}(vii) in the case when $\dim_{\mathcal M}E<d$.

\begin{thm}
\label{thm-9.1}
Let $\mu\in \mathcal P(\Sigma)$ and let $E$ be an analytic subset of $\Sigma$. Suppose that $\|T_i\|<1/2$ for all $1\leq i\leq m$. Then for each $0<\delta<d$, the following properties hold:
\begin{itemize}
\item[(i)]
$\dim_{\rm H} \left\{\ba\in \R^{md}:\; \mu\left(\{x\in \Sigma:\; \underline{\dim}_{\rm loc}(\pi^\ba_*\mu, \pi^\ba x)<\min\{ d, S(\mu,x)\}-\delta\}\right)>0\right\}\leq dm-\delta.$
\item[(ii)]
$\dim_{\rm H} \left\{\ba\in \R^{md}:\; \mu\left(\{x\in \Sigma:\; \overline{\dim}_{\rm loc}(\pi^\ba_*\mu, \pi^\ba x)<D(\mu,x)-\delta\}\right)>0\right\}\leq dm-\delta.$
\item[(iii)]
$\dim_{\rm H} \left\{\ba\in \R^{md}:\; \underline{\dim}_{\rm H} \pi^\ba_*\mu<\min\{ d, \underline{S}(\mu)\}-\delta\right\}\leq dm-\delta.$
\item[(iv)]
$\dim_{\rm H} \left\{\ba\in \R^{md}:\; \overline{\dim}_{\rm H} \pi^\ba_*\mu<\min\{ d, \overline{S}(\mu)\}-\delta\right\}\leq dm-\delta.$
\item[(v)]
$\dim_{\rm H} \left\{\ba\in \R^{md}:\; \underline{\dim}_{\rm P} \pi^\ba_*\mu< \underline{D}(\mu)-\delta\right\}\leq dm-\delta.$
\item[(vi)]
$\dim_{\rm H} \left\{\ba\in \R^{md}:\; \overline{\dim}_{\rm P} \pi^\ba_*\mu<\overline{D}(\mu)-\delta\right\}\leq dm-\delta.$
\item[(vii)]
$\dim_{\rm H} \left\{\ba\in \R^{md}:\; \dim_{\rm H} \pi^\ba(E)<\min\{ d, \dim_{\mathcal M}(E)\}-\delta\right\}\leq dm-\delta.$
\item[(viii)]
$\dim_{\rm H} \left\{\ba\in \R^{md}:\; \dim_{\rm P} \pi^\ba(E)<\sup_{\eta\in {\mathcal P}(\Sigma):\; {\rm spt}(\eta)\subset  E}\underline{D}(\eta)-\delta\right\}\leq dm-\delta.$
\item[(ix)]
$\dim_{\rm H} \left\{\ba\in \R^{md}:\; \underline{\dim}_{\rm B} \pi^\ba(E)<\underline{\dim}_CE-\delta\right\}\leq dm-\delta.$
\item[(x)]
$\dim_{\rm H} \left\{\ba\in \R^{md}:\; \overline{\dim}_{\rm B} \pi^\ba(E)<\overline{\dim}_CE-\delta\right\}\leq dm-\delta.$

\end{itemize}
\end{thm}

\begin{thm}
\label{thm-9.2}
 Suppose that $\|T_i\|<1/2$ for all $1\leq i\leq m$. Let $E$ be an analytic subset of $\Sigma$ with  $\dim_{\mathcal M}E<d$. Set $\tau=\min\{\log \alpha_1(T_i)/\log \alpha_d(T_i):\; i=1,\ldots, m\}$. Then for each $0<\delta<d$,
\begin{equation}
\label{e-9.1'}
\begin{split}
\dim_{\rm H}& \left\{\ba\in \R^{md}:\; \dim_{\rm H} \pi^\ba(E)<\min\{ d, \dim_{\mathcal M}(E)\}-\delta\right\}\\
&\leq \max\left\{dm-\frac{\delta}{1-\tau},\; dm+\dim_{\mathcal M}E-d-\delta\right\}.
\end{split}
\end{equation}
\end{thm}

The proofs of the above theorems are based on the following two lemmas, which are the generalizations of Lemmas \ref{lem-2.5} and \ref{lem-tran} respectively.

\begin{lem}[{\cite[Lemma 4.3 and (4.9)]{Fami08}}]
Let $\nu$ be a Borel probability measure with support in $B(0,\rho)\subset \R^{md}$ such that $\nu(B(\ba,r))\leq c_0 r^q$ for all $\ba\in \R^{md}$ and  $r>0$, where $(m-1)d<q\leq md$.   Let $0<s<d$ be a number such that $md-(d-s)<q\leq md$ with $q-s$ non-integral. Assume that $\|T_i\|<1/2$ for all $1\leq i\leq m$ and $\rho>0$. Then there is a number $c$ such that for all distinct $x,y\in \Sigma$ and $r>0$,
\begin{equation}
\label{e-9.1}
\int_{B(0,\rho)} \frac{d\nu(\ba)}{|\pi^\ba(x)-\pi^\ba(y)|^s}\leq \frac{c \alpha_1(T_{x\wedge y})^{md-q}}{\phi^{s+md-q}(T_{x\wedge y})}\leq \frac{c}{\phi^{s+md-q}(T_{x\wedge y})}.
\end{equation}
\end{lem}

\begin{lem}
Let $\nu$ be a Borel probability measure with support in $B(0,\rho)\subset \R^{md}$ such that $\nu(B(\ba,r))\leq c_0 r^q$ for all $\ba\in \R^{md}$ and $r>0$,  $(m-1)d<q\leq md$.   Assume that $\|T_i\|<1/2$ for all $1\leq i\leq m$ and $\rho>0$. Then there is a number $c$ such that
\begin{equation}
\label{e-9.2}
\nu\{\ba\in B(0,\rho):\; |\pi^\ba x-\pi^\ba y|\leq r\}\leq c Z_{x \wedge y}(r) \max\left\{1,\left(\frac{\alpha_1(T_{x\wedge y})}{r}\right)^{md-q}\right\}
\end{equation}
for all distinct $x,y\in \Sigma$ and $r>0$. Consequently,
\begin{equation}
\label{e-9.2'}
\nu\{\ba\in B(0,\rho):\; |\pi^\ba x-\pi^\ba y|\leq r\}\leq c Z_{x\wedge y}(r) r^{q-md}
\end{equation}
for all distinct $x,y\in \Sigma$ and $0<r\leq 1$.
\end{lem}
\begin{proof}
 We will apply some estimations given in \cite{Fami08}. Let $x, y\in \Sigma$ with $x\neq y$.  Set $n=|x\wedge y|$,  $x'=\sigma^n x$ and $y'=\sigma^n y$.
Let $L_{x,y}:\;\R^{md}\to \R^d$ be the linear mapping defined by $\ba\mapsto \pi^\ba x'-\pi^\ba y'$.
Clearly
\begin{equation}
\label{e-lxy}
|L_{x,y}(\ba)|\leq |\pi^\ba x'|+|\pi^\ba y'|\leq \sum_{n=0}^\infty (\|T_{x'|n}\|+\|T_{y'|n}\|)|\ba|< \sum_{n=0}^\infty 2^{1-n}|\ba|= 4|\ba|.
\end{equation}
Let $\nu_0$ denote the push-forward of $\nu$ by  $L_{x,y}$, i.e., $\nu_0=\nu\circ L_{x,y}^{-1}$.  It was proved by Falconer and Miao (see \cite[Lemmas~4.1 and 4.2]{Fami08}) that there is a number $c_1$ (which is independent of $x,y$) such that
\begin{equation}
\label{e-9.4}
\nu_0(B_{\R^d}(z,r))\leq c_1 r^{(1-m)d+q}\quad \mbox{ for all }z\in \R^d \mbox{ and }r>0,
\end{equation}
where $B_{\R^d}(\cdot,\cdot)$ stands for a closed ball in $\R^d$.
Notice that
\begin{equation}
\label{e-9.5}
\begin{split}
&\nu\{\ba\in B(0,\rho):\; |\pi^\ba x-\pi^\ba y|\leq r\}\\
&\mbox{}\quad =\nu\{\ba\in B(0,\rho):\; T_{x\wedge y}(\pi^\ba x'-\pi^\ba y')\in  B_{\R^d}(0, r)\}\\
&\mbox{}\quad =\nu\{\ba\in B(0,\rho):\; \pi^\ba x'-\pi^\ba y'\in T_{x\wedge y}^{-1}B_{\R^d}(0, r)\}\\
&\mbox{}\quad =\nu_0\left(T_{x\wedge y}^{-1}B_{\R^d}(0, r)\right).
\end{split}
\end{equation}
Since $\nu$ is supported in $B(0,\rho)$,  it follows from  \eqref{e-lxy} that $\nu_0$ is supported in $B_{\R^d}(0, 4\rho)$. Hence
\begin{equation}
\label{e-9.6}
\nu_0\left(T_{x\wedge y}^{-1}B_{\R^d}(0, r)\right)=\nu_0\left(B_{\R^d}(0, 4\rho)\cap (T_{x\wedge y}^{-1}B_{\R^d}(0, r))\right).
\end{equation}
It is easy to see that the set $B_{\R^d}(0, 4\rho)\cap (T_{x\wedge y}^{-1}B_{\R^d}(0, r))$ is contained in a rectangle with side lengths $\ell_1\leq \cdots\leq \ell_d$, where
$$
\ell_i=\min\left\{8\rho,\; \frac{2r}{\alpha_i(T_{x\wedge y})}\right\}\leq c_2\min\left\{1, \; \frac{r}{\alpha_i(T_{x\wedge y})}\right\},\qquad i=1,\ldots, d,
$$
with $c_2:=\max\{8\rho, 2\}$. Hence if $r\leq \alpha_1(T_{x\wedge y})$,
then $B_{\R^d}(0, 4\rho)\cap \left(T_{x\wedge y}^{-1}B_{\R^d}(0, r)\right)$ can be covered by
\begin{align*}
2^d \prod_{i=1}^d\frac{\ell_i}{r/\alpha_1(T_{x\wedge y})}&\leq (2c_2)^d \alpha_1(T_{x\wedge y})^d r^{-d} \prod_{i=1}^d\min\left\{1,\; \frac{r}{\alpha_i(T_{x\wedge y})}\right\}\\
& = (2c_2)^d \alpha_1(T_{x\wedge y})^d r^{-d} Z_{x\wedge y}(r) \qquad \mbox{(by \eqref{ef-4.2})}
\end{align*}
 balls of  radius $2\sqrt{d} r/\alpha_1(T_{x\wedge y})$; so by \eqref{e-9.4},
 \begin{align*}
 \nu_0&\left(B_{\R^d}(0, 4\rho)\cap \left(T_{x\wedge y}^{-1}B_{\R^d}(0, r)\right)\right)\\
  &\leq (2c_2)^d \alpha_1(T_{x\wedge y})^d r^{-d} Z_{x\wedge y}(r) c_1 \left(\frac{2\sqrt{d} r}{\alpha_1(T_{x\wedge y})}\right)^{(1-m)d+q}\\
 &\leq c_3 Z_{x\wedge y}(r)\cdot \left(\frac{\alpha_1(T_{x\wedge y})}{r}\right)^{md-q},
 \end{align*}
 where $c_3:=c_1(2c_2)^d(2\sqrt{d})^d$. If $r>\alpha_1(T_{x\wedge y})$, by \eqref{ef-4.2} we have
 $Z_{x\wedge y}(r)=1$, so
 $$
 \nu_0\left(B_{\R^d}(0, 4\rho)\cap (T_{x\wedge y}^{-1}B_{\R^d}(0, r))\right)\leq 1= Z_{x\wedge y}(r).
 $$
Hence for all $r>0$,
$$\nu_0\left(B_{\R^d}(0, 4\rho)\cap (T_{x\wedge y}^{-1}B_{\R^d}(0, r))\right)\leq \max\{1,c_3\} Z_{x \wedge y}(r) \max\left\{1,\left(\frac{\alpha_1(T_{x\wedge y})}{r}\right)^{md-q}\right\}.
$$
Combining this with \eqref{e-9.5},  \eqref{e-9.6}, and setting $c=\max\{1, c_3\}$, we obtain \eqref{e-9.2}. The inequality \eqref{e-9.2'} follows directly from \eqref{e-9.2} since
$$\max\left\{1,\left(\frac{\alpha_1(T_{x\wedge y})}{r}\right)^{md-q}\right\}\leq r^{q-md}.
$$
whenever $0<r\leq 1$.
\end{proof}
Now we are ready to prove Theorem \ref{thm-9.1}.
\begin{proof}[Proof of Theorem \ref{thm-9.1}]
Notice that the mapping $(\ba, x,r)\mapsto \pi^\ba_*\mu(B(\pi^\ba x,r))$ is upper semi-continuous on $\R^{md}\times \Sigma\times (0,\infty)$. It follows that the mappings $(\ba,x)\mapsto \underline{\dim}_{\rm loc}(\pi^\ba_*\mu, \pi^\ba x)$, $(\ba,x)\mapsto \overline{\dim}_{\rm loc}(\pi^\ba_*\mu, \pi^\ba x)$ are Borel measurable. Hence by standard arguments, all the functions $\underline{\dim}_{\rm H} \pi^\ba_*\mu$, $\overline{\dim}_{\rm H} \pi^\ba_*\mu$, $\underline{\dim}_{\rm P} \pi^\ba_*\mu$, $\overline{\dim}_{\rm P} \pi^\ba_*\mu$ are Borel measurable in $\ba$.

We first prove (i). Let $F_1(\delta)$ denote the set of $\ba\in \R^{md}$ so that
$$\mu\{x\in \Sigma:\; \underline{\dim}_{\rm loc}(\pi^\ba_*\mu, \pi^\ba x)<\min\{d, S(\mu,x)\}-\delta\}>0. $$
Since $\ba\mapsto \mu\{x\in \Sigma:\; \underline{\dim}_{\rm loc}(\pi^\ba_*\mu, \pi^\ba x)<\min\{d, S(\mu,x)\}-\delta\}$ is Borel measurable (which follows from the Borel measurability of the mappings $\ba\mapsto \underline{\dim}_{\rm loc}(\pi^\ba_*\mu, \pi^\ba x)$ and $x\mapsto S(\mu,x)$), $F_1(\delta)$ is a Borel subset of $\R^{md}$. Suppose on the contrary that (i) does not hold, that is, $\dim_{\rm H} F_1(\delta)>dm-\delta$. By Frostman's lemma (see e.g.~\cite[Theorem 8.8]{Mattila1995}), there is a Borel probability measure $\nu$ supported on $F_1(\delta)$ such that $\nu(B(\ba,r))\leq c_0 r^{md-\delta}$ for all $\ba\in \R^{md}$ and $r>0$. We may further assume that $\nu$ is supported on $B_\rho:=B(0,\rho)$ for some $\rho>0$.
Next we show that for a given $s\in (0,d-\delta)$ with $s+\delta$ non-integral, for $\nu$-a.e.~$\ba\in B_{\rho}$,
\begin{equation}
\label{e-9.9F}
\underline{\dim}_{\rm loc}(\pi^\ba_*\mu, \pi^\ba x)\geq s \quad \mbox{ for $\mu$-a.e.~$x$ with } S(\mu,x)>s+\delta.
\end{equation}
The proof runs along similar lines as that of Theorem \ref{thm-lower}(ii). For each positive integer $N$, let $\Lambda_N$ be the set of $x$ for which
 $$
 \int_\Sigma\frac{1}{\phi^{s+\delta}(T_{x\wedge y})} d\mu(y)<N.
 $$
 Notice that by Lemma \ref{lem-Smu}, the set of all $x$ for which $S(\mu,x)>s+\delta$ is contained in the union of $\Lambda_N$ for $N\geq 1$.  Appying Fubini's theorem,
  \begin{align*}
 \int_{B_{\rho}}\int_{\Lambda_N} \int_{\R^d} \frac{d\pi^\ba_*\mu(z)}{|\pi^\ba x-z|^s} d\mu(x) d\nu(\ba)&= \int_{B_{\rho}}\int_{\Lambda_N} \int_{\Sigma} \frac{1}{|\pi^\ba x-\pi^\ba y|^s} d\mu(y)d\mu(x) d\nu(\ba)\\
&=\int_{\Lambda_N}  \int_{\Sigma} \int_{B_{\rho}} \frac{1}{|\pi^\ba x-\pi^\ba y|^s}d\nu(\ba)
 d\mu(y)d\mu(x) \\
 &\leq \int_{\Lambda_N}  \int_{\Sigma}  \frac{c}{\phi^{s+\delta}(T_{x\wedge y})}d\mu(y)d\mu(x)\\
 & \qquad\quad\mbox{(by \eqref{e-9.1} in which we take $q=md-\delta$)}\\
 &\leq cN.
  \end{align*}
 It follows that for $\nu$-a.e.~$\ba\in B_{\rho}$,
 $\displaystyle \int_{\Lambda_N} \int_{\R^d} \frac{d\pi^\ba_*\mu(z)}{|\pi^\ba x-z|^s} d\mu(x)<\infty$ and hence
 $$
 \int_{\R^d} \frac{d\pi^\ba_*\mu(z)}{|\pi^\ba x-z|^s}<\infty \quad\mbox{ for $\mu$-a.e.~$x\in \Lambda_N$}.
 $$
 Taking the union over $N$, we have for $\nu$-a.e.~$\ba\in B_{\rho}$,
 $$
 \int_{\R^d} \frac{d\pi^\ba_*\mu(z)}{|\pi^\ba x-z|^s}<\infty \quad\mbox{ for $\mu$-a.e.~$x$ with $S(\mu,x)>s+\delta$}.
 $$
It follows from  Lemma \ref{lem-SY} that  for $\nu$-a.e.~$\ba\in B_{\rho}$,
 $$
 \underline{\dim}_{\rm loc}(\pi^\ba_*\mu, \pi^\ba x)\geq s \quad\mbox{ for $\mu$-a.e.~$x$ with $S(\mu,x)>s+\delta$}.
 $$
This proves \eqref{e-9.9F}. Thus we have shown that for all  $s\in (0,d-\delta)$ with $s+\delta$ non-integral,
$$
\mu\left(\{x\in \Sigma:\; S(\mu,x)-\delta>s> \underline{\dim}_{\rm loc}(\pi^\ba_*\mu, \pi^\ba x)\}\right)=0
$$
for $\nu$-a.e.~$\ba\in B_{\rho}$.  Take the union over all  rational $s$ in $(0,d-\delta)$ with $s+\delta$ non-integral, we conclude that
for $\nu$-a.e.~$\ba\in B_{\rho}$,
$$\mu\left(\{x\in \Sigma:\; \min\{S(\mu,x),d\}-\delta> \underline{\dim}_{\rm loc}(\pi^\ba_*\mu, \pi^\ba x)\}\right)=0,$$
which contradicts that $\nu$ is supported on $F_1(\delta)$. This proves (i).

Next we prove (ii). Let $F_2(\delta)$ denote the set of $\ba\in \R^{md}$ so that
$$\mu\{x\in \Sigma:\; \overline{\dim}_{\rm loc}(\pi^\ba_*\mu, \pi^\ba x)<D(\mu,x)-\delta\}>0. $$
Then $F_2(\delta)$ is a Borel subset of $\R^{md}$. Suppose on the contrary that $\dim_{\rm H} F_2(\delta)>dm-\delta$. By Frostman's lemma  there is a Borel probability measure $\nu$ supported on $F_2(\delta)$ such that $\nu(B(\ba,r))\leq c_0 r^{md-\delta}$ for all $\ba\in \R^{md}$ and $r>0$. We may further assume that $\nu$ is supported on $B_\rho$ for some $\rho>0$. Next we show that for every $x\in \Sigma$,
\begin{equation}
\label{e-9.9}
  \overline{\dim}_{\rm loc}(\pi^\ba_*\mu, \pi^\ba x)\geq D(\mu,x)-\delta \quad \mbox{  for $\nu$-a.e.~$\ba\in B_\rho$}.
\end{equation}
To this end, we modify the proof of Proposition \ref{pro-1.3} slightly.
Let $x\in \Sigma$ and $0<r\leq 1$. Applying Fubini's theorem,
\begin{eqnarray*}
\int_{B_\rho}\pi^\ba_*\mu\left(B_r (\pi^\ba x)\right) d\nu(\ba)&=&
 \int_{B_\rho}\int_{\Sigma} {\bf 1}_{\{y:\;|\pi^\ba y-\pi^\ba x |\leq r\}}\; d\mu(y)  d\nu(\ba) \\
 & =&
 \int_\Sigma \int_{B_\rho} {\bf 1}_{\{\ba:\;|\pi^\ba y-\pi^\ba x |\leq r\}}\;  d\nu(\ba) d\mu(y)  \\
 &=& \int_\Sigma  \nu \{\ba\in {B_\rho}:\; |\pi^\ba x-\pi^\ba y |\leq r\} \; d\mu(y)\\
& \leq & c\int_\Sigma   Z_{x\wedge y}(r) r^{-\delta}\; d\mu(y)\\
&& \qquad\quad\mbox{(by \eqref{e-9.2'} in which we take $q=md-\delta$)}\\
&=& c r^{-\delta} G_\mu(x,r).
\end{eqnarray*}
  Hence by Fatou's lemma,  for any $t\in \R$,
 \begin{eqnarray}
 \int_{B_\rho} \liminf_{r\to 0} r^{-t} \pi^\ba_*\mu\left(B_r (\pi^\ba x)\right) \;d\nu(\ba)& \leq & \liminf_{r\to 0} \int_{B_\rho} r^{-t} \pi^\ba_*\mu\left(B_r (\pi^\ba x)\right)\; d\nu(\ba) \nonumber
 \\
 &\leq& c \liminf_{r\to 0} r^{-t-\delta} G_\mu(x,r).  \label{e-D2''}
 \end{eqnarray}
  Now we assume that $t<D(\mu,x)-\delta$. By the definition of $D(\mu,x)$ (see \eqref{e-D1}), we get $\liminf_{r\to 0} r^{-t-\delta} G_\mu(x,r)=0$. Combining this with \eqref{e-D2''} yields that
  $$
  \int_{B_\rho} \liminf_{r\to 0} r^{-t} \pi^\ba_*\mu\left(B_r (\pi^\ba x)\right) \;d\nu(\ba)=0,
  $$
 which implies that
  $
  \liminf_{r\to 0} r^{-t} \pi^\ba_*\mu\left(B_r (\pi^\ba x)\right)=0$  for $\nu$-a.e.~$\ba\in {B_\rho}$.
   Hence
  $$
  \overline{\dim}_{\rm loc}(\pi^\ba_*\mu, \pi^\ba x)\geq t
  $$
  for $\nu$-a.e.~$\ba\in {B_\rho}$.  Letting $t\nearrow D(\mu,x)-\delta$ yields \eqref{e-9.9}.  Now let $A$ denote the set of $(\ba,x)\in B_\rho\times \Sigma$ such that $$\overline{\dim}_{\rm loc}(\pi^\ba_*\mu, \pi^\ba x)\geq D(\mu,x)-\delta.$$
Then $A$ is a Borel subset of $B_\rho\times \Sigma$. By \eqref{e-9.9} and Fubini's theorem, $\nu\times \mu(A)=1$. Applying Fubini's theorem again, we have
$$
\mu\left\{x\in \Sigma: \; \overline{\dim}_{\rm loc}(\pi^\ba_*\mu, \pi^\ba x)\geq D(\mu,x)-\delta\right\}=1
$$
for $\nu$-a.e.~$\ba\in B_\rho$, which contradicts that $\nu$ is supported on $F_2(\delta)$. This proves (ii).

Now we claim that (iii) and (iv) follow directly from (i), and (v) and (vi) follow directly from (ii). Here we only prove the direction (i)$\Longrightarrow$(iii), since the other implications can be proved in a similar way. To prove (iii), by (i) it is enough to show that
\begin{equation}
\label{e-9.11}
\begin{split}
\{\ba &:\; \underline{\dim}_{\rm H} \pi^\ba_*\mu<\min\{ d, \underline{S}(\mu)\}-\delta\}\\
&\subset
\left\{\ba:\; \mu\left(\{x:\; \underline{\dim}_{\rm loc}(\pi^\ba_*\mu, \pi^\ba x)<\min\{ d, S(\mu,x)\}-\delta\}\right)>0\right\}.
\end{split}
\end{equation}
To see this, let $\ba\in \R^{md}$  such that $\underline{\dim}_{\rm H} \pi^\ba_*\mu<\min\{ d, \underline{S}(\mu)\}-\delta$. Then there is a number $s$ such that
 $$
 \underline{\dim}_{\rm H} \pi^\ba_*\mu<s<\min\{ d, \underline{S}(\mu)\}-\delta,
 $$
 which implies that
 $$
 \mu\{x: \;\underline{\dim}_{\rm loc}(\pi^\ba_*\mu, \pi^\ba x)<s\}>0\quad  \mbox{ and } \quad
 \mu\{x:\; s<\min\{ d, S(\mu,x)\}-\delta\}=1,
 $$
 hence
 $$
 \mu\{x: \;\underline{\dim}_{\rm loc}(\pi^\ba_*\mu, \pi^\ba x)<\min\{ d, S(\mu,x)\}-\delta\}>0.
 $$
 This proves \eqref{e-9.11}.

 Next we prove (v)$\Longrightarrow$(vii). Let $E$ be a non-empty analytic subset of $\Sigma$. By Remark \ref{rem-5.3}, we can pick a sequence $(\mu_n)\subset \mathcal P(\Sigma)$ with ${\rm spt}(\mu_n)\subset E$ such that $\underline{S}(\mu_n)\geq \dim_{\mathcal M}E-1/n$.  Since $\underline{\dim}_{\rm H}\pi^\ba_*\mu_n\leq \dim_{\rm H} \pi^\ba (E)$, it follows that
 \begin{equation*}
 \begin{split}
 \{\ba:&\; \dim_{\rm H} \pi^\ba(E)<\min\{ d, \dim_{\mathcal M}E\}-\delta\}\\
 &\subset
\left\{\ba:\; \underline{\dim}_{\rm H}\pi^\ba_*\mu_n<\min\{ d, \underline{S}(\mu_n)\}-\delta+\frac{1}{n}\right\}
\end{split}
\end{equation*}
for each $n\in \N$. Combining this with (v) yields (vii).

Using a similar argument, we can show that  (vi)$\Longrightarrow$(viii). We leave the details to the reader.

Finally it remains to prove (ix) and (x).  We only prove (ix) since the proof of (x) is similar.  We may assume that $E$ is compact.  Notice that for every $\epsilon>0$, $N_\epsilon(\pi^\ba(E))$ is upper semicontinuous in $\ba$. It follows that the mapping $\ba\to \overline{\dim}_{\rm B}\pi^\ba(E)$ is Borel measurable.  Let $F_3(\delta)$ denote the set of $\ba\in \R^{md}$ so that
$$\overline{\dim}_{\rm B}(\pi^\ba(E))<\overline{\dim}_CE-\delta. $$
Then $F_3(\delta)$ is a Borel subset of $\R^{md}$. Suppose that (ix) does not hold, i.e. $\dim_H F_3(\delta)>md-\delta$.  Below we derive a contradiction.

By Frostman's lemma  there is a Borel probability measure $\nu$ supported on $F_3(\delta)$ such that $\nu(B(\ba,r))\leq c_0 r^{md-\delta}$ for all $\ba\in \R^{md}$ and $r>0$. We may further assume that $\nu$ is supported on $B_\rho$ for some $\rho>0$.  Next we will show that
\begin{equation}
\label{e-9.12}
\overline{\dim}_{\rm B}(\pi^\ba(E))\geq \overline{\dim}_CE-\delta \quad \mbox{ for $\nu$-a.e.~$\ba\in B_\rho$,}
\end{equation}
which clearly contradicts that $\nu$ is supported on $F_3(\delta)$.  To prove \eqref{e-9.12}, we follow the lines of the proof of Theorem \ref{thm-box}(ii) with minor modifications.  Let $ \mu \in\mathcal{ P}(E) $. By
		Fubini’s theorem,
				\begin{align}
		\int_{B_\rho}  \left( \mu\times \mu\right) & \{(\bi,\bj):\; \left|\pi^{\ba}\bi-\pi^{\ba}\bj \right | \leq r\}\;d \nu(\ba) \nonumber\\
		={}&\iint \nu\{\ba\in B_{\rho}:\; \left|\pi^{\ba}\bi-\pi^{\ba}\bj \right | \leq r\}\;d\mu(\bi)d\mu(\bj) \nonumber\\
		\leq{}& \iint c\cdot Z_{\bi\wedge \bj}(r)r^{-\delta}\;d\mu(\bi)d\mu(\bj), \label{apply'}
		\end{align}
		by using \eqref{e-9.2'} in which we take $q=md-\delta$.
			If $ 	\overline{\dim}_CE>t'>t+\delta$, then there exists a non-increasing sequence $  \{ r_k\}_{k=1}^{\infty} $ with $ r_{k} \to 0 $ and $ 0<r_k<2^{-k} $, such that $C_{r_k}(E)\geq r_k^{-t'}$. Thus by Lemma \ref{lem-potential}, for each $k$ there exists $\mu_k \in \mathcal{P}(E)$ such that
		\[ \iint Z_{\bi\wedge \bj}(r_k)\;d\mu_{k}(\bi)d\mu_{k}(\bj)=\frac{1}{C_{r_k}(E)}\leq r_{k}^{t'}. \]\		
Applying \eqref{apply'} to each $\mu_k$ and summing over $k$,
		\begin{align*}
			\int_{B_\rho} & \left(\sum_{k=1}^{\infty} r_{k}^{-t}  (\pi^\ba_*\mu_k\times \pi^\ba_*\mu_k) \{(x,y): \left|x-y \right | \leq r_k\}\right)\;d \nu(\ba) \\
			&\mbox{}\qquad= \sum_{k=1}^{\infty} r_{k}^{-t}\int_{B_\rho} \left( \mu_k\times \mu_k\right) \{(\bi,\bj): \left|\pi^{\ba}\bi-\pi^{\ba}\bj \right | \leq r_k\}\;d \nu(\ba)\\
			&\mbox{}\qquad\leq c\sum_{k=1}^{\infty} r_{k}^{-t} \iint Z_{\bi\wedge \bj}(r_k)r_k^{-\delta}\;d\mu_k(\bi)d\mu_k(\bj)   \\
			&\mbox{}\qquad \leq  c\sum_{k=1}^{\infty}r_{k}^{t'-t-\delta}\leq  c\sum_{k=1}^{\infty}2^{-k(t'-t-\delta)}<\infty.
		\end{align*}
		Hence for $ \nu$-a.e.~$\ba\in B_\rho$ there is   $ M_{\ba} < \infty$  such that
		$$ (\pi^\ba_*\mu_k\times \pi^\ba_*\mu_k)\{(x,y): \left|x-y \right| \leq r_k\} \leq M_{\ba} r_{k}^{t},  \quad \text{for  } k\in \N.$$
	 For such $\ba$,  applying  Lemma \ref{from F}  to the set $ \pi^\ba(E) $ yields that 	
		$$ N_{r_k}(\pi^{\ba}(E))\geq c_d (M_{\ba})^{-1} r_{k}^{-t},  \quad \text{for   } k\in \N.$$
 It follows that   $$\overline{\dim}_{\rm B}(\pi^\ba(E))=\limsup_{r \to 0} \frac{\log N_r(\pi^\ba (E))} {-\log r}\geq t.$$ This holds for all $ t < 	\overline{\dim}_CE-\delta $, thus \eqref{e-9.12} holds.
\end{proof}

\bigskip
\begin{proof}[Proof of Theorem \ref{thm-9.2}] Here we follow the strategy of the proof of \cite[Theorem 4.9]{Fami08} with suitable modifications.

Write $q=\max\left\{dm-\frac{\delta}{1-\tau},\; dm+\dim_{\mathcal M}E-d-\delta\right\}$ and $$
F_4(\delta)=\left\{\ba\in \R^{md}:\; \dim_{\rm H} \pi^\ba(E) <  \dim_{\mathcal M}E-\delta\right\}. $$
According to a general result of Mattila and Mauldin (see \cite[Theorem 2.1]{MaMa97}), the mapping $\ba\mapsto \dim_{\rm H} \pi^\ba(E)$ is Borel measurable. Hence $F_4(\delta)$ is a Borel subset of $\R^{md}$. Suppose on the contrary that \eqref{e-9.1'} does not hold, i.e. $\dim_{\rm H} F_4(\delta)>q$. Then by Frostman's lemma, there is a Borel probability measure $\nu$ supported on $F_4(\delta)$ such that $\nu(B(\ba,r))\leq c_0 r^{q}$ for all $\ba\in \R^{md}$ and $r>0$. We may further assume that $\nu$ is supported on $B_\rho:=B(0,\rho)$ for some $\rho>0$.

Let $\epsilon>0$. As shown in the proof of Theorem \ref{main2}(ii), there exists $\mu\in \mathcal P(\Sigma)$ such that $\mu$ is supported on $E$ and
\begin{equation}
\label{e-9n1}
\mu([x|n])\leq c_1 \phi^{\dim_{\mathcal M}E-\epsilon}(T_{x|n})
\end{equation}
for all $x\in \Sigma$ and $n\in \N$. Let $s\in [0, \dim_{\mathcal M}E-\delta-\epsilon)$ so that $s+md-q$ is non-integral. Since $q=\max\{dm+\dim_{\mathcal M}E-d-\delta, md-\delta/(1-\tau)\}$, it follows that
\begin{equation}
\label{e-9n2}
s+md-q< d-\epsilon
\end{equation}
and
\begin{equation}
\label{e-9n3}
s+(1-\tau)(md-q)<\dim_{\mathcal M}E-\epsilon.
\end{equation}
By \eqref{e-9n1} and \eqref{e-9n3}, we obtain that
\begin{equation}
\label{e-9n4}
\frac{\mu([x|n])}{\phi^{s+(1-\tau)(md-q)}(T_{x|n})}\leq c_1 \alpha_1(T_{x|n})^\theta\leq c_1 \alpha_+^{n\theta},
\end{equation}
where $\theta:=\dim_{\mathcal M}E-\epsilon-(s+(1-\tau)(md-q))>0$ and $\alpha_+$ is defined as in \eqref{e-eta}.

Next we show that $\underline{\dim}_{\rm H}\pi^\ba_*\mu \geq s$ for $\nu$-a.e.~$\ba\in B_\rho$. By the potential theoretic characterization of the Hausdorff dimension (see e.g. \cite[Theorem 4.13]{Fal03}), it is enough to show that for $\nu$-a.e.~$\ba\in B_\rho$, $\pi^\ba_*\mu$ has finite $s$-energy:
$$
I_s(\pi^\ba_*\mu):=\iint \frac{d\pi^\ba_*\mu(z)d\pi^\ba_*\mu(w)}{|z-w|^s}<\infty.
$$
Integrating over $B_\rho$ with respect to $\nu$ and using Fubini’s theorem,
\begin{align*}
\int_{B_\rho} I_s(\pi^\ba_*\mu)d\nu(\ba)&=\int_{B_\rho}\iint \frac{d\pi^\ba_*\mu(z)d\pi^\ba_*\mu(w)}{|z-w|^s}d\nu(\ba)\\
&=\int_{B_\rho}\iint \frac{d\mu(x)d\mu(y)}{|\pi^\ba x-\pi^\ba y|^s}d\nu(\ba)\\
&=\iint \int_{B_\rho}\frac{d\nu(\ba)}{|\pi^\ba x-\pi^\ba y|^s}d\mu(x)d\mu(y)\\
&\leq \iint \frac{c\alpha_1(T_{x\wedge y})^{md-q}}{\phi^{s+md-q}(T_{x\wedge y})}d\mu(x)d\mu(y) \quad \mbox{ (by \eqref{e-9n2} and \eqref{e-9.1})}.
\end{align*}
Notice that
$$
\frac{\alpha_1(T_{x\wedge y})^{md-q}}{\phi^{s+md-q}(T_{x\wedge y})}\leq \frac{\alpha_d(T_{x\wedge y})^{\tau(md-q)}}{\phi^{s+md-q}(T_{x\wedge y})} \leq \frac{1}{\phi^{s+(1-\tau)(md-q)}(T_{x\wedge y})}.
$$
It follows that
\begin{align*}
\int_{B_\rho} I_s(\pi^\ba_*\mu)d\nu(\ba)&\leq  \iint \frac{c}{\phi^{s+(1-\tau)(md-q)}(T_{x\wedge y})}d\mu(x)d\mu(y) \\
&\leq c\int \sum_{n=0}^\infty \left(\phi^{s+(1-\tau)(md-q)}(T_{y|n})\right)^{-1}\mu([y|n]) d\mu(y)\\
&\leq cc_1\int \sum_{n=0}^\infty \alpha_+^{n\theta} d\mu(y) \quad \mbox{ (by \eqref{e-9n4})}\\
&<\infty.
\end{align*}
Hence $\underline{\dim}_{\rm H}\pi^\ba\mu\geq s$ for $\nu$-a.e.~$\ba\in B_\rho$. Since $\mu$ is supported on $E$, it follows that $\dim_{\rm H}\pi^\ba(E) \geq s$ for $\nu$-a.e.~$\ba\in B_\rho$. As $s$ is arbitrarily taken in $[0, \dim_{\mathcal M}E-\delta-\epsilon)$ with $s+md-q$  non-integral, letting $\epsilon\to 0$ we obtain that
$$\dim_{\rm H}\pi^\ba(E) \geq \dim_{\mathcal M}E-\delta$$ for $\nu$-a.e.~$\ba\in B_\rho$. This contradicts the fact that $\nu$ is supported on $F_4(\delta)$.
\end{proof}

\begin{rem} In the special case when $E=\Sigma$,  \eqref{e-9.1'} slightly improves the estimate (4.23) in \cite[Theorem 4.9]{Fami08}.
\end{rem}

\section{Final remarks and questions}
\label{S10}

In the section we give a few remarks and questions.

In our main theorems, the assumption that $\|T_i\|<1/2$ for $1\leq i\leq m$ can be weaken to $\max_{i\neq j} (\|T_i\|+\|T_j\|)<1$. Indeed the first assumption is only used to guarantee the self-affine transversality condition (see Lemma \ref{lem-tran} and Lemma \ref{lem-2.5}). As pointed in \cite[Proposition 9.4.1]{BSS2021},  the second assumption is sufficient for the self-affine transversality condition.

Under one of the above norm assumptions,  for a given Borel set $E\subset \Sigma$ and a Borel probability measure $\mu$ on $\Sigma$,  although we have provided the formal formulas for the various dimensions of the projections of $E$ and $\mu$ under $\pi^\ba$ for almost all $\ba$, however it looks difficult to calculate or estimate these dimensions directly. For instance, it seems hard to find an easily checked necessary and sufficient condition on $E$ such that $\dim_{\rm H} \pi^\ba(E)=\overline{\dim}_{\rm B}\pi^\ba(E)$ for almost all $\ba$, although we have the following theoretic criterion by Theorem \ref{thm-main2}(ii): either $\dim_{\mathcal M}E>d$ or $\dim_{\mathcal M}E=\overline{\dim}_C\overline{E}$.    Meanwhile since these formal formulas depend on the matrices $T_1,\ldots, T_m$,  it arise a natural question when they continuously depend on these matrices.  As a partial result,  the affinity dimension $\dim_{\mathcal M}\Sigma$ continuously depends on these matrices \cite{FengShmerkin2014, Morris2016}. It would be interesting to further investigate the above continuity problem in the more general case.

As mentioned in Section \ref{S1}, our dimensional results on projected sets and measures on typical self-affine sets are analogous to the theorems on dimensions of orthogonal projections and images under factional Brownian motions. The proofs of these results use similar transversality type ideas and potential theoretic approaches, although the technical details are quite different. We remark that these constancy results also extend to dimensions of a general projection scheme introduced by Peres and Schlag \cite{PeresSchlag2000} under certain regularity and transversality assumptions. Moreover, projection theorems can be established for dimensions of projected sets and measures on certain translational families of self-conformal sets under mild assumptions.  The details will be given subsequently.

Recently  Burrel, Falconer and Fraser \cite{BFF2021} proved the constancy result for the intermediate dimensions under orthogonal projections. By adapting and extending the arguments in \cite{BFF2021} and the present paper, Zhou Feng obtained an analogue of this result in the typical self-affine setting \cite{ZhouFeng2022}.
%To be more precise, consider a family of mappings:
%$$
%\pi_\lambda:\; \Omega\to \R^m,\quad \lambda\in Q,
%$$
%where $(\Omega,d)$ is a compact metric space, $Q$ is an open connected subset of $\R^n$. Suppose that the mapping $\lambda\to \pi_\lambda(x)$ is in $C^\infty(Q)$ for every fixed $x\in \Omega$ and
 %

\bigskip

{\noindent \bf Acknowledgements}.
 The authors would like to thank Thomas Jordan for his permission to include Lemma \ref{LemJ} and its proof. They are grateful to Junjie Miao for  early discussions, and  to Julien Barral, Changhao Chen, Kenneth Falconer and Zhou Feng for helpful comments. This research was partially supported by the General Research Fund grants (projects CUHK14301017, CUHK14303021) from the
Hong Kong Research Grant Council, and by a direct grant for research from the Chinese University
of Hong Kong.

\end{document}